\documentclass[11pt, reqno]{amsart}
\usepackage{a4wide}
\usepackage[english, activeacute]{babel}
\usepackage{pifont}
\usepackage{amsmath,amsthm,amsxtra}

\usepackage{epsfig}
\usepackage{amssymb}

\usepackage{latexsym}
\usepackage{amsfonts}
\usepackage{dsfont}
\usepackage{hyperref}
\usepackage{xparse}
\usepackage{tabularx} 
\usepackage{multicol}
\usepackage{mathrsfs}
\usepackage{xcolor}
\usepackage{color}
\usepackage{hyperref}
\usepackage{pgf,tikz}
\usepackage{listings}

\newcommand{\R}{\mathbb{R}}
\newcommand{\N}{\mathbb{N}}

\newcommand{\C}{\mathbb{C}}
\newcommand{\E}{\mathbb{E}}
\newcommand{\PP}{\mathbb{P}}

\newcommand{\dist}{\operatorname{dist}}

\newcommand{\diam}{\operatorname{diam}}

\newtheorem{thm}{Theorem}[section]
\newtheorem{teo}{Theorem}[section]
\newtheorem{cor}[thm]{Corollary}
\newtheorem{lem}[thm]{Lemma}
\newtheorem{lema}[thm]{Lemma}
\newtheorem{prop}[thm]{Proposition}

\newtheorem{ass}{Assumptions}

\newtheorem{defn}[thm]{Definition}

\theoremstyle{remark}
\newtheorem{rem}{Remark}[section]

\newcommand{\be}{\begin{equation}}
\newcommand{\ee}{\end{equation}}
\newcommand{\bp}{\begin{proof}}
\newcommand{\ep}{\end{proof}}
\newcommand{\bel}{\begin{equation}\label}
\newcommand{\eeq}{\end{equation}}
\newcommand{\bea}{\begin{eqnarray}}
\newcommand{\eea}{\end{eqnarray}}
\newcommand{\bee}{\begin{eqnarray*}}
\newcommand{\eee}{\end{eqnarray*}}
\newcommand{\ben}{\begin{enumerate}}
\newcommand{\een}{\end{enumerate}}

\newcommand{\vertiii}[1]{{\left\vert\kern-0.25ex\left\vert\kern-0.25ex\left\vert #1 
    \right\vert\kern-0.25ex\right\vert\kern-0.25ex\right\vert}}

\usepackage{hyperref}

\providecommand{\norm}[1]{\left\| #1 \right\|}

\date{}

\title{A new approach for the fractional Laplacian via deep neural networks}

\author{Nicol\'as Valenzuela}
\address{Departamento de Ingenier\'ia Matem\'atica DIM, and CMM UMI 2807-CNRS, Universidad de Chile, Beauchef 851 Torre Norte Piso 5, Santiago Chile}
\email{nvalenzuela@dim.uchile.cl}
\thanks{N.V. is partially supported by Fondecyt no. 1191412 and CMM Projects ``Apoyo a Centros de Excelencia'' ACE210010 and Fondo Basal FB210005}

\date{\today}
\subjclass[2000]{Primary: 35R11, Secondary: 62M45, 68T07}
\keywords{Deep Neural Networks, Fractional Laplacian, Approximation}

\numberwithin{equation}{section}

\begin{document}

\maketitle \markboth{Nicol\'as Valenzuela}{A new approach for the fractional Laplacian via deep neural networks}

\begin{abstract}

\noindent
The fractional Laplacian has been strongly studied during past decades, see e.g. \cite{CS}. In this paper we present a different approach for the associated Dirichlet problem, using recent deep learning techniques. In fact, intensively PDEs with a stochastic representation have been understood via neural networks, overcoming the so-called \emph{curse of dimensionality}. Among these equations one can find parabolic ones in $\mathbb{R}^d$, see \cite{Hutz}, and elliptic in a bounded domain $D \subset \mathbb{R}^d$, see \cite{Grohs}. 

\medskip

In this paper we consider the Dirichlet problem for the fractional Laplacian with exponent $\alpha \in (1,2)$. We show that its solution, represented in a stochastic fashion by \cite{AK1}, can be approximated using deep neural networks. We also check that this approximation does not suffer from the curse of dimensionality. The spirit of our proof follows the ideas in \cite{Grohs}, with important variations due to the nonlocal nature of the fractional Laplacian; in particular, the stochastic representation is given by an $\alpha$-stable isotropic L\'evy process, and not given by a standard Brownian motion.
\end{abstract}

\tableofcontents

\section{Introduction}\label{Sect:1}

\subsection{Motivation} Deep neural networks (DNNs) have become recent key actors in the study of partial differential equations (PDEs) \cite{WE1,WE2,Hutz}. Among them, deep learning based algorithms have provided new insights for the approximation of certain PDEs, see e.g. \cite{Beck1,Beck2,Grohs2,WE1,WE2}. The corresponding numerical simulations suggest that DNNs overcome the so-called \emph{curse of dimensionality}, in the sense that the number of real parameters that describe the DNN is bounded by a polynomial on the dimension $d$, and on the reciprocal of the accuracy of the approximation. Even better, recent works have theoretically proved that certain PDEs can be approximated by DNNs, overcoming the curse of dimensionality, see e.g. \cite{Berner1,Gonnon,Grohs,Grohs3,Hutz,Jentz1}.

\medskip

In order to describe the previous results in more detail, we start by considering the classical Dirichlet boundary value Problem in $d$-dimensions over a bounded, convex domain $D \subset \R^d$:
\[
\left\{\begin{aligned}
	-\Delta u(x) &= f(x) \quad x \in D,\\
	u(x) &= g(x) \quad x \in \partial D,
\end{aligned}\right.
\] 
where $f,g$ are suitable continuous functions. In a recent work, Grohs and Herrmann \cite{Grohs} proved that DNNs overcome the curse of dimensionality in the approximation of solution of the above problem. More precisely, they used stochastic techniques such as the Feynman-Kac formula, the so-called \emph{Walk-on-Spheres} (WoS) processes (defined in Section \ref{Sect:4}) and Monte Carlo simulations in order to show that DNNs approximate the exact solution, with arbitrary precision. 

\medskip

The main purpose of this paper is to extend the nice results obtained by Grohs and Herrman in the case of the fractional Laplacian $(-\Delta)^{\alpha/2}$, with $\alpha \in (0,2)$, formally defined in $\R^d$ as
\begin{equation}\label{FractLap}
	-(-\Delta)^{\alpha/2} u(x) = c_{d,\alpha} \lim_{\varepsilon \downarrow 0} \int_{\R^d \setminus B(0,\varepsilon)} \frac{u(y)-u(x)}{|y-x|^{d+\alpha}}dy, \hspace{.5cm} x \in \R^d,
\end{equation}
where $c_{d,\alpha}=- \frac{2^\alpha \Gamma((d+\alpha)/2)}{\pi^{d/2}\Gamma(-\alpha/2)}$ and $\Gamma(\cdot)$ is the classical Gamma function. We also prove that DNNs overcome the curse of dimensionality in this general setting, a hard problem, specially because of the nonlocal character of the problem. However, some recent findings are key to fully describe the problem here. Indeed, Kyprianou et al. \cite{AK1} showed that the Feynman-Kac formula and the WoS processes are also valid in the nonlocal case. We will deeply rely on these results to reproduce the Grohs and Herrmann program.

\subsection{Setting} Let $\alpha \in (0,2)$, $d \in \N$ and $D \subset \R^d$ a bounded domain. Consider the following Dirichlet boundary value problem
\begin{equation}\label{eq:1.1}
	\left\{ \begin{array}{rll}
		(-\Delta)^{\alpha/2}u(x)=  f(x) &  x \in D,\\
		u(x)= g(x) & x \in D^c.
	\end{array} \right.
\end{equation}
Here, $f,g$ are functions that satisfy suitable assumptions. More precisely, we ask for the following: %Let additionally $f$ and $g$ be Lipschitz continuous functions.
\begin{itemize}
	\item $g:D^c \to \R$ is a $L_g$-Lipschitz continuous function in $L^{1}_{\alpha}(D^c)$, $L_g >0$, that is to say
	\begin{equation}\label{Hg0}
		\int_{D^c} \frac{|g(x)|}{1+|x|^{d+\alpha}}dx < \infty. \tag{Hg-0}
	\end{equation}
	\item $f:D \to \R$ is a $L_f$-Lipschitz continuous function, $L_f >0$, such that
	\begin{equation}\label{Hf0}
		f \in C^{\alpha + \varepsilon_0}(\overline{D}) \qquad \hbox{for some fixed} \qquad \varepsilon_0>0. \tag{Hf-0}
	\end{equation}
\end{itemize}
These assumptions are standard in the literature (see e.g. \cite{AK1}), and are required to give a rigorous sense to the {\bf continuous solution} in $L^1_\alpha(\R^d)$ of \eqref{eq:1.1} in terms of the stochastic representation
\begin{equation}\label{u(x)}
	u(x) = \E_x \left[ g(X_{\sigma_D}) \right] + \E_x \left[ \int_0^{\sigma_D} f(X_s) ds \right],
\end{equation}
where $(X_t)_{t\geq 0}$ is an $\alpha$-stable isotropic L\'evy process and $\sigma_D$ is the exit time of $D$ for this process. See Theorem \ref{teo:sol} below for full details.

\medskip

Problem \eqref{eq:1.1} has attracted considerable interest in past decades. Starting from the foundational work by Caffarelli and Silvestre \cite{CS}, the study of fractional problems has always required a great amount of detail and very technical mathematics. The reader can consult the monographs by \cite{Acosta,Bonito,Gulian,AK1,Lischke1}. The work of Kyprianou et al. \cite{AK1} proved that the solution of Problem \eqref{eq:1.1} can be represented with the WoS processes described formally in Section \ref{Sect:4}, namely
\begin{equation}\label{u(x)2}
	u(x) = \E_{x} \left[ g(\rho_N) \right] + \E_{x}\left[\sum_{n=1}^{N} r_n^{\alpha} V_1(0,f(\rho_{n-1} +r_n\cdot))\right],
\end{equation}
where $\left(\rho_{n}\right)_{n=0}^{N}$ is the WoS process starting at $\rho_0 = x \in D$, $N=\min \{n\in \N: \rho_n \notin D\}$ and $r_n = \dist(\rho_{n-1},\partial D)$. $V_1(0,1(\cdot))$ is defined in Section \ref{Sect:5} and represents the expected occupation of the stable process exiting the unit ball centered at the origin.

\medskip

In this paper we propose a new approach to the Problem \eqref{eq:1.1} in terms of some deep learning techniques. We will work with the DNNs described formally in the Section \ref{Sect:3}. In particular, we work with DNNs with $H \in \N$ hidden layers, one input and one output layer, each layer with dimension $k_{i} \in \N$, $i=0,...,H+1$. For $i=1,...,H+1$ the weights and biases of the DNN are denoted as $W_i \in \R^{k_i \times k_{i-1}}$ and $B_i \in \R^{k_i}$. The DNN is represented by his weights and biases as $\Phi = ((W_1,B_1),...,(W_{H+1},B_{H+1}))$. For $x_0 \in \R^{k_0}$ the realization of the DNN $\Phi$ is defined as
\[\mathcal{R}(\Phi)(x_0) = W_{H+1}x_{H} + B_{H+1},\]
where $x_i \in \R^{k_i}$, $i=1,...,H$ is defined as
\[
x_i = A(W_{i}x_{i-1}+B_{i}).
\]
Here $A(\cdot)$ is the activation function of the DNN. In our case we work with activation function of ReLu type ($A(x)=\max\{x,0\}$). The number of parameters used to describe the DNN and the dimension of the layers are defined as
\[
\mathcal{P}(\Phi) = \sum_{i=1}^{H+1} k_i (k_{i-1}+1), \quad \mathcal{D}(\Phi) = (k_0,...,k_{H+1}),
\]
respectively. For the approximation of \eqref{u(x)}, we need to assume some hypothesis in relation to the implicated functions. In particular, we suppose that $g$, $\dist(\cdot,\partial D)$, $(\cdot)^{\alpha}$ and $f$ can be approximated by ReLu DNNs $\Phi_g$, $\Phi_{\dist}$, $\Phi_{\alpha}$ and $\Phi_{f}$ that overcome the curse of dimensionality. The full details of the hypotheses can be found in Assumptions \ref{Sup:g}, \ref{Sup:D} and \ref{Sup:f} defined in Section \ref{Sect:6}. Our main result is the following:
\begin{thm}\label{Main} 
	Let $\alpha \in (1,2)$, $p,s \in (1,\alpha)$ such that $s < \frac{\alpha}{p}$ and $q \in \left[s,\frac{\alpha}{p}\right)$. Assume that \eqref{Hg0} and \eqref{Hf0} are satisfied. Suppose that for every $\delta_{\alpha}, \delta_{\dist}, \delta_{f}, \delta_{g} \in (0,1)$ there exist ReLu DNNs $\Phi_g$, $\Phi_{\alpha}, \Phi_{\dist}$ and $\Phi_f$ satisfying Assumptions \ref{Sup:g}, \ref{Sup:D} and \ref{Sup:f}, respectively. Then for every $\epsilon \in (0,1)$, there exists a ReLu DNN $\Psi_{\epsilon}$ and its realization $\mathcal{R}(\Psi_{\epsilon}): D \to \R$ which is a continuous function such that:
	\begin{enumerate}
		\item Proximity in $L^q(D)$: If $u$ is the solution of \eqref{eq:1.1}
		\begin{equation}
			\left(\int_D \left|u(x) - \left(\mathcal{R}(\Psi_{\epsilon})\right)(x)\right|^q dx\right)^{\frac 1q} \leq \epsilon.
		\end{equation}
		\item Bounds: There exists $\widehat{B},\eta>0$ such that %the ReLu DNN $\Psi_{\epsilon}$ satisfies that the maximum number of nodes in a single layer of $\Psi_\epsilon$, denoted $\vertiii{\mathcal{D}(\Psi_{\epsilon})} $, obeys the bound
		\begin{equation}
			\mathcal{P}(\Psi_{\epsilon}) \leq \widehat{B}|D|^{\eta}d^{\eta} \epsilon^{-\eta}.
		\end{equation}
		The constant $\widehat{B}$ depends on $\norm{f}_{L^{\infty}(D)}$, the Lipschiptz constants of $g$, $\mathcal{R}(\Phi_f)$ and $\mathcal{R}(\Phi_{\alpha})$, and on $\diam(D)$.
	\end{enumerate}
\end{thm} 
\subsection{Idea of the proof} In this section we sketch the proof of Theorem \ref{Main}. Following \cite{Grohs}, we will work with each of the terms in \eqref{u(x)2} by separated, and for each of them we will prove that the value for the solution can be approximated by DNN that do not suffer the curse of dimensionality. For full details see Propositions \ref{Prop:homo} and \ref{Prop:6p1}. We will sketch the proof of the homogeneous part, the other term is pretty similar.

\medskip

The proof will be divided in several steps: First of all we approximate in expectation the respective term of \eqref{u(x)2} with Monte Carlo simulations. Due the nature of WoS processes we need to approximate also random variables such as copies of the number of iterations in the WoS process $N$ and copies of the norm of isotropic $\alpha$-stable process exiting a unit ball centered at the origin, $|X_{\sigma_{B(0,1)}}|$. Next we find Monte Carlo simulations that approximate the term of \eqref{u(x)2} punctually and latter in $L^q(D)$, whose copies satisfy also the approximations of $N$ and $|X_{\sigma_{B(0,1)}}|$ mentioned before. Those additional terms help us to find ReLu DNNs that approximates the WoS proccesses, then we can approximate the Monte Carlo simulation by a ReLu DNN. Next we use the found ReLu DNN to approximate the term of \eqref{u(x)2}, using suitable choices of the parameters in order to have an accuracy of $\epsilon \in (0,1)$ in the approximation. Finally we study the found ReLu DNN, and we prove that the total number of parameters that describes de DNN are at most polynomial in the dimension and in the reciprocal of the accuracy.

\subsection{Discussion} As stated before, for the approximation of the solutions of Problem \eqref{eq:1.1} we follow the ideas presented in the work of Grohs and Herrmann \cite{Grohs}, with several changes due the non local nature of the fractional Laplacian. In particular:
\begin{enumerate}
	\item The non local problem has the boundary condition $g$ defined on the complement of the domain $D$ and the local problem has $g$ defined on $\partial D$. This variation changes the way to approximates $g$ by DNNs. The Assumption \ref{Sup:g} is classical in the literature for functions defined in  unbounded sets (see, e.g. \cite{Hutz}).
	\item The isotropic $\alpha$-stable process associated to the fractional Laplacian has no second moment, therefore the approximation can not be approximated in $L^{2}(D)$, but in $L^{q}(D)$ for some  suitable $q<2$.
	\item The process associated to the local case is a Brownian motion that is continuous, then the norm of the process exiting the unit ball centered at the origin is equal to 1, i.e, $|X_{\sigma_{B(0,1)}}|=1$. In the non local case the isotropic $\alpha$-stable is a pure jump process, therefore $|X_{\sigma_{B(0,1)}}|>1$. This is the reason that in our proof we approximate the copies of this random variable to have that the copies of $X_{\sigma_{B(0,1)}}$ exits near the domain $D$.
	\item The last notable difference is about the sum that appears in \eqref{u(x)2}: The value of $r_n$ is raised to the power of $\alpha$, then we need an extra hypothesis for the approximation of the function $(\cdot)^{\alpha}$ by DNNs. In the local case $r_n$ is squared, then can be approximated by DNNs using the Lemma \ref{lem:DNN_mult} stated in the Section \ref{Sect:3}.
\end{enumerate}

\medskip
A similar way to represent solutions to the Problem \eqref{eq:1.1} is found in the work of Gulian and Pang \cite{Gulian}. Thanks to stochastic calculus results (see, e.g. \cite{Applebaum,Bertoin}), the processes described in that work and the isotropic $\alpha$-stable processes are similar. In addition, in that paper they found a Feynman-Kac formula for the parabolic generalized problem for the associated fractional Laplacian. In a possible extension we could see if the parabolic generalized problem can be adapted to our setting, i.e. the solution of the parabolic case can be approximated by DNNs that overcome the curse of dimensionality.

\medskip

Finally, just before finishing this work, we learned about the research by Changtao Sheng et al. \cite{ultimo}, who have showed numerical simulations of the Problem \eqref{eq:1.1} using similar Monte Carlo methods. However, our methods are radically different in terms of the main goal, which is here to approximate the solution by DNNs in a rigorous fashion.

\section{Preliminaries}\label{Sect:2}

%In this section we blahb

\subsection{Notation} Along this paper, we shall use the following conventions:
\begin{itemize}
%\item $(\Omega, \mathcal F, \PP)$
\item $\N=\{1,2,3,...\}$ will be the set of Natural numbers.
\item For any $q\geq 1$, $(\Omega, \mathcal F, \mu)$ measure space,  $L^q(\Omega,\mu)$ denotes the Lebesgue space of order $q$ with the measure $\mu$. If $\mu$ is the Lebesgue measure, then the Lebesgue space will be denoted as $L^q(\Omega)$.
%\item For any $q \geq 1$, and $S \subset \R^d$, $L^{q}(S)$ denotes the Lebesgue space of order $q$, with the Lebesgue measure.
\end{itemize}

\subsection{A quick review on Lévy processes}\label{Sect:2p1} Let us introduce a brief review on the Lévy processes needed for the proof of the main results. For a detailed account on these processes, see e.g. \cite{Applebaum, Bertoin,AK2,Schilling}.

\begin{defn}
	 $L:=(L_{t})_{t \geq 0}$ is a Lévy process in $\R^d$ if it satisfies $L_0 = 0$ and
	
	\begin{enumerate}
		\item[i)] $L$ has independent increments, namely, for all $n \in \N$ and for each $0 \leq t_{1} < ... < t_{n} < \infty $, the random variables $(L_{t_{2}}-L_{t_{1}},...,L_{t_{n}}-L_{t_{n-1}})$ are independent.
		\item[ii)] $L$ has stationary increments, namely, for all $s \geq 0$, $L_{t+s}-L_{s}$ and $L_{t}$ have the same law.
		\item[iii)] $L_{t}$ is continuous on the right and has limit on the left for all $t > 0$ (i.e., $(L_{t})_{t \geq 0}$ is \textbf{càdlàg}).
	\end{enumerate}
\end{defn}

Examples of Lévy processes are the Brownian motion, but also processes with jumps such as the \emph{Poisson process} and the \emph{compound Poisson process} \cite{Applebaum}.

\begin{defn}
	The Poisson process of intensity $\lambda>0$ is a Lévy process $N$ taking values in $\N \cup \{0\}$ wherein each $N(t) \sim Poisson(\lambda t)$, then we have
	\[
	\PP(N(t) = n) = \frac{(\lambda t)^n}{n!} e^{-\lambda t}.
	\]
\end{defn} 

\begin{defn}
	Let $(Z(n))_{n \in \N}$ be a sequence of i.i.d. random variables taking values in $\R^d$ with law $\mu$ and let $N$ be a Poisson process with intensity $\lambda$ that is independent of all the $Z(n)$. The compound Poisson process $Y$ is defined as follows:
	\[
	Y(t) = Z(1) + ... + Z(N(t)),
	\]	
	for each $t \geq 0$.
\end{defn}

Another important element is the so-called Lévy's characteristic exponent.
\begin{defn}
Let $(L_t)_{t \geq 0}$ be a Lévy process in $\R^d$.
	\begin{itemize}
	\item Its characteristic exponent $\Psi:\R^d \rightarrow \C$ is the continuous function that satisfies $\Psi(0)=0$ and for all $t \geq 0$,
	\begin{equation}\label{eq:EC}
		 \E\left[e^{i\xi \cdot L_{t}}\right] = e^{-t\Psi(\xi)}, \qquad \xi \in \R^d\setminus \{0\}.
	\end{equation}
	\item A Lévy triple is $(b,A,\Pi)$, where $b \in \R^d$, $A \in \R^{d \times d}$ is a positive semi-definite matrix, and $\Pi$ is a Lévy measure in $\R^d$, i.e.
	\begin{equation}
		\Pi(\{0\})=0 \quad \hbox{  and  } \quad  \int_{\R^d} (1 \wedge |z|^2) \Pi(dz) < \infty.
	\end{equation}
	\end{itemize}
\end{defn}
A Lévy process is uniquely determined via its Lévy triple and its characteristic exponent.
\begin{teo}[Lévy-Khintchine, \cite{AK2}]
	Let $(b,A,\Pi)$ be a Lévy triple. Define for each $\xi \in \R^d$
	\begin{equation}\label{eq:LK}
		\Psi(\xi) = ib\cdot \xi + \frac{1}{2} \xi \cdot A\xi + \int_{\R^d} \left( 1 - e^{i\xi\cdot z} + i\xi\cdot z {\bf 1}_{\{|z|< 1\}} \right) \Pi(dz).
	\end{equation}
	If $\Psi$ is the characteristic exponent of a Lévy process with triple $(b,A,\Pi)$ in the sense of \eqref{eq:EC}, then it necessarily satisfies \eqref{eq:LK}. Conversely, given \eqref{eq:LK} there exists a probability space $(\Omega, \mathcal{F},\PP)$, on which a Lévy process is defined having characteristic exponent $\Psi$ in the sense of \eqref{eq:EC}.
\end{teo}
For the next Theorem we define the \emph{Poisson random measure}.
\begin{defn}
	Let $(S,\mathcal{S},\eta)$ be an arbitrary $\sigma$-finite measure space and $(\Omega,\mathcal{F},\PP)$ a probability space. Let $N: \Omega \times \mathcal{S} \to \N \cup \{0,\infty\}$ such that $(N(\cdot,A))_{A \in \mathcal{S}}$ is a family of random variables defined  on $(\Omega,\mathcal{F},\PP)$. For convenience we supress the dependency of $N$ on $\omega$. $N$ is called a Poisson random measure on $S$ with intensity $\eta$ if
	\begin{enumerate}
		\item[i)] For mutually disjoint $A_1,...,A_n$ in $\mathcal{S}$, the variables $N(A_1),...,N(A_n)$ are independent.
		\item[ii)] For each $A \in \mathcal{S}$, $N(A) \sim \hbox{Poisson}(\eta(A))$,
		\item[iii)] $N(\cdot)$ is a measure $\PP$-almost surely.
	\end{enumerate} 
\end{defn}
\begin{rem}
	In the next theorem we use $S \subset [0,\infty) \times \R^d$, and the intensity $\eta$ will be defined on the product space.
\end{rem}
%\begin{rem}
%	Let $(U,\mathcal{U})$ be a measurable space. Suppose that $S = [0,\infty) \times U$, $\mathcal{S} = \mathcal{B}([0,\infty))\times U$ on the product space, and $\eta(dt,du) = dt \times \mu(t,du)$ for some family $(\mu(t,\cdot))_{t\geq0}$ on the space $(U,\mathcal{U})$. The associated Poisson random measure, written $N(dt,du),$ $(t,u) \in [0,\infty) \times U$ is called a \emph{Poisson point process}.
%\end{rem}
From the Lévy-Khintchine formula, every Lévy process can be decomposed in three components: a Brownian part with drift, the large jumps and the compensated small jumps of the process $L$. 
\begin{teo}[Lévy-Itô Decomposition]\label{teo:levy_ito}
	 Let $L$ be a Lévy process with triple $(b,A,\Pi)$. Then there exists process $L^{(1)}$, $L^{(2)}$ y $L^{(3)}$ such that for all $t \geq 0$
	 \[
	 L_t = L^{(1)}_t + L^{(2)}_t + L^{(3)}_t,
	 \] 
	 where 
	\begin{enumerate}
		\item $L^{(1)}_t = bt + AB_{t}$ and $B_{t}$ is a standard $d$-dimensional Brownian motion.
		\item $L^{(2)}_t$ satisfies
		\[
		\displaystyle L^{(2)}_t = \int_{0}^{t}\int_{|z|\geq 1} zN(ds,dz),
		\]
		where $N(ds,dz)$ is a Poisson random measure on $[0,\infty)\times\{z \in \R^d: |z| \geq 1\}$ with intensity 
		\[
		\displaystyle \Pi(\{z\in \R^d :|z|\geq 1\})dt \times \frac{\Pi(dz)}{\Pi(\{z\in \R^d :|z|\geq 1\})}.
		\]
		If $\Pi(\{z\in \R^d :|z|\geq 1\})=0$, then $L^{(2)}$ is the process identically equal to 0. In other words $L^{(2)}$ is a compound Poisson process.
		\item The process $L^{(3)}_t$ satisfies
		\[
		\displaystyle L^{(3)}_t = \int_{0}^{t} \int_{|z|<1} z \widetilde{N}(ds,dz),
		\]
		where $\widetilde{N}(ds,dz)$ is the compensated Poisson random measure, defined by
		\[
		\displaystyle \widetilde{N}(ds,dz) = N(ds,dz) - ds\Pi(dz),
		\]
		with $N(ds,dz)$ the Poisson random measure on $[0,\infty)\times \{z \in \R^d: |z|<1\}$ with intensity 
		\[
		\displaystyle ds \times \left.\Pi(dz)\right|_{\{z \in \R^d : |z| < 1\}}.
		\]
	\end{enumerate}
\end{teo}

\subsection{Lévy processes and the Fractional Laplacian} A particular set of Lévy processes are the so-called isotropic $\alpha$-stable processes, for $\alpha \in (0,2)$. The following definitions can be found in \cite{AK2} in full detail.
\begin{defn}\label{def:iso}
	Let $\alpha \in (0,2)$. $X := (X_t)_{t \geq 0}$ is an isotropic $\alpha$-stable process if $X$ has a Lévy triple $(0,0,\Pi)$, with 
	\begin{equation}\label{eq:levy_measure}
		\Pi(dz) = 2^{-\alpha} \pi^{-d/2} \frac{\Gamma((d+\alpha)/2)}{|\Gamma(-\alpha/2)|} \frac{1}{|z|^{\alpha + d}} dz, \hspace{.5cm} z \in \R^d.
	\end{equation}
	Recall that $\Gamma$ here is the Gamma function.
\end{defn}

\begin{defn}[Equivalent definitions of an isotropic $\alpha$-stable process] $\left.\right.$
	\begin{enumerate}
		\item $X$ is an isotropic $\alpha$-stable process iff
		\begin{equation}\label{eq:scaling}
			\hbox{for all }\quad c>0, \quad (cX_{c^{-\alpha}t})_{t \geq 0} \quad \hbox{and} \quad X \quad \hbox{have the same law,}
		\end{equation}
		and for all the orthogonal transformations $U$ on $\R^d$,
		\[
		(UX_t)_{t\geq 0} \quad \hbox{and} \quad X \quad \hbox{have the same law.}
		\]
		\item An isotropic $\alpha$-stable process is an stable process whose characteristic exponent is given by
		\begin{equation}
			\Psi(\xi) = |\xi|^\alpha, \hspace{.5cm} \xi \in \R^d.
		\end{equation}
	\end{enumerate}
\end{defn}
\begin{rem}
	For the first definition, we say that $X$ satisfies the \emph{scaling property} and is \emph{rotationally invariant}, respectively.
\end{rem}
Note by Definition \ref{def:iso} and by Lévy-Itô decomposition (Theorem \ref{teo:levy_ito}),an isotropic $\alpha$-stable process can be decomposed as
\begin{equation}
	X_{t} = \int_0^t \int_{|z| \geq 1} zN(ds,dz) + \int_0^t \int_{|z| < 1} z \widetilde{N}(ds,dz).
\end{equation}
From this equation, we can conclude that every isotropic $\alpha$-stable process is a pure jump process, whose jumps are determined by the Lévy measure defined in \eqref{eq:levy_measure}. We enunciate further properties about these processes:
\begin{teo}
	Let $g$ be a locally bounded, submultiplicative function and let $L$ a Lévy process, then the following are equivalent:	
	\begin{enumerate}
		\item $\E[g(L_t)]< \infty$ for some $t > 0$.
		\item $\E[g(L_t)] < \infty $ for all $t > 0$.
		\item $\int_{|z|>1} g(z) \Pi(dz) < \infty.$
	\end{enumerate}
\end{teo}
An important result from the previous theorem gives necessary and sufficient conditions for the existence of the $p$ moment of an isotropic $\alpha$-stable process.
\begin{cor}\label{cor:Xmoment}
	Let $X$ be an $\alpha$-stable process and $p>0$, then the following are equivalent
	\begin{enumerate}
		\item $p < \alpha$.
		\item $\E[|X_t|^p]< \infty$ for some $t>0$.
		\item $\E[|X_t|^{p}]< \infty$ for all $t > 0$.
		\item $\int_{|z|>1} |z|^{p} \Pi(dz) < \infty$.
	\end{enumerate}
\end{cor}
\begin{rem}
	If $\alpha \in (0,1)$, by this corollary we have that $X$ has no first moment. Otherwise, if $\alpha \in (1,2)$ then it has finite first moment, but no second moment.
\end{rem}

\subsection{Type $s$ spaces and Monte Carlo Methods.}
We now introduce some results that controls the difference between the expectation of a random variable and a Monte Carlo operator associated to his expectation in $L^{p}$ norm, $p>1$. For more details see \cite{Hutz2}. In the following results are simplified the results of \cite{Hutz2}. Along 
this section we work with real valued Banach spaces.
\medskip

We start with some concepts related to Banach spaces. The reader can consult \cite{Hutz2,ProbBan} for more details in this topic.%A real valued 
\begin{defn}
	Let $(\Omega, \mathcal{F}, \PP)$ be a probability space, let $J$ be a set, and let $r_j : \Omega \to \{-1,1\}$, $j \in J$, be a family of independent random variables with for all $j \in J$,
	\[
	\PP(r_j = 1) = \PP(r_j = -1) = \frac{1}{2}.
	\]
	Then we say that $(r_j)_{j \in J}$ is a $\PP$-Rademacher family.
\end{defn}

\begin{defn}\label{def:stype}
	Let $(r_j)_{j \in \N}$ a $\PP$-Rademacher family. Let $s \in (0,\infty)$. A Banach space $(E,\norm{\cdot}_{E})$ is said to be of type $s$ if there is a constant $C$ such that for all finite sequences $(x_j)$ in $E$,
	\[
	\E\left[\norm{\sum_{j}r_j x_j}^s_E\right]^{\frac{1}{s}} \leq C \left(\sum_{j} \norm{x_j}_{E}^{s}\right)^{\frac{1}{s}}.
	\]
	The supremum of the constants $C$ is called the type $s$-constant of $E$ and it is denoted as $\mathscr{T}_s(E)$.
\end{defn}

\begin{rem}
	The existence of a finite constant $C$ in Definition \ref{def:stype} is valid for $s \leq 2$ only (see, e.g. \cite{ProbBan} Section 9 for more details).
\end{rem}

\begin{rem}
	Any Banach space $(E,\norm{\cdot}_E)$ is of type $1$. Moreover, triangle inequality ensures that $\mathscr{T}_1(E)=1$.
\end{rem}
%\begin{defn}
%	Let $s \in (0,\infty)$ and let $(E,\norm{\cdot}_E)$ be a Banach space. Then we denote by $\mathscr{T}_s(E) \in [0,\infty]$ the extended real number given by
%	\[
%	\mathscr{T}_s(E) = \sup \left(\left\{r\in[0,\infty): \begin{aligned}
%		&\exists \hbox{ probability space } (\Omega,\mathcal{F},\PP):\\
%		&\exists \PP\hbox{-Rademacher family } (r_j)_{j \in \N}:\\
%		&\exists k \in \N: \exists x_1,...,x_k \in E\setminus\{0\}:\\
%		&\centering r= \frac{\E\left[\norm{\sum_{j=1}^{k} r_j x_j}^s_E\right]^{\frac{1}{s}}}{\left(\sum_{j=1}^{k} \norm{x_j}^s_E\right)^{\frac{1}{s}}}
%	\end{aligned}\right\} \cup \{0\}\right),
%	\]
%	and we call $\mathscr{T}_s(E)$ the type $s$-constant of $E$.
%\end{defn}
%\begin{defn}
%	Let $s \in (0,\infty)$ and let $(E,\norm{\cdot}_E)$ be a Banach space which satisfies $\mathscr{T}_s(E) < \infty$. Then we say that $(E,\norm{\cdot}_E)$ has type $s$.
%\end{defn}
%\begin{rem}
%	For all $s \in (2,\infty)$ and all Banach spaces $(E,\norm{\cdot}_E)$ with $E \neq \{0\}$ it holds that $\mathscr{T}_s(E) = \infty$.
%\end{rem}
\begin{rem}
	Notice that for all Banach spaces $(E,\norm{\cdot}_E)$, the function $(0,\infty) \ni s\to \mathscr{T}_s(E) \in [0,\infty]$ is non-decreasing. This implies for all $s \in (0,1]$ and all Banach spaces $(E,\norm{\cdot}_E)$ with $E \neq \{0\}$ that $\mathscr{T}_s(E) = 1$.
\end{rem}
\begin{rem}
	For all $s \in (0,2]$ and all Hilbert spaces $(H,\langle\cdot,\cdot \rangle_H, \norm{\cdot}_H)$ with $H \neq \{0\}$ it holds that $\mathscr{T}_s(H) = 1$.
\end{rem}
\begin{defn}
	Let $(r_j)_{j\in \N}$ a $\PP$-Rademacher family. Let $q,s \in (0,\infty)$ and $(E,\norm{\cdot}_E)$ be a Banach space. The $\mathscr{K}_{q,s}$ $(q,s)$-Kahane-Khintchine constant of the space $E$ is the extended real number given by the supremum of a constant $C$ such that for all finite sequences $(x_j)$ in $E$,
	\[
		\E\left[\norm{\sum_{j}r_j x_j}^q_E\right]^{\frac{1}{q}} \leq C\,\E\left[\norm{\sum_{j}r_j x_j}^s_E\right]^{\frac{1}{s}}
	\]
	
\end{defn}
%\begin{defn}
%	Let $q,s \in (0,\infty)$. Then we denote by $\mathscr{K}_{q,s} \in [0,\infty]$ the extended real number given by
%	\[
%	\mathscr{K}_{q,s} = \sup \left\{r \in [0,\infty) \quad : \quad \begin{aligned}
%		&\exists \R\hbox{-Banach space } (E,\norm{\cdot}_E):\\
%		&\exists \hbox{ probability space } (\Omega,\mathcal{F},\PP):\\
%		&\exists \PP\hbox{-Rademacher family } (r_j)_{j \in \N} :\\
%		&\exists k \in \N \exists x_1,...,x_k \in E\setminus\{0\}:\\
%		& r = \frac{\E\left[\norm{\sum_{j=1}^{k}r_j x_j}^q_E\right]^{\frac{1}{q}}}{\E\left[\norm{\sum_{j=1}^{k}r_j x_j}^s_E\right]^{\frac{1}{s}}}
%	\end{aligned}\right\},
%	\]
%	\[
%	\begin{aligned}
%	& \mathscr{K}_{q,s} \\
%	& := \sup \left\{r \geq 0 \; : \; \begin{aligned}
%		&\hbox{There exist a Banach space $(E,\norm{\cdot}_E)$, a prob. space } (\Omega,\mathcal{F},\PP):\\
%		& \hbox{a $\PP$-Rademacher family  $(r_j)_{j \in \N}$, $k \in \N$ and $x_1,...,x_k \in E\setminus\{0\}$ s.t. }  \\
%		&  r = \frac{\E\left[\norm{\sum_{j=1}^{k}r_j x_j}^q_E\right]^{\frac{1}{q}}}{\E\left[\norm{\sum_{j=1}^{k}r_j x_j}^s_E\right]^{\frac{1}{s}}}
%	\end{aligned}\right\},
%	\end{aligned}
%	\]
%	and we call $\mathscr{K}_{q,s}$ the $(q,s)$-Kahane-Khintchine constant.
%\end{defn}
\begin{rem}
	For all $q,s \in (0,\infty)$ it holds that $\mathscr{K}_{q,s} < \infty$. Moreover, if $q\leq s$ by Jensen's inequality implies that $\mathscr{K}_{q,s} = 1$.
\end{rem}
\begin{defn}
	Let $q,s \in (0,\infty)$ and let $(E,\norm{\cdot}_E)$ be a Banach space. Then we denote by $\Theta_{q,s}(E) \in [0,\infty]$ the extended real number given by
	\[
	\Theta_{q,s}(E) = 2 \mathscr{T}_{s}(E)\mathscr{K}_{q,s}.
	\]
\end{defn}
Consider the case $(\R,|\cdot|)$. Notice that $(\R,\langle \cdot,\cdot \rangle_{\R}, |\cdot|)$ is a Hilbert space with the inner product $\langle x,y \rangle_{\R} = xy$. Then it holds for all $s \in (0,2]$ that
\[
\mathscr{T}_{s} := \mathscr{T}_{s}(\R) = 1,
\]
in other words, $(\R,|\cdot|)$ has type $s$ for all $s \in (0,2]$. Moreover, it holds for all $q \in (0,\infty)$, $s \in (0,2]$ that
\[
\Theta_{q,s} := \Theta_{q,s}(\R) = 2\mathscr{K}_{q,s} < \infty.
\]
With this in mind, we enunciate a particular version of the Corollary 5.12 found in \cite{Hutz2}, replacing the Banach space $(E,\norm{\cdot}_E)$ by $(\R,|\cdot|)$.
\begin{cor}\label{cor:MCq}
	Let $M \in \N$, $s \in [1,2]$, $(\Omega,\mathcal{F},\PP)$ be a probability space, and let $\xi_j \in L^1(\PP,|\cdot|)$, $j \in \{1,...,M\}$, be independent and identically distributed. Then, for all $q \in [s,\infty]$,
	\begin{equation}\label{eq:MCq}
		\begin{aligned}
		\norm{\E[\xi_1] - \frac{1}{M} \sum_{j=1}^{M} \xi_j}_{L^{q}(\Omega,\PP)} &= \frac{1}{M} \E\left[\left|\sum_{j=1}^{M}\xi_j - \E \left[\sum_{j=1}^{M} \xi_j\right]\right|^{q}\right]^{\frac{1}{q}}\\
		&\leq \frac{\Theta_{q,s}}{M^{1-\frac{1}{s}}} \E\left[\left|\xi_1 - \E[\xi_1]\right|^{q}\right]^{\frac{1}{q}}.
		\end{aligned}
	\end{equation}
\end{cor}
\begin{rem}
	The choice of the Banach space as $(\R,|\cdot|)$ ensures that $\Theta_{q,s}$ is finite and the bound above converges for suitable $M$ large.
\end{rem}

\section{Deep Neural Networks}\label{Sect:3}

In this section we review recent results on the mathematical analysis of neural networks needed for the proof of the main theorem. For a detailed description, see e.g. \cite{Grohs2,Hutz,Jentz1}.

\subsection{Setting} for $d \in \N$ define 
\[
A_d : \R^d \rightarrow \R^d
\] 
the ReLU activation function such that for all $z \in \R^d$, $z=(z_1,...,z_d)$, with
\[
A_d(z)=(\max\{z_1,0\},...,\max\{z_d,0\}).
\] 
Let also
\begin{enumerate}
\item[(NN1)] $H \in \N$ be the number of hidden layers; 
\item[(NN2)] $(k_i)_{i=0}^{H+1}$ be a positive integer sequence; 
\item[(NN3)] $W_{i} \in \R^{k_i \times k_{i-1}}$, $B_{i} \in \R^{k_i}$, for any $i=1,...,H+1$ be the weights and biases, respectively;
\item[(NN4)] $x_0 \in \R^{k_0}$, and for $i = 1,...,H$ let 
\begin{equation}\label{x_i}
x_i = A_{k_i}(W_ix_{i-1}+B_i).
\end{equation}
\end{enumerate}
We call 
\begin{equation}\label{eq:DNN_def}
\Phi := (W_i,B_i)_{i=1}^{H+1} \in \prod_{i=1}^{H+1} \left(\R^{k_i \times k_{i-1}} \times \R^{k_i}\right)
\end{equation}
the DNN associated to the parameters in (NN1)-(NN4). The space of all DNNs in the sense of \eqref{eq:DNN_def} is going to be denoted by $\textbf{N}$, namely
\[
\textbf{N} = \bigcup_{H \in \N} \bigcup_{(k_0,...,k_{H+1})\in \N^{H+2}} \left[\prod_{i=1}^{H+1} \left(\R^{k_i \times k_{i-1}} \times \R^{k_i}\right)\right].
\]
Define the realization of the DNN $\Phi \in \textbf{N}$ as
\begin{equation}\label{Realization}
	\mathcal{R}(\Phi)(x_0) =  W_{H+1}x_H + B_{H+1}.
\end{equation}
Notice that $\mathcal{R}(\Phi) \in C(\R^{k_0},\R^{k_{H+1}})$. For any $\Phi \in \textbf{N}$ define
\begin{equation}\label{P_D}
	\mathcal{P}(\Phi) = \sum_{n	=1}^{H+1} k_n (k_{n-1}+1), \qquad \mathcal{D}(\Phi) = (k_0,k_1,...,k_{H+1}),
\end{equation}
and
\begin{equation}\label{norma}
\vertiii{  \mathcal{D}(\Phi )  } = \max\{ k_0,k_1,...,k_{H+1} \}.
\end{equation}
The entries of $(W_i,B_i)_{i=1}^{H+1}$ will be the weights of the DNN, $\mathcal{P}(\Phi)$ represents the number of parameters used to describe the DNN, working always with fully connected DNNs, and $\mathcal{D}(\Phi)$ representes the dimension of each layer of the DNN. Notice that $\Phi \in \textbf{N}$ has $H+2$ layers: $H$ of them hidden, one input and one output layer.
\begin{rem}
	For $\Phi \in \textbf{N}$ one has
	\[
	\vertiii{\mathcal{D}(\Phi)} \leq \mathcal{P}(\Phi) \leq (H+1)\vertiii{\mathcal{D}(\Phi)} (\vertiii{\mathcal{D}(\Phi)}+1).
	\]
	Indeed, from the definition of $\vertiii{\cdot}$,
	\[
	\vertiii{\mathcal{D}(\Phi)} \leq \sum_{n=1}^{H+1} k_n \leq \mathcal{P}(\Phi).
	\]
	In addition, the definition of $\mathcal{P}(\Phi)$ implies that
	\[
	\mathcal{P}(\Phi) \leq \sum_{n=1}^{H+1} \vertiii{\mathcal{D}(\Phi)}(\vertiii{\mathcal{D}(\Phi)}+1) = (H+1)\vertiii{\mathcal{D}(\Phi)}(\vertiii{\mathcal{D}(\Phi)}+1).
	\]
\end{rem}

\begin{rem}\label{rem:CoD}
	From the previous remark one has
	\[
	\mathcal{P}(\Phi) \leq 2(H+1)\vertiii{\mathcal{D}(\Phi)}.
	\]
	If $\vertiii{\mathcal{D}(\Phi)}$ grows at most polynomially in both the dimension of the input layer and the reciprocal of the accuracy $\varepsilon$ of the DNN, then $\mathcal{P}(\Phi)$ satisfies that bound too. This means that, with the right bound on $\vertiii{\mathcal{D}(\Phi)}$, the DNN $\Phi$ do not suffer the curse of dimensionality, in the sense established in Section \ref{Sect:1}.
\end{rem}

\subsection{Operations} In this section we summarize some operations between DNNs. We start with the definition of two vector operators 
\begin{defn}
	Let $\textbf{D} = \bigcup_{H \in \N} \N^{H+2}$. 
	\begin{enumerate}
		\item Define $\odot: \textbf{D} \times \textbf{D} \rightarrow \textbf{D}$ such that for all $H_1,H_2 \in \N$, $\alpha = (\alpha_0,...,\alpha_{H_1+1}) \in \N^{H_1+2}$, $\beta = (\beta_0,...,\beta_{H_2+1}) \in \N^{H_2+2}$ it satisfied
		\begin{equation}
			\alpha \odot \beta = (\beta_0,\beta_1,...,\beta_{H_2},\beta_{H_2+1}+\alpha_0,\alpha_1,...,\alpha_{H_1+1}) \in \N^{H_1+H_2+3}.
		\end{equation}
		\smallskip
		\item Define $\boxplus: \textbf{D} \times \textbf{D} \rightarrow \textbf{D}$ such that for all $H \in \N$, $\alpha = (\alpha_0,...,\alpha_{H+1}) \in \N^{H+2}$, $\beta = (\beta_0,...,\beta_{H+1}) \in \N^{H+2}$ it satisfied
		\begin{equation}
			\alpha \boxplus \beta = (\alpha_0, \alpha_1 + \beta_1,...,\alpha_H + \beta_H, \beta_{H+1}) \in \N^{H+2}.
		\end{equation}
		\smallskip
		\item Define $\mathfrak{n}_n \in \textbf{D}$, $n \in \N$, $n \geq 3$ as
		\begin{equation}
			\mathfrak{n}_n = (1,\underbrace{2,\ldots,2}_{(n-2)\hbox{-times}},1) \in \N^n.
		\end{equation} 
	\end{enumerate}
\end{defn}

\begin{rem}
	From these definitions and the norm $\vertiii{\cdot}$ defined in \eqref{norma}, we the following bounds are clear
	\begin{enumerate}
		\item For $H_1, H_2 \in \N$, $\alpha \in \N^{H_1 + 2}$ and $\beta \in \N^{H_2 + 2}$,
		\[
		\vertiii{\alpha \odot \beta} \leq \max\{\vertiii{\alpha},\vertiii{\beta},\alpha_0 + \beta_{H_2 + 1}\}.
		\]
		\item For $H \in \N$ and $\alpha,\beta \in \N^{H+2}$,
		\[
		\vertiii{\alpha \boxplus \beta} \leq \vertiii{\alpha} + \vertiii{\beta}.
		\]
		\item For $n \in \N$, $n \geq 3$, $\vertiii{\mathfrak{n}_n} = 2$.
	\end{enumerate}
\end{rem}

Now we state classical operations between DNNs. For a full details of the next Lemmas, the reader can consult e.g. \cite{Hutz,Jentz1}.

\begin{lema}\label{lem:DNN_id}
	Let $Id_{\R} : \R \rightarrow \R$ be the identity function on $\R$ and let $H \in \N$. Then $Id_{\R} \in \mathcal{R}\left(\{\Phi \in \textbf{N}: \mathcal{D}(\Phi)=\mathfrak{n}_{H+2}\}\right)$.
\end{lema}

\begin{rem}
	A similar consequence is valid in $\R^d$. Let $Id_{\R^d} : \R^d \to \R^d$ be the identity function on $\R^d$ and let $H \in \N$. Therefore $Id_{\R^d} \in \mathcal{R}\left(\{\Phi \in \textbf{N}: \mathcal{D}(\Phi)=d\mathfrak{n}_{H+2}\}\right)$. The case with $H=3$ is proved on \cite{Jentz1}.
\end{rem}

\begin{rem}
	Let $H \in \N$ and $\Phi \in \textbf{N}$ such that $\mathcal{R}(\Phi) = Id_{\R^d}$. Then by Remark 3.3 we have that $\vertiii{\mathcal{D}(\Phi)} = 2d$.
\end{rem}

%\begin{lema}\label{lem:DNN_afin}
%	Let $d,m \in \N$, $\lambda \in \R$, $b \in \R^d$, $a \in \R^m$. Let $\Psi \in \textbf{N}$ such that $\mathcal{R}(\Psi) \in C(\R^d,\R^m)$. Therefore
%	\begin{equation}
%		\lambda\left(\left(\mathcal{R}(\Psi)\right)(\cdot + b) + a\right) \in \mathcal{R}(\{\Phi \in \textbf{N}: \mathcal{D}(\Phi) = \mathcal{D}(\Psi)\}).
%	\end{equation}
%\end{lema}

\begin{lema}\label{lem:DNN_comp}
	Let $d_1,d_2,d_3 \in \N$, $f \in C(\R^{d_2},\R^{d_3})$, $g \in C(\R^{d_1},\R^{d_2})$, $\alpha,\beta \in \textbf{D}$ such that $f \in \mathcal{R}(\{\Phi \in \textbf{N}: \mathcal{D}(\Phi) = \alpha \})$ and $g \in \mathcal{R}(\{\Phi \in \textbf{N}: \mathcal{D}(\Phi) = \beta \})$. Therefore $(f \circ g) \in \mathcal{R}(\{\Phi \in \textbf{N}: \mathcal{D}(\Phi) = \alpha \odot \beta\})$.
\end{lema}

\begin{rem}
	Let $\Phi_f, \Phi_g, \Phi \in \N$ such that $\mathcal{R}(\Phi_f) = f$, $\mathcal{R}(\Phi_g) = g$ and $\mathcal{R}(\Phi) = f \circ g$. Then by Remark 3.3 it follows that
	\[
	\vertiii{\mathcal{D}(\Phi)} \leq \max\{\vertiii{\mathcal{D}(\Phi_f)},\vertiii{\mathcal{D}(\Phi_g)},2d_2\}.
	\]
\end{rem}

\begin{lema}\label{lem:DNN_sum}
	Let $M,H,p,q \in \N$, $h_i \in \R$, $\beta_i \in \textbf{D}$, $f_i \in C(\R^p,\R^q)$, $i=1,...,M$ such that for all $i=1,...,M$ $\dim(\beta_i) = H+2$ and $f_i \in \mathcal{R}(\{\Phi \in \textbf{N}: \mathcal{D}(\Phi) = \beta_i\})$. Then
	\begin{equation}
		\sum_{i=1}^{M} h_i f_i \in \mathcal{R}\left(\left\{\Phi \in \textbf{N}: \mathcal{D}(\Phi) = \overset{M}{\underset{i=1}{\boxplus}} \beta_i \right\}\right).
	\end{equation}
\end{lema}
\begin{rem}
	For $i=1,...,M$ let $\Phi_i \in \textbf{N}$ such that $\mathcal{R}(\Phi_i)=f_i$ and let $\Phi \in \textbf{N}$ such that
	\[
	\mathcal{R}(\Phi) = \sum_{i=1}^{M} h_i f_i.
	\]
	It follows from Remark 3.3 that
	\[
	\vertiii{\mathcal{D}(\Phi)} \leq \sum_{i=1}^{M} \vertiii{\mathcal{D}(\Phi_i)}.
	\]
\end{rem}

%\begin{cor}
%	Let $M,p,q \in \N$, $M_i \in \N$, $h_i \in \R$, $\beta_i \in \textbf{D}$, $f_i \in C(\R^p,\R^q)$, $i=1,...,M$ such that for all $i=1,...,M$ $\dim(\beta_i) = H_i + 2$ and $f_i \in \mathcal{R}\left(\left\{\Phi \in \textbf{N}: \mathcal{D}(\Phi) = \beta_i\right\}\right)$. Let $H = \max\{H_1,...,H_M\}$. Therefore
%	\[
%	\sum_{i=1}^{M} h_i f_i \in \mathcal{R}\left(\left\{\Phi \in \textbf{N}: \mathcal{D}(\Phi) = \overunderset{M}{i=1}{\boxplus}\left(d\mathfrak{n}_{H-H_i+2} \odot \beta_i\right)\right\}\right).
%	\]
%\end{cor}
%\begin{rem}
%	For $i=1,...,M$ let $\Phi_i \in \textbf{N}$ such that $\mathcal{R}(\Phi_i)=f_i$ and let $\Phi \in \textbf{N}$ such that
%	\[
%	\mathcal{R}(\Phi) = \sum_{i=1}^{M} h_i f_i.
%	\]
%	Therefore
%	\[
%	\dim(\mathcal{D}(\Phi)) = H + 3,
%	\]
%	\[
%	\mathcal{P}(\Phi) \leq \sum_{i=1}^{M} 2\left(\mathcal{P}(\Phi_i) + 2d(2d+1)(H-H_{i}) + d\right) + \sum_{i=1}^{M} \sum_{\underset{j \neq i}{j=1}}^{M} (2d + \vertiii{\mathcal{D}(\Phi_i)})(2d + \vertiii{\mathcal{D}(\Phi_j)}),
%	\]
%	and
%	\[
%	\mathcal{D}(\Phi) \leq 2dM + \sum_{i=1}^{M} \vertiii{\mathcal{D}(\Phi_i)}.
%	\]
%\end{rem}

The following Lemma comes from \cite{Elb} and is adapted to our notation.

\begin{lema}\label{lem:DNN_para}
	Let $H,d,d_i \in \N$, $\beta_i \in \textbf{D}$, $f_i \in C(\R^d,\R^{d_i})$, $i=1,2$ such that for $i=1,2$ $\dim(\beta_i) = H+2$ and $f_i \in \mathcal{R}\left(\left\{ \Phi \in \textbf{N}: \mathcal{D}(\Phi) = \beta_i\right\}\right)$. Then
	\begin{equation}
		(f_1,f_2) \in \mathcal{R}\left(\left\{\Phi \in \textbf{N}: \mathcal{D}(\Phi) = (d,\beta_{1,1}+\beta_{2,1},...,\beta_{1,H+1}+\beta_{2,H+1}) \right\}\right).
	\end{equation}
\end{lema}
\begin{rem}
	Let $\Phi_1, \Phi_2, \Phi \in \textbf{N}$ such that $\mathcal{R}(\Phi_i) = f_i$, $i=1,2$ and $\mathcal{R}(\Phi) = (f_1,f_2)$. Notice by Lemma \ref{lem:DNN_para} and definition of the norm $\vertiii{\cdot}$ in \eqref{norma} that
	\[
	\vertiii{\mathcal{D}(\Phi)} \leq \vertiii{\mathcal{D}(\Phi_1)} + \vertiii{\mathcal{D}(\Phi_2)}.
	\]
\end{rem}

For sake of completeness, we state the following lemma with his proof. We continue with the notation from \cite{Hutz}:

\begin{lema}\label{lem:DNN_mat}
	Let $H,p,q,r \in \N$, $M \in \R^{r\times q}$, $\alpha \in \textbf{D}$, $f \in C(\R^p,\R^q)$, such that $dim(\alpha)=H+2$ and $f \in \mathcal{R}(\{\Phi \in \textbf{N}: \mathcal{D}(\Phi)=\alpha\})$. Then
	\begin{equation}
		Mf \in \mathcal{R}\left(\left\{\Phi \in \textbf{N}: \mathcal{D}(\Phi)=(\alpha_0,...,\alpha_H,r)\right\}\right).
	\end{equation}
\end{lema}
\begin{proof}
	Let $H, \alpha_0, ..., \alpha_{H+1} \in \N$, $\Phi_f \in \textbf{N}$ satisfying that
	\[
	\alpha = \left(\alpha_0,...,\alpha_{H+1}\right), \qquad \mathcal{R}(\Phi_f) = f, \qquad \hbox{and} \qquad \mathcal{D}(\Phi_f) = \alpha.
	\]
	Note that $p = \alpha_0$ and $q = \alpha_{H+1}$. Let $\left((W_1,B_1),...,(W_{H+1},B_{H+1})\right) \in \prod_{n=1}^{H+1} \left(\R^{\alpha_n \times \alpha_{n-1}} \times \R^{\alpha_n}\right)$ satisfty that
	\[
	\Phi_f = \left((W_1,B_1),...,(W_{H+1},B_{H+1})\right).
	\]
	Let $M \in \R^{r \times \alpha_{H+1}}$ and define
	\[
	\Phi = \left((W_1,B_1),...,(W_{H},B_{H}),(MW_{H+1},MB_{H+1})\right).
	\]
	Notice that $(MW_{H+1},MB_{H+1}) \in \R^{r \times \alpha_{H}} \times \R^{r}$, then $\Phi \in \N$. For $y_0 \in \R^{\alpha_0}$, and $y_i$, $i=1,...,H$ defined as in (NN4) we have
	\[
	\left(\mathcal{R}(\Phi)\right)(y_0) = MW_{H+1}y_{H} + MB_{H+1} = M \left(W_{H+1}y_{H}+B_{H+1}\right) = M \left(\mathcal{R}(\Phi_f)\right)(y_0).
	\]
	Therefore
	\[
	\mathcal{R}(\Phi) = Mf, \qquad \hbox{and} \qquad \mathcal{D}(\Phi) = \left(\alpha_0,...,\alpha_{H},r\right), 
	\]
	and the Lemma is proved.
\end{proof}
\begin{rem}
	Let $\Phi_f, \Phi \in \textbf{N}$ such that $\mathcal{R}(\Phi_f) = f$ and $\mathcal{R}(\Phi) = Mf$. From previous Lemma and the definition of $\vertiii{\cdot}$ it follows that
	\[
	\vertiii{\mathcal{D}(\Phi)} \leq \max\{\vertiii{\mathcal{D}(\Phi_f)},r\}.
	\]
\end{rem}

The following Lemma is from \cite{Grohs}.
\begin{lema}\label{lem:DNN_mult}
	There exists constants $C_1,C_2,C_3,C_4 > 0$ such that for all $\kappa > 0$ and for all $\delta \in \left(0,\frac{1}{2}\right)$ there exists a ReLu DNN $\Upsilon \in \textbf{N}$, with $\mathcal{R}(\Upsilon) \in C(\R^2,\R)$ such that
	\begin{equation}
		\sup_{a,b \in [-\kappa,\kappa]} |ab-\left(\mathcal{R}(\Upsilon)\right)(a,b)| \leq \delta.
	\end{equation}
	Moreover, for all $\delta \in \left(0,\frac{1}{2}\right)$ ,
	\begin{align}
		\mathcal{P}(\Upsilon) &\leq C_1 \left( \log_2 \left(\frac{\max\{\kappa,1\}}{\delta}\right)\right) + C_2, \\
		\dim(\mathcal{D}(\Upsilon)) &\leq C_3 \left( \log_2 \left(\frac{\max\{\kappa,1\}}{\delta}\right)\right) + C_4.
	\end{align}
\end{lema}

\medskip

\section{Walk-on-spheres Processes}\label{Sect:4}

We start with some key notation that will be extensively used along this paper. 

\medskip

Let $(\Omega, \PP, \mathcal{F})$ be a filtered probability space with $\mathcal{F}=\left(\mathcal{F}_t\right)_{t\geq 0}$. Let $\left(X_{t}\right)_{t \geq 0}$ be an isotropic $\alpha$-stable process starting at $X_0$. For $x \in D$ denote $\PP_x$ the probability measure conditional to $X_0 = x$ and $\E_x$ the respective expectation. Finally, define for any $B \subset \R^d$ the exit time for the set $B$ as 
\[
\sigma_B = \inf \{t \geq 0 : X_t \notin B\}.
\]
Now we introduce the classical WoS process.

\begin{defn}[\cite{AK1}]\label{defn:WoS}
The Walk-on-Spheres (WoS) process $\rho := (\rho_n)_{n\in \N}$ is defined as follows: 
\begin{itemize}
\item $\rho_0 = x$, $x \in D$;
\item given $\rho_{n-1}$, $n\geq 1$, the distribution of $\rho_n$ is chosen according to an independent sample of $X_{\sigma_{B_n}}$ under $\PP_{\rho_{n-1}}$, where $B_n$ is the ball centered on $\rho_{n-1}$ and radius $r_{n} = \dist(\rho_{n-1},\partial D)$. 
\end{itemize}
\end{defn}
\begin{rem}
	Notice by the Markov property that the process $\rho$ can be written as the recurrence
	\[
	\rho_n = \rho_{n-1} + Z_n, \quad n \in \N,
	\] 
	where $Z_n$ is an independent sample of $X_{\sigma_{B(0,r_n)}}$ under $\PP_0$.
\end{rem}
From the previous Remark, it is possible to rewrite $\rho_n$ for $n \in \N$ depending on $x \in D$ and on $n$ independent processes distributing accord $X_{\sigma_{B(0,1)}}$, as indicates the following Lemma.
\begin{lem}
	The WoS process $\rho := \left(\rho_n\right)_{n \in \N}$ can be defined as follows
	\begin{itemize}
		\item $\rho_0 = x$, $x \in D$;
		\item for $n \geq 1$,
		\begin{equation}\label{eq:paseo}
			\rho_n = \rho_{n-1} + r_n Y_n,
		\end{equation}
		where $Y_{n}$ is an independent sample of $X_{\sigma_{B(0,1)}}$ and $r_n = \dist(\rho_{n-1},\partial D)$.
	\end{itemize}
\end{lem}
\begin{proof}
Note by the scaling property \eqref{eq:scaling} that
\[
X_{t} \qquad \hbox{and} \qquad r_n X_{r_n^{-\alpha}t}
\]
have the same distribution for all $n \in \N$. Therefore
\begin{equation}\label{eq:scaling1}
\begin{aligned}
	\sigma_{B(0,r_n)} &= \inf \{t \geq 0 : X_{t} \notin B(0,r_n)\} \\
	&= r_n^{\alpha} \inf \{r_n^{-\alpha}t \geq 0: r_n X_{r^{-\alpha}_n t} \notin B(0,r_n)\}\\
	&= r_n^{\alpha} \inf \{s \geq 0: r_n X_s \notin B(0,r_n)\}\\
	&= r_n^{\alpha} \inf\{s \geq 0: X_s \notin B(0,1) \} = r_n^{\alpha} \sigma_{B(0,1)}.
\end{aligned}
\end{equation}
This equality and the scaling property implies that
\begin{equation}\label{eq:scaling2}
X_{\sigma_{B(0,1)}} \qquad \hbox{and} \qquad r_n^{-1}X_{r_n^{\alpha}\sigma_{B(0,1)}}
\end{equation}
are equal in law under $\PP_0$, and then from Remark 4.1 $Z_n$ and $r_n X_{\sigma_{B(0,1)}}$ have the same distribution under $\PP_0$. We can conclude that for $n\geq 1$, $\rho_n$ can be written as the recurrence
\[
\rho_n = \rho_{n-1} + r_n Y_n,
\]
where $Y_n$ is an independent sample of $X_{\sigma_{B(0,1)}}$.
\end{proof}
To study the WoS processes, we need to know about the processes $X_{\sigma_{B(0,1)}}$. The following result gives the distribution density of $X_{\sigma_{B(0,1)}}$. 
\begin{teo}[Blumenthal, Getoor, Ray, 1961. \cite{Blumenthal1}]\label{teo:Blu} Suppose that $B(0,1)$ is a unit ball centered at the origin and write $\sigma_{B(0,1)} = \inf \{ t > 0 : X_t \notin B(0,1)\}$. Then,
	\begin{equation}
		\PP_0\left(X_{\sigma_{B(0,1)}} \in dy\right) = \pi^{-(d/2+1)} \Gamma \left(\frac d2 \right) \sin(\pi\alpha/2) \left|1-|y|^{2}\right|^{-\alpha/2} |y|^{-d} dy, \qquad |y| > 1.
	\end{equation}
\end{teo}

Using this result, one can prove a key result for the expectation of $X_{\sigma_{B(0,1)}}$ moments.

\begin{cor}\label{cor:Kab}
	For all $\alpha \in (0,2)$, $\beta \in [0,\alpha)$ we have
	\begin{equation}\label{eq:Kab}
	\E_0 \left[\left|X_{\sigma_{B(0,1)}}\right|^{\beta}\right] = \frac{\sin(\pi \alpha /2)}{\pi} \frac{\Gamma\left(1- \frac{\alpha}{2}\right)\Gamma\left(\frac{\alpha-\beta}{2}\right)}{\Gamma\left(1-\frac{\beta}{2}\right)} =: K(\alpha,\beta).
	\end{equation}
\end{cor}

\begin{rem}
	Notice that the value of $K(\alpha,\beta)$ does not depend of the dimension $d$.
\end{rem}
\begin{rem}
	the condition $\beta < \alpha $ is necessary due to Corollary \ref{cor:Xmoment}. If $\beta \geq \alpha$ then $\E[|X_t|^{\beta}] = \infty$ for all $t>0$. Moreover, the integral
	\[
	\int_1^{\infty} \frac{r^{\beta-1}}{(r^2-1)^{\alpha/2}}dr,
	\]
	obtained in the proof of the Corollary \ref{cor:Kab} does not converges if $\beta \geq \alpha$.
\end{rem}

\begin{proof}[Proof of Corollary \ref{cor:Kab}]
	Let $\alpha \in (0,2)$, $\beta \in [0,\alpha)$. Notice by Theorem \ref{teo:Blu} and definition of the expectation that
	\begin{equation}\label{eq:E0X}
	\begin{aligned}
	\E_0 \left[\left|X_{\sigma_{B(0,1)}}\right|^{\beta}\right] &= \int_{|y|>1} |y|^{\beta} \PP_0 \left(X_{\sigma_{B(0,1)}} \in dy\right) \\
	&= \pi^{-(d/2+1)} \Gamma \left(\frac d2 \right) \sin(\pi\alpha/2) \int_{|y|>1} \left|1-|y|^{2}\right|^{-\alpha/2} |y|^{\beta-d} dy.
	\end{aligned}
	\end{equation}
	Using spherical coordinates one has
	\[
	\int_{|y|>1} \left|1-|y|^{2}\right|^{-\alpha/2} |y|^{\beta-d} dy = \int_{\mathbb{S}^{d-1}} \int_{1}^{\infty} \left|1-r^2\right|^{-\alpha/2} r^{\beta-d} r^{d-1} dr dS,
	\]
	where $\mathbb{S}^{d-1}$ is the surface area of the unit $(d-1)$-sphere embedded in dimension $d$. One has that \cite{surface}
	\[
	\left|\mathbb{S}^{d-1}\right| = \frac{2\pi^{d/2}}{\Gamma \left(\frac{d}{2}\right)},
	\]
	and then
	\[
	\int_{|y|>1} \left|1-|y|^{2}\right|^{-\alpha/2} |y|^{\beta-d} dy = \frac{2\pi^{d/2}}{\Gamma \left(\frac{d}{2}\right)} \int_1^{\infty} \frac{r^{\beta-1}}{(r^2-1)^{\alpha/2}} dr.
	\]
	Replacing this result into \eqref{eq:E0X} give us that
	\[
	\E_0 \left[\left|X_{\sigma_{B(0,1)}}\right|^{\beta}\right] = \frac{2}{\pi} \sin(\pi \alpha/2)\int_1^{\infty} \frac{r^{\beta-1}}{(r^2-1)^{\alpha/2}}dr.
	\]
	Now we are able to use a change of variables $u = 1/r$, then
	\[
	\int_1^{\infty} \frac{r^{\beta-1}}{(r^2-1)^{\alpha/2}}dr = \int_{0}^{1} \frac{1}{u^2} u^{1-\beta} \frac{u^{\alpha}}{(1-u^2)^{-\alpha/2}} du = \int_0^1 u^{\alpha-\beta-1}(1-u^2)^{-\alpha/2}du.
	\]
	Using another change of variable, $t = u^2$ we have
	\[
	\int_0^1 \frac{1}{2t^{\frac{1}{2}}} t^{\frac{\alpha-\beta-1}{2}}(1-t)^{-\frac{\alpha}{2}}dt = \frac{1}{2}\int_0^1 t^{\frac{\alpha-\beta}{2}-1}(1-t)^{1-\frac{\alpha}{2}-1} dt.
	\]
	This result implies that
	\[
	\E_x\left[\left|X_{\sigma_{B(0,1)}}\right|^{\beta}\right] = \frac{\sin(\pi \alpha/2)}{\pi} \int_0^1 t^{\frac{\alpha-\beta}{2} - 1}(1-t)^{1-\frac{\alpha}{2} - 1}dt.
	\]
	The integral has the form of the \emph{Beta function}, formally defined as:
	\[
	B(z,w) := \int_0^1 u^{z-1}(1-u)^{w-1}du,
	\]
	For full details of the Beta function see, e.g. \cite{beta}. In particular, the Beta function satisfies
	\[
	B(z,w) = \frac{\Gamma(z)\Gamma(w)}{\Gamma(z+w)}.
	\]
	Finally
	\[
	\E_0 \left[\left|X_{\sigma_{B(0,1)}}\right|^{\beta}\right] = \frac{\sin(\pi \alpha/2)}{\pi} B\left(\frac{\alpha-\beta}{2},1-\frac{\alpha}{2}\right) = \frac{\sin(\pi\alpha/2)}{\pi}\frac{\Gamma\left(1-\frac{\alpha}{2}\right)\Gamma\left(\frac{\alpha-\beta}{2}\right)}{\Gamma\left(1-\frac{\beta}{2}\right)}.
	\]
\end{proof}

\subsection{Relation between WoS and isotropic $\alpha$-stable process}
Recall from Section \ref{Sect:4} that the process $(\rho_n)_{n\geq0}$ is related to a family of processes distributing accord $X_{\sigma_{B(0,1)}}$. We want now to have a relation between the processes $(X_t)_{t \geq0}$ and $(\rho_n)_{n\geq 0}$. For this define $\widetilde{r}_1 := \dist(x,\partial D)$, $\widetilde{B}_1 := B(x,r_1)$, $\tau_1 := \sigma_{\widetilde{B}_1}$ and for all $n \geq 1$ define:
%Notice that the processes $(X_t)_{t \geq 0}$ and $(\rho_n)_{n \geq 0}$ are related as follows. Define $\widetilde{r}_1 := \dist(x,\partial D)$, $\widetilde{B}_1 := B(x,r_1)$ and $\tau_1 := \sigma_{\widetilde{B}_1}$. For all $n \geq 1$ define:
\begin{align}
	\widetilde{r}_{n+1} &:= \hbox{dist}(X_{\mathcal{I}(n)},\partial D), \\
	\widetilde{B}_{n+1} &:=B\left(X_{\mathcal{I}(n)},\widetilde{r}_{n+1}\right), \\
	\tau_{n+1} &:= \inf\{t \geq 0 : X_{t + \mathcal{I}(n)} \notin \widetilde{B}_{n+1}\}, \label{eq:tau_n}
\end{align}
where 
\begin{equation}\label{eq:I_n}
 \mathcal{I}(n) := \sum_{i=1}^{n} \tau_i, \qquad \mathcal{I}(0) = 0.
\end{equation}
$\mathcal{I}(n)$ represents the total time of the process $X_t$ takes to exit the $n$ balls $\widetilde{B}_{1},...,\widetilde{B}_{n}$. The following Lemma establishes that for all $n \in \N$, $\rho_n$ is equally distributed to the process $(X_t)_{t\geq0}$ exiting the $n$ balls $\widetilde{B}_{1},...,\widetilde{B}_{n}$, that is, $X_{\mathcal{I}(n)}$. 
%is a sort of \emph{total exit time}.

\begin{lem}
	For all $n \geq 0$ and $x \in D$, $\rho_n$ and $X_{\mathcal{I}(n)}$ have the same distribution starting at $x$.
\end{lem}
\begin{proof}
 Note that under $\PP_{X_{\mathcal{I}(n-1)}}$, $X_{\mathcal{I}(n)}$ has the same distribution as $X_{\sigma_{\widetilde{B}_n}}$. Thus, by the Markov property and the scaling property one has
\[
X_{\mathcal{I}(n)} = X_{\mathcal{I}(n-1)} + \widetilde{r}_{n} Y'_n,
\]
where $Y'_n$ is an independent sample of $X_{\sigma_{B(0,1)}}$ under $\PP_0$. From the two constructions below in addition with induction, one has that $\rho_n$ and $X_{\mathcal{I}(n)}$ has the same distribution starting at $x$, for all $n \geq 0$ and $x \in D$.
\end{proof}
Let
\begin{equation}\label{eq:N}
N = \min \{n \in \N : \rho_n \notin D\}.
\end{equation}
This random variable describes the quantity of balls $\widetilde{B}_{n}$ that the process $(X_t)_{t\geq0}$ exits before exits the domain $D$. The following Theorem ensures that $N$ is almost surely finite.
\begin{teo}[ \cite{AK1}, Theorem 5.4] \label{teo:N}
	Let $D$ be a open and bounded set. Therefore for all $x \in D$, there exists a constant $\widetilde{p} = \widetilde{p}(\alpha,d) > 0$ independent of $x$ and $D$, and a random variable $\Gamma$ such that $N \leq \Gamma$ $\PP_x$-a.s., where
	\begin{equation}
		\PP_x(\Gamma = k) = (1-\widetilde{p})^{k-1}\widetilde{p}, \hspace{.5cm} k \in \N.
	\end{equation}
\end{teo}
\begin{rem}
	Although the random variable $\Gamma$ has the same distribution for each $x \in D$, it is not the same random variable for each $x \in D$. 
\end{rem}
\begin{rem}
	This theorem implies that
	\[
	\PP_x(N>n) \leq \PP_x(\Gamma > n) = (1-\widetilde{p})^n, \qquad n \in \N.
	\]
\end{rem}
The definition of $\mathcal{I}(n)$ and $N$ in \eqref{eq:I_n} and \eqref{eq:N} imply that the total time of $(X_t)_{t \geq 0}$ that takes to exit $N$ balls $\widetilde{B}_{1},...,\widetilde{B}_{N}$ is equal to the time of $(X_t)_{t\geq0}$ that takes to exit $D$. More precisely 
\begin{lem}\label{lem:INsigmaD}
	For $x \in D$, let $X_{t}$ be an isotropic $\alpha$-stable process. Therefore, a.s.
	\[
	\mathcal{I}(N) = \sigma_D.
	\]
\end{lem}
\begin{proof}
	For the inequality $\geq$, note by definition of $N$ that
	\[
	X_{\mathcal{I}(N)} \notin D.
	\]
	Recall that $\sigma_D$ is the infimum time $t\geq 0$ such that $X_t \notin D$, then
	\[
	\mathcal{I}(N) \geq \sigma_D.
	\]
	For $\leq$ suppose by contradiction that $\sigma_D < \mathcal{I}(N)$. If $\sigma_D < \mathcal{I}(N-1)$, then
	\[
	X_{\mathcal{I}(N-1)} \notin D.
	\]
	This is a contradiction with the definition of $N$, because $N-1$ is a natural less than $N$ satisfying the above condition. Therefore $\mathcal{I}(N-1) \leq \sigma_D$ and this implies that there exists $t^* \geq 0$ such that 
	\[
	\mathcal{I}(N)>\sigma_D = \mathcal{I}(N-1) + t^*,
	\] 
	Using the definition of $\mathcal{I}(n)$, for $n \in \N$ and the supposition $\sigma_D < \mathcal{I}(N)$, one has
	\[
	t^{*} < \tau_N,
	\]
	but
	\[
	X_{\sigma_D} = X_{\mathcal{I}(N-1) + t^*} \notin D,
	\]
	therefore, from the definition of $\tau_N$ in \eqref{eq:tau_n},
	\[
	\tau_N \leq t^*,
	\]
	a contradiction. Therefore $\mathcal{I}(N-1) \leq \sigma_D$ and we can conclude that
	\[
	\mathcal{I}(N) = \sigma_D.
	\]
\end{proof}
\begin{rem}\label{rem:igualdades}
	From the relation between $X_{\mathcal{I}(n)}$ and $\rho_n$ for $n \in \N$, it follows that
	\[
	\E_x \left[\rho_{N}\right] = \E_x \left[X_{\mathcal{I}(N)}\right] = \E_x \left[X_{\sigma_D}\right].
	\]
\end{rem}
Remark \ref{rem:igualdades} give us a relation between $(X_{t})_{t\geq0}$ and $(\rho_{n})_{n\geq0}$. Figure \ref{fig:1} shows in an example the relation between WoS and isotropic $\alpha$-stable processes, exiting a bounded domain $D$.
\begin{figure}[h]\label{fig:1}

	\tikzset{every picture/.style={line width=0.75pt}} %set default line width to 0.75pt        
	
	\begin{tikzpicture}[x=0.75pt,y=0.75pt,yscale=-1.2,xscale=1.2]
		%uncomment if require: \path (0,227); %set diagram left start at 0, and has height of 227
		
		%Shape: Polygon Curved [id:ds5481216897626047] 
		\draw   (195.24,55.73) .. controls (217.25,95.11) and (227.6,75.21) .. (240.42,68.57) .. controls (253.24,61.92) and (289.55,71.86) .. (274.05,137.07) .. controls (258.54,202.28) and (185.1,157.73) .. (110.07,205.3) .. controls (35.05,252.88) and (19.82,205.49) .. (46.3,117.58) .. controls (72.78,29.67) and (103.97,1.75) .. (133.95,3.47) .. controls (163.94,5.2) and (173.23,16.35) .. (195.24,55.73) -- cycle ;
		%Shape: Donut [id:dp9475561166128219] 
		\draw  [color={rgb, 255:red, 255; green, 0; blue, 0 }  ,draw opacity=1 ] (180.02,124.21) .. controls (180.37,124) and (180.82,124.11) .. (181.03,124.46) .. controls (181.24,124.81) and (181.13,125.26) .. (180.78,125.47) .. controls (180.44,125.68) and (179.98,125.57) .. (179.77,125.22) .. controls (179.56,124.87) and (179.67,124.42) .. (180.02,124.21)(151.48,77.1) .. controls (177.85,61.13) and (212.17,69.55) .. (228.14,95.92) .. controls (244.12,122.28) and (235.69,156.61) .. (209.33,172.58) .. controls (182.96,188.56) and (148.64,180.13) .. (132.66,153.77) .. controls (116.69,127.4) and (125.11,93.08) .. (151.48,77.1) ;
		%Straight Lines [id:da08867971992222445] 
		\draw [color={rgb, 255:red, 74; green, 144; blue, 226 }  ,draw opacity=1 ]   (106,93.8) -- (95.75,95.1) -- (98.7,97.12) -- (85.84,101.38) -- (87.28,84.36) -- (91.13,103.59) -- (107.97,107.36) -- (99.46,114.97) -- (108.95,115.04) -- (115.34,102.31) -- (121.17,103.62) -- (93,123.15) -- (103.15,123.38) -- (95.98,129.6) -- (112.65,123.47) -- (107.16,121.01) -- (117.44,132.83) -- (122.75,157.29) -- (121.71,160.44) -- (125.11,150.19) -- (107.45,142.5) -- (117.01,151.44) -- (111.79,151.97) -- (125.71,180.92) -- (126.41,167.14) -- (138.39,172.88) -- (135.05,148.33) -- (181.74,125.66) -- (173.3,103.52) -- (164.6,142.37) -- (150.83,112.77) -- (136.51,102.42) -- (159.9,81.48) -- (159.78,86.55) -- (171.83,88.91) -- (182.93,81.87) -- (198.61,84.71) -- (181.32,95.11) -- (185.65,110.7) -- (198.11,106.69) -- (224.57,122.14) -- (227.9,136.14) -- (239.01,137.97) -- (255.25,158.25) -- (295.67,145) ;
		%Shape: Donut [id:dp9818874567558531] 
		\draw  [color={rgb, 255:red, 255; green, 0; blue, 0 }  ,draw opacity=1 ] (237.53,138.27) .. controls (237.78,138.11) and (238.11,138.2) .. (238.27,138.45) .. controls (238.42,138.71) and (238.34,139.04) .. (238.09,139.19) .. controls (237.83,139.34) and (237.5,139.26) .. (237.35,139.01) .. controls (237.19,138.75) and (237.27,138.42) .. (237.53,138.27)(220.67,110.45) .. controls (236.29,100.98) and (256.63,105.97) .. (266.09,121.59) .. controls (275.55,137.21) and (270.56,157.55) .. (254.94,167.01) .. controls (239.32,176.48) and (218.99,171.49) .. (209.52,155.87) .. controls (200.06,140.25) and (205.05,119.91) .. (220.67,110.45) ;
		%Shape: Donut [id:dp4501049258896639] 
		\draw  [color={rgb, 255:red, 255; green, 0; blue, 0 }  ,draw opacity=1 ] (122.62,159.2) .. controls (122.88,159.04) and (123.22,159.12) .. (123.38,159.39) .. controls (123.54,159.65) and (123.46,159.99) .. (123.2,160.15) .. controls (122.93,160.31) and (122.59,160.22) .. (122.43,159.96) .. controls (122.27,159.7) and (122.36,159.36) .. (122.62,159.2)(105.24,130.52) .. controls (121.35,120.76) and (142.31,125.91) .. (152.07,142.01) .. controls (161.82,158.11) and (156.68,179.08) .. (140.57,188.83) .. controls (124.47,198.59) and (103.51,193.44) .. (93.75,177.34) .. controls (84,161.24) and (89.14,140.27) .. (105.24,130.52) ;
		%Shape: Donut [id:dp03682612098944227] 
		\draw  [color={rgb, 255:red, 255; green, 0; blue, 0 }  ,draw opacity=1 ] (105.2,91.8) .. controls (105.56,91.58) and (106.04,91.7) .. (106.26,92.07) .. controls (106.49,92.43) and (106.37,92.91) .. (106,93.14) .. controls (105.63,93.36) and (105.16,93.24) .. (104.93,92.87) .. controls (104.71,92.51) and (104.83,92.03) .. (105.2,91.8)(80.85,51.62) .. controls (103.41,37.95) and (132.78,45.16) .. (146.45,67.72) .. controls (160.12,90.28) and (152.91,119.65) .. (130.35,133.32) .. controls (107.79,146.99) and (78.42,139.78) .. (64.75,117.22) .. controls (51.08,94.66) and (58.29,65.29) .. (80.85,51.62) ;
		%Shape: Ellipse [id:dp7989680426515668] 
		\draw  [color={rgb, 255:red, 245; green, 166; blue, 35 }  ,draw opacity=1 ][fill={rgb, 255:red, 245; green, 166; blue, 35 }  ,fill opacity=1 ] (104.48,90.73) .. controls (105.14,90.31) and (106.02,90.51) .. (106.44,91.17) .. controls (106.86,91.83) and (106.67,92.71) .. (106,93.14) .. controls (105.34,93.56) and (104.46,93.36) .. (104.04,92.7) .. controls (103.62,92.03) and (103.81,91.15) .. (104.48,90.73) -- cycle ;
		%Shape: Circle [id:dp1993839698533686] 
		\draw  [color={rgb, 255:red, 245; green, 166; blue, 35 }  ,draw opacity=1 ][fill={rgb, 255:red, 245; green, 166; blue, 35 }  ,fill opacity=1 ] (122.15,158.47) .. controls (122.81,158.05) and (123.69,158.25) .. (124.11,158.91) .. controls (124.53,159.58) and (124.33,160.46) .. (123.67,160.88) .. controls (123.01,161.3) and (122.13,161.1) .. (121.71,160.44) .. controls (121.29,159.77) and (121.48,158.89) .. (122.15,158.47) -- cycle ;
		%Shape: Ellipse [id:dp005295872150943293] 
		\draw  [color={rgb, 255:red, 245; green, 166; blue, 35 }  ,draw opacity=1 ][fill={rgb, 255:red, 245; green, 166; blue, 35 }  ,fill opacity=1 ] (180.21,122.59) .. controls (180.88,122.17) and (181.76,122.37) .. (182.18,123.03) .. controls (182.6,123.69) and (182.4,124.57) .. (181.74,125) .. controls (181.07,125.42) and (180.2,125.22) .. (179.77,124.56) .. controls (179.35,123.89) and (179.55,123.01) .. (180.21,122.59) -- cycle ;
		%Shape: Ellipse [id:dp21665311099661144] 
		\draw  [color={rgb, 255:red, 245; green, 166; blue, 35 }  ,draw opacity=1 ][fill={rgb, 255:red, 245; green, 166; blue, 35 }  ,fill opacity=1 ] (237.04,136.86) .. controls (237.71,136.44) and (238.59,136.64) .. (239.01,137.3) .. controls (239.43,137.96) and (239.23,138.84) .. (238.57,139.27) .. controls (237.91,139.69) and (237.03,139.49) .. (236.6,138.83) .. controls (236.18,138.16) and (236.38,137.28) .. (237.04,136.86) -- cycle ;
		%Shape: Ellipse [id:dp3612050313284969] 
		\draw  [color={rgb, 255:red, 245; green, 166; blue, 35 }  ,draw opacity=1 ][fill={rgb, 255:red, 245; green, 166; blue, 35 }  ,fill opacity=1 ] (104.48,90.73) .. controls (105.36,90.17) and (106.52,90.43) .. (107.08,91.31) .. controls (107.64,92.19) and (107.37,93.36) .. (106.5,93.91) .. controls (105.62,94.47) and (104.45,94.21) .. (103.89,93.33) .. controls (103.34,92.45) and (103.6,91.29) .. (104.48,90.73) -- cycle ;
		%Shape: Ellipse [id:dp5607662073144901] 
		\draw  [color={rgb, 255:red, 245; green, 166; blue, 35 }  ,draw opacity=1 ][fill={rgb, 255:red, 245; green, 166; blue, 35 }  ,fill opacity=1 ] (121.42,158.37) .. controls (122.3,157.81) and (123.47,158.07) .. (124.03,158.95) .. controls (124.58,159.83) and (124.32,161) .. (123.44,161.55) .. controls (122.56,162.11) and (121.4,161.85) .. (120.84,160.97) .. controls (120.28,160.09) and (120.55,158.93) .. (121.42,158.37) -- cycle ;
		%Shape: Ellipse [id:dp27587146563448006] 
		\draw  [color={rgb, 255:red, 245; green, 166; blue, 35 }  ,draw opacity=1 ][fill={rgb, 255:red, 245; green, 166; blue, 35 }  ,fill opacity=1 ] (180.21,122.59) .. controls (181.09,122.03) and (182.26,122.29) .. (182.82,123.17) .. controls (183.37,124.05) and (183.11,125.22) .. (182.23,125.77) .. controls (181.35,126.33) and (180.19,126.07) .. (179.63,125.19) .. controls (179.07,124.31) and (179.34,123.15) .. (180.21,122.59) -- cycle ;
		%Shape: Ellipse [id:dp15569101962846066] 
		\draw  [color={rgb, 255:red, 245; green, 166; blue, 35 }  ,draw opacity=1 ][fill={rgb, 255:red, 245; green, 166; blue, 35 }  ,fill opacity=1 ] (237.08,136.93) .. controls (237.96,136.37) and (239.12,136.64) .. (239.68,137.51) .. controls (240.24,138.39) and (239.97,139.56) .. (239.1,140.12) .. controls (238.22,140.67) and (237.05,140.41) .. (236.49,139.53) .. controls (235.94,138.65) and (236.2,137.49) .. (237.08,136.93) -- cycle ;
		%Shape: Ellipse [id:dp18084601269040057] 
		\draw  [color={rgb, 255:red, 245; green, 166; blue, 35 }  ,draw opacity=1 ][fill={rgb, 255:red, 245; green, 166; blue, 35 }  ,fill opacity=1 ] (294.66,143.41) .. controls (295.54,142.85) and (296.7,143.11) .. (297.26,143.99) .. controls (297.82,144.87) and (297.56,146.03) .. (296.68,146.59) .. controls (295.8,147.15) and (294.63,146.89) .. (294.07,146.01) .. controls (293.52,145.13) and (293.78,143.97) .. (294.66,143.41) -- cycle ;
		
		% Text Node
		\draw (99.83,159.69) node [anchor=north west][inner sep=0.75pt]  [font=\scriptsize]  {$\rho _{1}$};
		% Text Node
		\draw (184.82,126.57) node [anchor=north west][inner sep=0.75pt]  [font=\scriptsize]  {$\rho _{2}$};
		% Text Node
		\draw (240.91,122.41) node [anchor=north west][inner sep=0.75pt]  [font=\scriptsize]  {$\rho _{3}$};
		% Text Node
		\draw (301.85,144.45) node [anchor=north west][inner sep=0.75pt]  [font=\scriptsize]  {$\rho _{4}$};
		% Text Node
		\draw (99.17,73.96) node [anchor=north west][inner sep=0.75pt]  [font=\scriptsize]  {$x$};
		% Text Node
		\draw (50.46,158.09) node [anchor=north west][inner sep=0.75pt]  [font=\scriptsize]  {$D$};

	\end{tikzpicture}
\caption{Illustration of isotropic $\alpha$-stable and WoS processes starting at $x$ exiting a domain $D$. The blue line represents the $\alpha$-stable process $(X_t)_{t\geq0}$, the orange dots are the WoS process $(\rho_n)_{n\geq 0}$ and the red balls are given by Definition \ref{defn:WoS}. In this case $N=4$ and $\rho_4 = X_{\sigma_D}$.}
\end{figure}
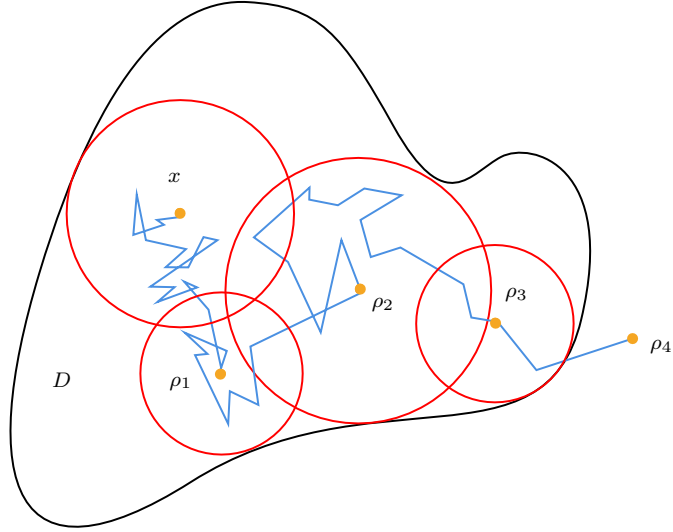
\section{Stochastic representation of the Fractional Laplacian}\label{Sect:5}

\subsection{Stochastic Representation}
Recall Problem \eqref{eq:1.1}. The following theorem gives an stochastic representation of the solution of problem \eqref{eq:1.1} from the process $(X_t)_{t\geq 0}$. The proof of this Theorem can be found in \cite{AK1}
\begin{teo}[\cite{AK1}, Theorem 6.1]\label{teo:sol}
	Let $d \geq 2$ and assume that $D$ is a bounded domain in $\R^d$. Additionally, assume \eqref{Hg0} and \eqref{Hf0}. Then there exist a {\bf unique continuous solution} for \eqref{eq:1.1} in $L^1_\alpha(\R^d)$, given by the explicit formula
	\begin{equation}\label{hipotesis_imp_2}
		u(x) = \E_x \left[ g(X_{\sigma_D}) \right] + \E_x \left[ \int_0^{\sigma_D} f(X_s) ds \right],
	\end{equation}
	valid for any $x \in D$.
\end{teo}

\medskip

The previous representation can be expressed in terms of the WoS process. For this define the expected occupation measure of the stable process prior to exiting a ball of radius $r>0$ centered in $x \in \R^d$ as follows:  
\begin{equation}\label{eq:V_r(x)}
	V_r(x,dy) := \int_0^{\infty} \PP_x\left(X_t \in dy, t < \sigma_{B(x,r)}\right) dt, \qquad x \in \R^d, \quad |y|<1, \quad r > 0.
\end{equation}
We have the following result for $V_1(0,dy)$.
\begin{teo}[\cite{AK1}, Theorem 6.2]\label{teo:V0dy}
	The measure $V_1(0,dy)$ is given for $|y|<1$, by 
	\begin{equation}\label{eq:V_1(0)}
		V_1(0,dy) = 2^{-\alpha} \pi^{-d/2} \frac{\Gamma(d/2)}{\Gamma(\alpha/2)^2}|y|^{\alpha-d}\left( \int_0^{|y|^{-2}-1} (u+1)^{-d/2} u^{\alpha/2-1}du \right)dy.
	\end{equation}
\end{teo}
Denote $\displaystyle V_r(x,f(\cdot)) = \int_{|y-x|<r}f(y) V_r(x,dy)$ for a bounded measurable function $f$. $V_r(x,f(\cdot))$ defines the expected value of $f$ under the measure $V_{r}(x,dy)$ over the ball $B(x,r)$. An important property of this expected value is the following: for $r>0$ and $x \in \R^d$
\begin{equation}\label{eq:V_prop}
	V_r(x,f(\cdot)) = V_{1}(0,f(x+r\cdot)).
\end{equation}

The proof of this property can be found in \cite{AK1}. The following Lemma uses the WoS process and the Theorem \ref{teo:V0dy}. Recall that $(\rho_n)_{n=1,...,N}$ represents the WoS process defined in Section \ref{Sect:4}, and $r_n = \dist(\rho_n,\partial D)$.
\begin{lema} [\cite{AK1}, Lema 6.3] \label{lem:sol2}
	For $x \in D$, $g \in L_{\alpha}^1(D^c)$ and $f \in C^{\alpha + \varepsilon_0}(\overline{D})$ we have the representation
	\begin{equation}
		u(x) = \E_{x} \left[ g(\rho_N) \right] + \E_{x}\left[\sum_{n=1}^{N} r_n^{\alpha} V_1(0,f(\rho_{n-1} +r_n\cdot))\right].
	\end{equation}
\end{lema}
\begin{rem}\label{rem:WOSf}
	Recall that $\rho_n$ and $X_{\mathcal{I}(n)}$ are equal on law under $\PP_x$ for all $n \in \N$. Therefore we can write
	\begin{equation}\label{eq:solWOS}
		u(x) = \E_{x} \left[ g\left(X_{\mathcal{I}(N)}\right) \right] + \E_{x}\left[\sum_{n=1}^{N} r_n^{\alpha} V_1\left(0,f\left(X_{\mathcal{I}(n-1)} +r_n\cdot\right)\right)\right].
	\end{equation}
\end{rem}

%\begin{rem} \label{rem:V10}
%	From the proof of this Lemma, it follows that  for any $x \in \R^d$, $r >0$ and bounded measurable $f$ (see \cite{AK1} for more details)
%	\[
%	V_r(x,f(\cdot)) = r^{\alpha} V_1(0,f(x+r\cdot)).
%	\]
%	Moreover,
%	\[
%	\E_x \left[ \int_0^{\sigma_D} f(X_s) ds\right] = \E_x \left[ \sum_{n=1}^{N} r_n^{\alpha} V_1(0,f(X_{\mathcal{I}(n-1)}+r_n\cdot))\right].
%	\]
%\end{rem}
\subsection{Equivalent representations of non-homogeneous solution} Consider again the problem \eqref{eq:1.1}. Remember from Remark \ref{rem:WOSf} that its solution can be written as
\begin{equation*}
	u(x) = \E_{x} \left[ g\left(X_{\mathcal{I}(N)}\right) \right] + \E_{x}\left[\sum_{n=1}^{N} r_n^{\alpha} V_1(0,f(X_{\mathcal{I}(n-1)} +r_n\cdot))\right].
\end{equation*}
Notice also that from the definition of $V_1(0,f(\cdot))$, it can be expressed as the expectation of $f$ under the measure $V_1(0,dy)$ on $B(0,1)$. This measure is not necessarily a probability measure, so we are going to normalize the measure $V_1(0,dy)$. For this define for all $d \geq 2$, $d \in \N$ and $\alpha \in (0,2)$
\[
\kappa_{d,\alpha} = \int_{B(0,1)} V_1(0,dy).
\]
In the following Lemma we prove that $\kappa_{d,\alpha}$ is positive and finite

%For this, notice that if $f \equiv 1$ then
%\begin{align}
%	V_1(0,1(X_{\mathcal{I}(n-1)} + r_n \cdot)) &= \int_{B(0,1)} V_1(0,dy) \notag\\
%	&= c_{d,\alpha} \int_{B(0,1)} |y|^{\alpha-d} \left(\int_{0}^{|y|^{-2}-1} (u+1)^{-d/2}u^{\alpha/2-1}du \right)dy,
%\end{align}
%where $\displaystyle c_{d,\alpha} = 2^{-\alpha}\pi^{-d/2} \frac{\Gamma(d/2)}{\Gamma(\alpha/2)^2}$. Finally, define $\displaystyle \kappa_{d,\alpha} = \int_{B(0,1)} V_1(0,dy)$.

\begin{lem}
	For all $d\geq 2$ and $\alpha \in (0,2)$, we have that $0<\kappa_{d,\alpha} <+\infty$.
\end{lem}
\begin{proof}
	Notice first from Theorem \ref{teo:Blu} that
	\[
	\kappa_{d,\alpha} = \widetilde{c}_{d,\alpha} \int_{B(0,1)} |y|^{\alpha-d} \left(\int_{0}^{|y|^{-2}-1} (u+1)^{-d/2}u^{\alpha/2-1}du \right)dy,
	\]
	where $\displaystyle \widetilde{c}_{d,\alpha} = 2^{-\alpha}\pi^{-d/2} \frac{\Gamma(d/2)}{\Gamma(\alpha/2)^2}$. Now we work with the interior integral. With a change of variables $\displaystyle u = \frac{1-t}{t}$ and integral properties one has:
	\begin{align*}
		 &\int_1^{|y|^2} t^{d/2} \left(\frac{1-t}{t}\right)^{\alpha/2-1}(-t^{-2})dt\\
		 &\qquad = \int_0^1 t^{d/2-\alpha/2-1}(1-t)^{\alpha/2-1}dt-\int_0^{|y|^2}t^{d/2-\alpha/2-1}(1-t)^{\alpha/2-1}dt.
	\end{align*}
	For $z,w > 0$, $x \in [0,1]$ let $B(z,w)$ and $I(x; z,w)$ be the Beta and the Incomplete Beta functions respectively, defined as
	\begin{align*}
		B(z,w) &:= \int_0^1 u^{z-1}(1-u)^{w-1}du,\\
		I(x;z,w) &:= \frac{1}{B(z,w)} \int_0^{x} u^{z-1}(1-u)^{w-1} du.
	\end{align*}
	For further details of these functions the reader can consult \cite{beta}. Notice that $\kappa_{d,\alpha}$ can be written in terms of $B(z,w)$ and $I(x;z,w)$. Indeed
	\begin{equation*}
		\kappa_{d,\alpha} =\widetilde{c}_{d,\alpha} B \left(\frac d2-\frac\alpha2,\frac\alpha2 \right)\int_{B(0,1)} |y|^{\alpha-d}\left( 1-I \left(|y|^2;\frac d2-\frac\alpha2,\frac\alpha2 \right)\right)dy.
	\end{equation*}
	Note by property of Beta function that
	\[
	B\left(\frac{d}{2}-\frac{\alpha}{2},\frac{\alpha}{2}\right) = \frac{\Gamma\left(\frac{d}{2}-\frac{\alpha}{2}\right)\Gamma\left(\frac{\alpha}{2}\right)}{\Gamma\left(\frac{d}{2}\right)}.
	\]
	The Gamma function is well defined and positive on $(0,\infty)$. If $d>\alpha$ then
	\[
	0 < B\left(\frac{d}{2}-\frac{\alpha}{2},\frac{\alpha}{2}\right) < + \infty.
	\]
	On the other hand side, note by the definition of $I(x;z,w)$ that for $x < 1$,
	\[
	0 \leq I(x;z,w) < \frac{1}{B(z,w)} \int_0^1 u^{z-1} (1-u)^{w-1}du = 1,
	\]
	Then for all $|y|^2<1$,
	\[
	0 < 1 - I \left(|y|^2; \frac{d}{2}-\frac{\alpha}{2}, \frac{\alpha}{2}\right) \leq 1.
	\]
	Therefore in $\kappa_{d,\alpha}$ we are integrating the multiplication of two positive functions over a set of positive measure. This implies that
	\[
	0 < \kappa_{d,\alpha} \leq \widetilde{c}_{d,\alpha} B\left(\frac{d}{2}-\frac{\alpha}{2},\frac{\alpha}{2}\right) \int_{B(0,1)} |y|^{\alpha - d} dy.
	\]
	The above integral can be calculated using a change of variables in spherical coordinates, and his value is finite. Finally we conclude that
	\[
	0 < \kappa_{d,\alpha} < + \infty.
	\]
\end{proof}

Now we are able to define a probability measure $\mu$ on $B(0,1)$ given by
\[
\mu(dy) := \kappa_{d,\alpha}^{-1} V_1(0,dy).
\]
Therefore, for any bounded measurable function $f$ we have
\begin{align}
	V_1(0,f(X_{\mathcal{I}(n-1)}+r_n\cdot)) &= \int_{B(0,1)} f(X_{\mathcal{I}(n-1)}+r_n y) V_1(0,dy) \notag\\
	&= \kappa_{d,\alpha} \int_{B(0,1)} f(X_{\mathcal{I}(n-1)}+r_n y) \mu(dy) \notag\\
	&= \kappa_{d,\alpha} \E^{(\mu)}\left[f(X_{\mathcal{I}(n-1)}+r_n \cdot)\right].
\end{align}
where $\E^{(\mu)}$ correspond to the expectation over the probability measure $\mu$ on $B(0,1)$. With this representation, we can rewrite the solution of \eqref{eq:1.1} as
\begin{equation}\label{eq:solfinal}
	u(x) = \E_x \left[g\left(X_{\mathcal{I}(N)}\right)\right] + \E_x \left[ \sum_{n=1}^{N} r_{n}^\alpha \kappa_{d,\alpha} \E^{(\mu)}\left[f\left(X_{\mathcal{I}(n-1)}+r_n \cdot\right)\right] \right].
\end{equation}

From the construction of $\kappa_{d,\alpha}$, the following properties are valid

\begin{lem}\label{lem:aux_f}
	One has
	\begin{enumerate}
		\item \[
		\E_x \left[ \sum_{n=1}^{N} r_n^{\alpha} \kappa_{d,\alpha}\right] = \E_x \left[ \sigma_D \right],
		\]
		\item \[
		\E_x \left[ \left|\sum_{n=1}^{N} r_n^{\alpha} \kappa_{d,\alpha} \right|^2\right] \leq \E_x \left[ \sigma_D^2 \right].
		\]
	\end{enumerate}
\end{lem}

\begin{proof}~{}
	
	\begin{enumerate}
		\item Notice from the definition of $V_1(0,f(\cdot))$, with $f\equiv 1$ that
		\[
		\kappa_{d,\alpha} = V_1\left(0,1\left(X_{\mathcal{I}(n-1)}+r_n \cdot\right)\right).
		\]
		It follows from \eqref{eq:V_prop} that
		\[
		r_n^{\alpha}\kappa_{d,\alpha} = V_{r_n} \left(X_{\mathcal{I}(n-1)},1(\cdot)\right).
		\]
		Moreover
		\[
		\begin{aligned}
			\E_x\left[\sum_{n=1}^{N} r_n^{\alpha} \kappa_{d,\alpha}\right] = &~{} \E_x\left[\sum_{n=1}^{N} V_{r_n}\left(X_{\mathcal{I}(n-1)},1(\cdot)\right)\right] \\
			=&~{}  \E_x \left[ \int_0^{\sigma_D} 1(X_s) ds\right] = \E_x \left[\sigma_D \right].
		\end{aligned}
		\]
		\item From the definition of $V_r(x,f(\cdot))$ with $f \equiv 1$, it follows that
		\[
		V_{r_n}\left(X_{\mathcal{I}(n-1)},1(\cdot)\right) = \E_{X_{\mathcal{I}(n-1)}} \left[\int_{0}^{\sigma_{B(X_{\mathcal{I}(n-1)},r_n)}} 1 (X_t) dt\right] = \E_{X_{\mathcal{I}(n-1)}} \left[\sigma_{B(X_{\mathcal{I}(n-1)},r_n)}\right].
		\]
		By definition of $\tau_n$, $n \in \N$, \eqref{eq:tau_n} and Markov property one has
		\[
		\E_{X_{\mathcal{I}(n-1)}}\left[\sigma_{B(X_{\mathcal{I}(n-1)},r_n)}\right] = \E_0 \left[\tau_{n}\right].
		\]
		Then, Jensen inequality implies
		\[
		\E_x \left[\left|\sum_{n=1}^{N} r_n^{\alpha} \kappa_{d,\alpha}\right|^2\right] = \E_x \left[\left|\sum_{n=1}^{N} \E_0 \left[\tau_n\right]\right|^2\right] \leq \E_x \left[ \E_0 \left[\left|\sum_{n=1}^{N} \tau_n\right|^2\right]\right].
		\]
		Finally, by tower property
		\[
		\E_x \left[\left|\sum_{n=1}^{N} r_n^{\alpha} \kappa_{d,\alpha}\right|^2\right] \leq \E_x \left[\mathcal{I}(N)^2\right] = \E_x\left[\sigma_D^2\right].
		\]
	\end{enumerate}
\end{proof}

\medskip

\section{Approximation of solutions of the Fractional Dirichlet problem using DNNs: the boundary data case}\label{Sect:6}

As usual, Problem \eqref{eq:1.1} can be decomposed in two subproblems, that will be treated in a separate way. We first deal with the homogeneous case.
 
\subsection{Homogeneous Fractional Laplacian}\label{Sect:6p1} We consider \eqref{eq:1.1} with $f \equiv 0$, namely,
\begin{equation} \label{eq:2.1}
	\left\{ \begin{array}{rll}
		(-\Delta)^{\alpha/2}u(x)&= 0 & \hbox{ for } x \in D,\\
		u(x)&= g(x) & \hbox{ for } x \in D^c.
	\end{array} \right.
\end{equation}
Note that under \eqref{Hg0}, one has from \eqref{hipotesis_imp_2} %Notar que en este caso la solución estará dada por
\begin{equation} \label{eq:2.2}
	u(x) = \E_x \left[ g(X_{\sigma_D}) \right], \hspace{.5cm} x \in D.
\end{equation}
The main idea of this section is approximate the solution \eqref{eq:2.2} by a deep neural network with ReLu activation, with an accurateness $\varepsilon>0$. For this we going to assume that $g$ can be approximated by a ReLu DNN satisfying several hypotheses. These hypotheses are expressed in the following assumption.

\medskip

Recall that $\vertiii{\cdot}$ represents the maximum number of hidden layers dimensions introduced in \eqref{norma}, $\mathcal R$ is the realization of a DNN as in \eqref{Realization}, and $\mathcal D$ was introduced in \eqref{P_D}.

\begin{ass}\label{Sup:g}
	Let $d\geq 2$. Let $g:D^c\to \R$ satisfying \eqref{Hg0}. Let $\delta_g \in (0,1)$,  $a,b \geq 1$, $p \in (1,\alpha)$ and $B >0$. Then there exists a {\it ReLu DNN} $\Phi_g \in$ \textbf{N} with
	\begin{enumerate}
	\item $\mathcal R(\Phi_g):D^c \to \R$  is continuous, and
	\item The following are satisfied:
	\begin{align}
		|g(y)-\left(\mathcal{R}(\Phi_g)\right)(y)| &\leq\delta_g Bd^p(1+|y|)^{p}, \hspace{.5cm} \forall y \in D^c. \tag{Hg-1} \label{H1}\\
		|\left(\mathcal{R}(\Phi_g)\right)(y)| &\leq Bd^p(1+|y|)^{p}, \hspace{1.0cm} \forall y \in D^c. \tag{Hg-2} \label{H2}\\
		\vertiii{\mathcal{D}(\Phi_g)} &\leq B d^b \delta_g^{-a}, \tag{Hg-3} \label{H3}
	\end{align}
	\end{enumerate}
\end{ass}
%Para $q \in [1,\infty)$, y un espacio de probabilidad $(\Omega,\mathcal{F},\PP)$ se denotará
%\begin{equation}
%	\E^{(\PP)}\left[|\cdot|^q\right]^{\frac{1}{q}} := \left(\int_{\Omega} |y|^{q}\PP(dy)\right)^{\frac{1}{q}}
%\end{equation}
%Notar además que $\norm{\cdot}_{L^{q}(\Omega,\PP)} = \E^{(\PP)}\left[|\cdot|^q\right]^{\frac{1}{q}}$ y en particular $\norm{\cdot}_{L^{q}(\Omega,\PP_x)} = \E^{(\PP_x)}[|\cdot|^{q}]^{\frac{1}{q}} = \E_x[|\cdot|^q]^{\frac{1}{q}}$

\begin{rem}
We use the hypotheses presented in \cite{Hutz} for the approximation of function defined over non bounded sets.
\end{rem}

In addition to the previous assumptions, we will require \emph{structural properties} related to the domain $D$ itself.
\begin{ass}\label{Sup:D}
Let $\alpha\in (1,2)$, $a,b \geq 1$ and $B>0$. Suppose that $D$ bounded domain enjoys the following structure.
	\begin{enumerate}
		\item For any $\delta_{\dist} \in (0,1)$, the function $x \mapsto \dist(x,\partial D)$ can be approximated by a ReLu DNN $\Phi_{\dist} \in \textbf{N}$ such that 
		\[
		\sup_{x \in D} \left|\dist(x,\partial D) - \left(\mathcal{R}(\Phi_{\dist})\right)(x)\right| \leq \delta_{\dist}, \tag{HD-1} \label{HD-1} 
		\]
		and
		\[
		\vertiii{\mathcal{D}(\Phi_{\dist})} \leq Bd^b\lceil \log(\delta_{\dist}^{-1}) \rceil^{a}. \tag{HD-2} \label{HD-2}
		\]
		\item For all $\delta_\alpha \in (0,1)$ there exists a ReLu DNN $\Phi_{\alpha} \in \textbf{N}$ such that
		\begin{equation}
			\sup_{|x|\leq\diam(D)}\left|\left(\mathcal{R}(\Phi_{\alpha})\right)(x) -x^{\alpha}\right| \leq \delta_{\alpha}. \tag{HD-3} \label{HD-3}
		\end{equation}
		and
		\[
		\vertiii{\mathcal{D}(\Phi_{\alpha})} \leq Bd^b \delta_{\alpha}^{-a}. \tag{HD-4} \label{HD-4}
		\]
		Moreover, $\mathcal{R}(\Phi_{\alpha})$ is a $L_{\alpha}$-Lipschitz function, $L_{\alpha}>0$, for $|x|\leq \diam(D)$.
	\end{enumerate}
\end{ass} 

\begin{rem}
	Notice that Assumption \eqref{HD-3} is assured by Hornik's Theorem \cite{Hornik}. Also, (HD-2) may seem too demanding because of the log term, but actually this is the situation in the case of a ball, see \cite{Grohs}.
\end{rem}

In the next proposition prove the existence of a ReLu DNN such that the Dirichlet problem without source is well approximated.
 
\begin{prop}\label{Prop:homo}
	Let $\alpha \in (1,2)$, $L_g >0$ and
	\begin{equation}\label{condiciones}
	\hbox{ $p,s \in (1,\alpha)$ such that $s < \frac{\alpha}{p}$ \quad and \quad $q \in \left[s,\frac{\alpha}{p}\right)$. }
	\end{equation}
	Suppose that the function $g$ satisfies \eqref{Hg0} and Assumptions \ref{Sup:g}. Suppose additionally that $D$ satisfies Assumptions \ref{Sup:D}.
%	for any $\delta_{\dist} \in (0,1)$, the function $x \to \dist(x,\partial D)$ can be approximated by a ReLu DNN $\Phi_{\dist}$ such that
%	\begin{equation}\label{eqn:dist}
%	\sup_{x \in D} \left| \dist(x,\partial D) - \left(\mathcal{R}(\Phi_{\dist})\right)(x) \right| \leq \delta_{\dist}.
%	\end{equation}
	Then for all $\varepsilon \in (0,1)$ there exists a ReLu DNN $\Psi_{1,\varepsilon}$ that satisfies
	\begin{enumerate}
	\item Proximity in $L^q(D)$: %there exists $\delta_g$ small enough such that
	\begin{equation} \label{eq:2.3}
		\left(\int_{D} \left|\E_{x}\left[g\left(X_{\mathcal{I}(N)}\right)\right]-\left(\mathcal{R}(\Psi_{1,\varepsilon})\right)(x)\right|^q dx\right)^{\frac{1}{q}} \leq \varepsilon.
	\end{equation}
	\item Realization: $\mathcal{R}(\Psi_{1,\varepsilon})$ has the following structure: there exists $M \in \N$, $\overline{N}_{i} \in \N$, $Y_{i,n}$ i.i.d. copies of $X_{\sigma_{B(0,1)}}$ under $\PP_0$, for $i=1,...,M$, $n=1,...,\overline{N}_i$ such that for all $x \in D$,
	\begin{equation}
		\mathcal{R}(\Psi_{1,\varepsilon})(x) = \frac{1}{M} \sum_{i=1}^{M} \left(\mathcal{R}(\Phi_g) \circ \mathcal{R}(\Phi^{i}_{\overline{N}_i})\right)\left(x\right),
	\end{equation}
	where $\Phi^{i}_{\overline{N}_i}$ is a ReLu DNN approximating $X^i_{\mathcal{I}(\overline{N}_i)} = X^i_{\mathcal{I}(\overline{N}_i)}(x,Y_{i,1},...,Y_{i,\overline{N}_i})$. 
	\item Bounds: There exists $\widetilde{B}>0$ such that
	\begin{equation}
		\vertiii{\mathcal{D}(\Psi_{1,\varepsilon})} \leq \widetilde{B}|D|^{\frac{1}{q}\left(2a + ap+\frac{s}{s-1}(1+2a+ap)\right)} d^{b+2ap+2ap^2 + \frac{ps}{s-1}(1+2a + ap) }\varepsilon^{-a-\frac{s}{s-1}(1+2a+ap)}.
	\end{equation}
	\end{enumerate}
\end{prop}

%\begin{rem}
%\end{rem}
%
%\begin{rem}
%\end{rem}

\begin{rem}
The hypotheses \eqref{condiciones} are non empty if $\alpha\in (1,2)$. This requirement is standard in the literature devoted to the Fractional Laplacian, where some proofs are highly dependent on the cases $\alpha\in (0,1]$ versus $\alpha\in (1,2).$ %See \cite{XXXXX} for further examples.
\end{rem}
\begin{rem}
	The condition $\alpha \in (1,2)$ is very important. In particular, if $\alpha \in (0,1)$ the hyportheses \eqref{condiciones} are empty, and moreover, processes that we are working with not necessarily have finite expectation and there are not guarantee on the convergence of the ReLu DNNs. 
\end{rem}

\subsection{Proof of Proposition \ref{Prop:homo}: existence} The proof will be divided in several steps. As explained before, we follow the ideas in \cite{Grohs}, with several changes due to the nonlocal character of the treated equation.

\medskip 

Let $s,p$ and $q$ be as in \eqref{condiciones}.

\medskip

\noindent
{\bf Step 1. Preliminaries.} Let $(\rho_n)_{n \in \N}$ be a WoS process introduced in Definition \ref{defn:WoS} starting from $x \in D$. Let also $\mathcal I(N)$ and $N$ be defined as in \eqref{eq:I_n} and \eqref{eq:N}. Recall that from \eqref{hipotesis_imp_2},
\begin{equation}\label{igualdades}
	u(x) = \E_x[g(\rho_{N})] = \E_x[g(X_{\mathcal{I}(N)})] .
\end{equation}
From the construction of $X_{\mathcal{I}(N)}$, one has that $X_{\mathcal{I}(N)} \in D^c$ and it depends on $N$ i.i.d. copies of $X_{\sigma_{B(0,1)}}$, namely
\begin{equation}\label{eq:X_IN}
X_{\mathcal{I}(N)} = X_{\mathcal{I}(N)}(x,Y_1,...,Y_N),
\end{equation}
where each $Y_n$, $n=1,\ldots, N$, is an independent copy of $X_{\sigma_{B(0,1)}}$ under $\PP_0$.

\medskip

Let $M \in \N$. Consider $M$ copies of $X_{\mathcal{I}(N)}$ starting from $x \in D$, as described in \eqref{eq:X_IN}. We denote such copies as
\begin{equation}\label{eq:Y_n}
X^{i}_{\mathcal{I}(N_i)} = X^{i}_{\mathcal{I}(N_i)}(x,Y_{i,1},...,Y_{i,N_i}), 
\end{equation}
with $Y_{i,n}$, $i=1,...,M$, $n = 1,...,N_i$ i.i.d. copies of $X_{\sigma_{B(0,1)}}$ under $\PP_0$, and where each $N_i$ is an i.i.d. copy of $N$. Notice that for each copy, $N_i$ can be different (as a random variable).

\medskip

With this in mind, and following \cite{Grohs}, we introduce the \emph{Monte Carlo operator}
\begin{equation}\label{eq:E_M}
	E_{M}(x) := \frac{1}{M} \sum_{i=1}^{M} \left(\mathcal{R}(\Phi_g)\right)\left(X^{i}_{\mathcal{I}(N_i)}\right),
\end{equation}
%where $X^{i}_{\mathcal{I}(N_i)} = X^{i}_{\mathcal{I}(N_i)}(x,Y_{i,1},...,Y_{i,N_i})$, with $Y_{i,n}$, $i=1,...,M$, $n = 1,...,N_i$ i.i.d. copies of $X_{\sigma_{B(0,1)}}$ under $\PP_0$, 
where $\mathcal{R}(\Phi_g)$ denotes the realization as a continuous function of the DNN $\Phi_g \in \textbf{N}$ that approximates $g$ in Assumption \ref{Sup:g}. Notice that $E_{M}(x)$ may not be a DNN in the general case.

\medskip

Our main objective in the following steps is to obtain suitable bounds on the difference between the expectation of $g(X_{\mathcal I(N)})$ and $E_{M}(x)$, in a certain sense to be determined. Step 2 controls the difference between $\E_x\left[g\left(X_{\mathcal{I}(N)}\right)\right]$ and the intermediate term $\E_x\left[\left(\mathcal{R}(\Phi_g)\right)\left(X_{\mathcal{I}(N)}\right)\right]$. Notice that this last term is not necessarily a DNN, because of the quantity $X_{\mathcal{I}(N)}$.

\medskip

\noindent
{\bf Step 2.} Define
\[
J_1 := \left|\E_x\left[g\left(X_{\mathcal{I}(N)}\right)\right]-\E_x\left[\left(\mathcal{R}(\Phi_g)\right)\left(X_{\mathcal{I}(N)}\right)\right]\right|.
\]
Notice that by Jensen inequality and hypothesis \eqref{H1}  one has
\[
\begin{aligned}
J_1 \leq \E_x \left[\left|g\left(X_{\mathcal{I}(N)}\right)-\left(\mathcal{R}(\Phi_g)\right)\left(X_{\mathcal{I}(N)}\right)\right|\right]  \leq Bd^{p}\delta_g \E_x \left[\left(1+\left|X_{\mathcal{I}(N)}\right|\right)^{p}\right].
\end{aligned}
\]
Recall that we have an expression for $\E_0\left[\left|X_{\sigma_{B(0,1)}}\right|^{\beta}\right]$ with $\beta<\alpha$ from Corollary \ref{cor:Kab}. The idea is to find a bound for $\E_x\left[\left(1+\left|X_{\mathcal{I}(N)}\right|\right)^p\right]$ in terms of \eqref{eq:Kab}.

\medskip

Let $R>1$ be large enough to have $D \subset B^*:=B(x,R)$. The right hand side of the previous inequality is going to be separated in two terms: the case where $\sigma_D = \sigma_{B^{*}}$, and otherwise. Notice that $\sigma_D> \sigma_{B^*}$ is not possible. We obtain:
\[
\begin{aligned}
	J_1&\leq Bd^{p} \delta_g \Big( \E_x \left[\left( 1 + \left|X_{\mathcal{I}(N)}\right| \right)^{p} {\bf 1}_{\{ \sigma_D = \sigma_{B^*} \} }\right] + \E_x \left[\left( 1 + \left|X_{\mathcal{I}(N)}\right| \right)^{p} {\bf 1}_{\{\sigma_D < \sigma_{B^*} \}}\right] \Big).
\end{aligned}
\]
In the case of the equality, the processes $X_{\mathcal{I}(N)}$ and $X_{\sigma_{B^{*}}}$ are equal on law under $\PP_x$ from Lemma \ref{lem:INsigmaD} and Remark \ref{rem:igualdades}. Then the Markov property and the scaling property of the process (see \eqref{eq:scaling1} and \eqref{eq:scaling2}) can be used to get
\[
\E_x \left[\left( 1 + \left|X_{\mathcal{I}(N)}\right| \right)^{p} {\bf 1}_{\{ \sigma_D = \sigma_{B^*} \} }\right] =  \E_0 \left[\left(1+ \left|x + RX_{\sigma_{B(0,1)}}\right|\right)^{p}\right] .
\]
On the other hand note that if $\sigma_D < \sigma_{B^{*}}$ then $X_{\mathcal{I}(N)} \in B^{*} \setminus D$. Therefore
\[
\E_x \left[\left( 1 + \left|X_{\mathcal{I}(N)}\right| \right)^{p} {\bf 1}_{\{\sigma_D < \sigma_{B^*} \}}\right] \leq \sup_{y \in B^{*} \setminus D} (1 + |y|)^{p}.
\]
We conclude
\[
J_1\leq B d^{p} \delta_g \Bigg( \E_0 \left[\left(1+ \left|x + RX_{\sigma_{B(0,1)}}\right|\right)^{p}\right] + \sup_{y \in B^{*} \setminus D} (1 + |y|)^{p} \Bigg).
\]
Using the Minkowski inequality and the fact that the sets $D$ and $B^{*}\setminus D$ are bounded, one has
\[
\begin{aligned}
	J_1 &\leq B d^{p} \delta_g \left( \left(1+ |x| + R\E_0\left[\left|X_{\sigma_{B(0,1)}}\right|^{p}\right]^{\frac{1}{p}}\right)^{p} + \sup_{y \in B^{*} \setminus D} (1 + |y|)^{p} \right)\\
	&\leq B d^{p} \delta_g \left( \left(K_1 + R \E_0 \left[\left|X_{\sigma_{B(0,1)}}\right|^{p}\right]^{\frac{1}{p}}\right)^{p} + K_2^{p} \right),
\end{aligned}
\]
where $K_1$ and $K_2$ are constants such that for all $x \in D$, $y \in B^{*} \setminus D$
\begin{equation}\label{eq:K1K2}
1 + |x| \leq K_1 \quad \hbox{ and } \quad 1 + |y| \leq K_2.
\end{equation}
By Corollaries \ref{cor:Xmoment} and \ref{cor:Kab} one has
\[
\E_0 \left[\left|X_{\sigma_{B(0,1)}}\right|^{p}\right]^{\frac{1}{p}} = K(\alpha,p)^{\frac{1}{p}} < \infty \quad \Longleftrightarrow \quad p < \alpha. 
\]
Therefore, from the choice of $p$, one has that $J_1$ is finite and bounded as follows:
\begin{equation}\label{eq:cotaJ1}
J_1 \leq B d^{p} \delta_g \left(\left(K_1 + R K(\alpha,p)^{\frac{1}{p}} \right)^{p} + K_2^{p} \right),
\end{equation}
with $K(\alpha,p)<+\infty$ defined in \eqref{eq:Kab}.

\medskip

\noindent
{\bf Step 3.} In this step we control the difference between the intermediate term $\E_x\left[\left(\mathcal{R}(\Phi_g)\right)\left(X_{\mathcal{I}(N)}\right)\right]$ previously introduced in Step 2, and the Monte Carlo $E_M(x)$ \eqref{eq:E_M}. Define
\[
J_2 := \norm{\E_x\left[\left(\mathcal{R}(\Phi_g)\right)\left(X_{\mathcal{I}(N)}\right)\right]-E_M(x)}_{L^q(\Omega,\PP_x)}.
\] 
In order to bound this term, we are going to use Corollary \ref{cor:MCq}. First of all notice from \eqref{H2} that	
\[
\E_x\left[|\left(\mathcal{R}(\Phi_g)\right)(X_{\mathcal{I}(N)})|\right] < Bd^p \E_x\left[\left(1+\left|X_{\mathcal{I}(N)}\right|\right)^p\right].
\]
Note by Step 2 that
\[
\E_x \left[\left(1 + \left|X_{\mathcal{I}(N)}\right|\right)^p\right] \leq \left(K_1 + R K(\alpha,p)^{\frac{1}{p}}\right)^p + K_2^{p} < +\infty,
\]
where $K_1$ and $K_2$ are defined in \eqref{eq:K1K2}. Therefore one can conclude that
\[
\E_x\left[\left|\left(\mathcal{R}(\Phi_g)\right)(X_{\mathcal{I}(N)})\right|\right]<\infty.
\]
Then for all $i \in \{1,...,M\}$, $\left(\mathcal{R}(\Phi_g)\right)(X^{i}_{\mathcal{I}(N_i)}) \in L^1(\Omega,\PP_x)$. For $s$ as in \eqref{condiciones}, Corollary \ref{cor:MCq} ensures that for all $q \in [s,\infty)$ (and in particular for all $q$ as in \eqref{condiciones}), one has
\begin{equation}\label{cota1_J2}
	J_2 \leq \frac{\Theta_{q,s}}{M^{1-\frac{1}{s}}}  \norm{\E_x\left[\left(\mathcal{R}(\Phi_g)\right)\left(X_{\mathcal{I}(N)}\right)\right] -  \left(\mathcal{R}(\Phi_g)\right)\left(X_{\mathcal{I}(N)}\right)}_{L^{q}(\Omega,\PP_x)}.
\end{equation}
Now we bound the norm on the right hand side of \eqref{cota1_J2}. By Minkowski's inequality one has
\[
\begin{aligned}
	& \norm{\E_x\left[\left(\mathcal{R}(\Phi_g)\right)\left(X_{\mathcal{I}(N)}\right)\right] -  \left(\mathcal{R}(\Phi_g)\right)\left(X_{\mathcal{I}(N)}\right)}_{L^{q}(\Omega,\PP_x)} \\
	&\qquad\leq \norm{\E_x\left[\left(\mathcal{R}(\Phi_g)\right)\left(X_{\mathcal{I}(N)}\right)\right]}_{L^q(\Omega,\PP_x)}+\norm{\left(\mathcal{R}(\Phi_g)\right)\left(X_{\mathcal{I}(N)}\right)}_{L^q(\Omega,\PP_x)}\\
	&\qquad\leq 2\E_x\left[\left|\left(\mathcal{R}(\Phi_g)\right)\left(X_{\mathcal{I}(N)}\right)\right|^q\right]^{\frac{1}{q}}.
\end{aligned}
\]
Now using hypothesis \eqref{H2} and the same results in previous Step to obtain
\[
\begin{aligned}
	& \norm{\E_x\left[\left(\mathcal{R}(\Phi_g)\right)\left(X_{\mathcal{I}(N)}\right)\right] -  \left(\mathcal{R}(\Phi_g)\right)\left(X_{\mathcal{I}(N)}\right)}_{L^{q}(\Omega,\PP_x)} \\	
	&\qquad\leq 2Bd^{p} \E_x \left[\left(1+\left|X_{\mathcal{I}(N)}\right|\right)^{pq}\right]^{\frac{1}{q}}\\
	&\qquad\leq 2Bd^{p} \left(\E_x \left[\left(1+\left|X_{\mathcal{I}(N)}\right|\right)^{pq}{\bf 1}_{\{\sigma_D=\sigma_{B^*} \} }\right]^\frac{1}{q} + \E_x \left[\left(1+\left|X_{\mathcal{I}(N)}\right|\right)^{pq}{\bf 1}_{\{\sigma_D<\sigma_{B^*} \}}\right]^{\frac{1}{q}}\right),
\end{aligned}
\]
where we recall that $B^{*}$ is a ball in $\R^d$ centered in $x$ with radious $R>1$ large enough such that $D \subset B^{*}$. Then using the scaling property of $X$ and Minkowski inequality, we have
\begin{equation}
\begin{aligned}\label{cota2_J2}
	& \norm{\E_x\left[\left(\mathcal{R}(\Phi_g)\right)\left(X_{\mathcal{I}(N)}\right)\right] -  \left(\mathcal{R}(\Phi_g)\right)\left(X_{\mathcal{I}(N)}\right)}_{L^{q}(\Omega,\PP_x)} \\
	&\hspace{1.6cm}\leq 2Bd^{p} \left( \E_0 \left[ \left(1 + \left|x+RX_{\sigma_{B(0,1)}}\right|\right)^{pq}\right]^{\frac{1}{q}} + \sup_{y \in B^{*}\setminus D} (1 + |y|)^{pq} \right)\\
	&\hspace{1.6cm}\leq 2Bd^{p} \left(  \left(K_1 + R \E_0\left[ \left|X_{\sigma_{B(0,1)}}\right|^{pq} \right]^{\frac{1}{pq}}\right)^{p} + K_2^{p}\right).
\end{aligned}
\end{equation}
Therefore, by \eqref{cota1_J2}, \eqref{cota2_J2} and Corollary \ref{cor:Kab} we have that $J_2$ is finite and bounded as follows:
\begin{equation}\label{eq:cotaJ2}
	J_2  \leq \frac{2\Theta_{q,s}}{M^{1-\frac{1}{s}}}Bd^{p} \left(  \left( K_1 + R K(\alpha,pq)^{\frac{1}{pq}}\right)^{p} +  K_2^{p}\right).
\end{equation}

\medskip

\noindent
{\bf Step 4.} Thanks to Steps 2 and 3 now it is possible to bound the difference
\[
\norm{\E_x \left[g\left(X_{\mathcal{I}(N)}\right)\right] - E_M(x)}_{L^q(\Omega,\PP_x)}.
\]
Indeed, first notice from Jensen inequality in \eqref{eq:Kab} that for $1<q<\frac{\alpha}p$ (see \eqref{condiciones}),
\[
K(\alpha,p)^{\frac{1}{p}} \leq K(\alpha,pq)^{\frac{1}{pq}}<+\infty .
\]
Condition $q < \frac{\alpha}{p}$ is necessary in order to have $K(\alpha,pq)$ finite (see Corollary \eqref{cor:Kab}). It follows from \eqref{eq:cotaJ1}, \eqref{eq:cotaJ2} and Minkowski's inequality that
\begin{equation}\label{eq:cotaP4_1}
\begin{aligned}
	&\norm{\E_x \left[g\left(X_{\mathcal{I}(N)}\right)\right] - E_M(x)}_{L^q(\Omega,\PP_x)} \leq J_1 + J_2\\
	&\leq \left(\delta_g + \frac{2\Theta_{q,s}}{M^{1- \frac{1}{s}}}\right)Bd^p\left(  \left(K_1 + R K(\alpha,pq)^{\frac{1}{pq}}\right)^{p} + K_2^{p}\right).
\end{aligned}
\end{equation}
Define
\begin{equation}
	C := \left(\left(K_1 +RK(\alpha,pq)^{\frac{1}{pq}}\right)^{p} + K_2^p\right) < \infty.
\end{equation}
Note that the choice of $R$ depends on the starting point $x$ in order to have $D \subset B(x,R)$. If we choose e.g. $R = 2\diam(D)$, it follows that for all $x \in D$, $D \subset B(x,2\diam(D))$ and then $C$ is uniform w.r.t. $x \in D$. Fubini and \eqref{eq:cotaP4_1} implies that
\begin{equation}\label{eq:cotaP4_2}
	\E_x \left[ \int_{D}\left| \E_x \left[g\left(X_{\mathcal{I}(N)}\right)\right] - E_M(x) \right|^q dx \right] \leq \left(\delta_g + \frac{2\Theta_{q,s}}{M^{1- \frac{1}{s}}}\right)^{q}|D|B^qd^{pq}C^q.
\end{equation}

In the following steps we are going to control two quantities that help us to obtain bounds for the random variables $N_i$ and $|Y_{i,n}|$, for all $i=1,...,M$, $n=1,...,N_i$. Although similar to the steps followed in \cite{Grohs}, here we need additional estimates because of the non continuous nature of the L\'evy jump processes.

\medskip

\noindent
{\bf Step 5.}
In order to bound the following expectation
\[\E_x \left[\left|\E_x[N] - \frac{1}{M} \sum_{i=1}^{M} N_i\right|^q\right],
\]
we are going to use Corollary \ref{cor:MCq}. Notice by Theorem \ref{teo:N} that for all $x \in D$ there exists a geometric random variable $\Gamma$ with parameter $\widetilde{p} = \widetilde{p}(\alpha,d) > 0$ such that
\[
\E_x\left[|N|\right] \leq \E_x\left[\Gamma\right] = \frac{1}{\widetilde{p}} < \infty,
\]
and then for all $i \in \{1,...,M\}$, $N_i \in L^1(\Omega,\PP_x)$. For $s$ as in \eqref{condiciones}, Corollary \ref{cor:MCq} implies for all $q$ as in \eqref{condiciones} that
\begin{equation}\label{eq:cotaP5_1}
	 \E_x\left[\left|\E_x[N]-\frac{1}{M} \sum_{i=1}^{M} N_i\right|^q\right]\leq \left( \frac{2\Theta_{q,s}}{M^{1-\frac{1}{s}}}\right)^q \E_x \left[|N|^{q}\right] \leq \left( \frac{2\Theta_{q,s}}{M^{1-\frac{1}{s}}}\right)^q \E_x \left[|N|^{2}\right],
\end{equation}
where we used that $q<2$ and then $\E_x\left[|\cdot|^q\right]\leq \E_x\left[|\cdot|^2\right]$. Recall that
\begin{equation}\label{eq:cotaN^2}
\E_x\left[|N|^2\right] \leq \E_x\left[\Gamma^2\right] = \frac{2-\widetilde{p}}{\widetilde{p}^2} < \infty,
\end{equation}
and therefore, it holds from \eqref{eq:cotaP5_1} and \eqref{eq:cotaN^2} that
\begin{equation}\label{eq:cotaP5_2}
	 \E_x\left[\left|\E_x[N]-\frac{1}{M} \sum_{i=1}^{M} N_i\right|^q\right] \leq \left( \frac{2\Theta_{q,s}}{M^{1-\frac{1}{s}}}\right)^q \frac{2-\widetilde{p}}{\widetilde{p}^{2}}.
\end{equation}
\medskip
\noindent
{\bf Step 6.} Finally, we want to estimate
\[
\E_x \left[\left|\E_x \left[\sum_{n=1}^{N} |Y_{n}|\right] - \frac{1}{M} \sum_{i=1}^{M} \sum_{n=1}^{N_i} |Y_{i,n}| \right|^{q} \right],
\]
where $Y_{i,n}$ were introduced in \eqref{eq:Y_n}. As in the previous step, we use the Corollary \ref{cor:MCq}. First of all, it follows from the independence of $\left(Y_{n}\right)_{n=1}^{k}$ and $N$ for fixed $k \in \N$ ($Y_{n}$ and $X$ are independent), and the law of total expectation that
\[
\begin{aligned}
	\E_x \left[\left|\sum_{n=1}^{N} |Y_{n}|\right|\right] &= \sum_{k \geq 1} \E_x \left[\left.\left|\sum_{n=1}^{N}|Y_{n}|\right| ~ \right| ~ N=k\right] \PP_x (N = k)\\
	&= \sum_{k \geq 1} \E_0 \left[\left|\sum_{n=1}^{k}|Y_{n}|\right|\right] \PP_x (N = k).
\end{aligned}
\]
Recall that $(Y_{n})_{n=1}^{k}$ are i.i.d. with the same distribution as $X_{\sigma_{B(0,1)}}$. Triangle inequality ensures that
\[
\begin{aligned}
	\E_x \left[\left|\sum_{n=1}^{N} |Y_{n}|\right|\right] &\leq \sum_{k\geq1} \sum_{n=1}^{k} \E_{0}\left[|Y_n|\right] \PP_x(N=k)\\
	&= \E_0\left[\left|X_{\sigma_{B(0,1)}}\right|\right] \sum_{k \geq 1} k \PP_x(N=k)\\
	&= K(\alpha,1) \E_x[N].
\end{aligned}
\]
Then for all $i \in \{1,...,M\}$, $\sum_{n=1}^{N_i} |Y_{i,n}| \in L^{1}(\Omega,\PP_x)$. Moreover, with similar arguments 
\[
\begin{aligned}
	\E_x \left[\left|\sum_{n=1}^{N} |Y_{n}|\right|^q\right] &= \sum_{k \geq 1} \E_x \left[\left.\left|\sum_{n=1}^{N}|Y_{n}|\right|^q \right| ~ N=k\right] \PP_x (N = k)\\
	&= \sum_{k \geq 1} \E_0 \left[\left|\sum_{n=1}^{k}|Y_{n}|\right|^q\right] \PP_x (N = k).
\end{aligned}
\]
Recall from the bounds of $q$ that appear in \eqref{condiciones}, one has that $q \in (1,2)$ and the function $|\cdot|^q$ is convex. This implies that, for all $k \in \N$
\[
\left|\sum_{n=1}^{k} \frac{|Y_n|}{k}\right|^q \leq \sum_{n=1}^{k} \frac{|Y_n|^q}{k}.
\] 
Therefore
\[
\left|\sum_{n=1}^{k}|Y_n|\right|^q \leq k^{q-1} \sum_{n=1}^{k} |Y_n|^{q}.
\]
Replacing this on the previous estimate one has
\begin{equation}\label{eq:cotaP6_1}
\begin{aligned}
	\E_x \left[\left|\sum_{n=1}^{N} |Y_{n}|\right|^q\right] &\leq \sum_{k\geq1} \sum_{n=1}^{k} k^{q-1} \E_{0}\left[|Y_n|^q\right] \PP_x(N=k)\\
	&= \E_0\left[\left|X_{\sigma_{B(0,1)}}\right|^q\right] \sum_{k \geq 1} k^q \PP_x(N=k)\\
	&= K(\alpha,q) \E_x[N^q].
\end{aligned}
\end{equation}
%\begin{equation}\label{eq:cotaP6_1}
%	\E_x \left[\left|\sum_{n=1}^{N} |Y_{i,n}|\right|^q\right] \leq K(\alpha,q) \E_x[N^{q}].
%\end{equation}
For $s$ as in \eqref{condiciones}, Corollary \ref{cor:MCq} implies for all $q$ as in \eqref{condiciones} that
\[
\E_x \left[\left|\E_x \left[\sum_{n=1}^{N} |Y_{n}|\right] - \frac{1}{M} \sum_{i=1}^{M} \sum_{n=1}^{N_i} |Y_{i,n}| \right|^{q} \right] \leq \left( \frac{2\Theta_{q,s}}{M^{1-\frac{1}{s}}}\right)^{q} \E_x \left[\left|\sum_{n=1}^{N}|Y_{n}|\right|^q\right].
\]
Therefore it follows from \eqref{eq:cotaN^2} and \eqref{eq:cotaP6_1} that

\begin{equation}\label{eq:cotaP6_2}
\E_x \left[\left|\E_x \left[\sum_{n=1}^{N} |Y_{n}|\right] - \frac{1}{M} \sum_{i=1}^{M} \sum_{n=1}^{N_i} |Y_{i,n}| \right|^{q} \right] \leq \left( \frac{2\Theta_{q,s}}{M^{1-\frac{1}{s}}}\right)^{q} K(\alpha,q) \frac{2-\widetilde{p}}{\widetilde{p}^2}.
\end{equation}
{\bf Step 7.}
Coupling the bounds obtained in \eqref{eq:cotaP4_2}, \eqref{eq:cotaP5_2} and \eqref{eq:cotaP6_2}, it holds that
\begin{equation}
\begin{aligned}
	&\E_x \left[ \int_{D}\left| \E_x \left[g\left(X_{\mathcal{I}(N)}\right)\right] - E_M(x) \right|^q dx  + \left|\E_x[N] - \frac{1}{M} \sum_{i=1}^{M} N_i\right|^q \right.  \\
	& \qquad   \left. + \left|\E_x\left[\sum_{n=1}^{N}|Y_n|\right]-\frac{1}{M}\sum_{i=1}^{M}\sum_{n=1}^{N_i} |Y_{i,n}|\right|^q\right] \\
	&\leq \left(\delta_g + \frac{2\Theta_{q,s}}{M^{1- \frac{1}{s}}}\right)^{q}|D|B^qd^{pq}C^q + \left(\frac{2\Theta_{q,s}}{M^{1-\frac{1}{s}}}\right)^q (1+ K(\alpha,q))\frac{2-\widetilde{p}}{\widetilde{p}^{2}} . %+ \left(2 \frac{\Theta_{q,s}(\R)}{M^{1-\frac{1}{s}}}\right)^{q} K(\alpha,q) \frac{2-\widetilde{p}}{\widetilde{p}^2}.\\
	&=: \hbox{error}_g^q
\end{aligned}
\end{equation}
Using now that $\E(Z) \leq c<+\infty$, we summarize the following result.

\begin{lem}\label{lem:sacada}
There exists 
%a subset $A \subset \Omega$ of positive probability such that for all $\omega \in A$,  there exist 
$\overline{N}_{i} \in \N$, $Y_{i,n}$ i.i.d. copies of $X_{\sigma_{B(0,1)}}$ under $\PP_0$, $i=1,...,M$, $n=1,...,\overline{N}_i$ such that
\begin{equation}
\begin{aligned}
	&\int_{D}\left| \E_x \left[g\left(X_{\mathcal{I}(N)}\right)\right] - \frac{1}{M} \sum_{i=1}^{M} \left(\mathcal{R}(\Phi_g)\right)\left(X^{i}_{\mathcal{I}(\overline{N}_i)}\right) \right|^q dx   \\
	& \quad  + \left|\E_x[N] - \frac{1}{M} \sum_{i=1}^{M} \overline{N}_i\right|^q + \left|\E_x\left[\sum_{n=1}^{N}|Y_n|\right]-\frac{1}{M}\sum_{i=1}^{M}\sum_{n=1}^{N_i} |Y_{i,n}|\right|^q\\
	&\qquad \leq \hbox{error}_g^q.
\end{aligned}
\end{equation}
\end{lem}

With a slight abuse of notation, we redefine $E_M$ from \eqref{eq:E_M} as
\begin{equation}\label{eq:E_M_final}
E_M(x) = \frac{1}{M} \sum_{i=1}^{M} \left(\mathcal{R}(\Phi_g)\right)\left(X^{i}_{\mathcal{I}(\overline{N}_i)}\right).
\end{equation}

\medskip

\noindent
{\bf Step 8.} We are going to prove now that $X^{i}_{\mathcal{I}(\overline{N}_i)}$ can be approximated by a ReLu DNN. Let $\delta_{\dist} \in (0,1).$ Recall that from \eqref{HD-1} there exists $\Phi_{\dist} \in \textbf{N}$ ReLu DNN such that for all $x \in D$
\[
\left|\left(\mathcal{R}(\Phi_{\dist})\right)(x) - \dist(x,\partial D)\right| \leq \delta_{\dist}.
\]
Define $\left(\Phi_{i,n}\right)_{i=1,...,M, n=1,...,\overline{N}_i} \in \textbf{N}$ as follows: for $x \in D$
\begin{equation}\label{eq:Phi_i1}
\left(\mathcal{R}(\Phi_{i,1})\right)(x) = x + Y_{i,1}\left(\mathcal{R}(\Phi_{\dist})\right)(x), 
\end{equation}
and for all $n = 2,...,\overline{N}_{i}$, $x \in D$
\begin{equation}\label{eq:Phi_in}
\left(\mathcal{R}(\Phi_{i,n})\right)(x) = \left(\mathcal{R}(\Phi_{i,n-1})\right)(x) + Y_{i,n}\left(\mathcal{R}(\Phi_{\dist}) \circ \mathcal{R}(\Phi_{i,n-1})\right)(x).
\end{equation}
In the Section \ref{Sect:6p3} we will see that $\left(\Phi_{i,n}\right)_{i=1,...,M, n=1,...,\overline{N}_i}$ is indeed a ReLu DNN. Note that, for $x \in D$, $i=1,...,M$,
\[
\left|X^{i}_{\mathcal{I}(1)}- \left(\mathcal{R}(\Phi_{i,1})\right)(x)\right| \leq |Y_{i,1}| \left|\left(\mathcal{R}(\Phi_{\dist})\right)(x) - \dist(x,\partial D)\right| \leq \delta_{\dist}\sum_{n=1}^{\overline{N}_i}|Y_{i,n}|,
\]
and for all $n = 2,...,\overline{N}_i$, by triangle inequality
\[
\begin{aligned}
&\left|X^{i}_{\mathcal{I}(n)}- \left(\mathcal{R}(\Phi_{i,n})\right)(x)\right|\leq \left|X^{i}_{\mathcal{I}(n-1)}- \left(\mathcal{R}(\Phi_{i,n-1})\right)(x)\right|\\
&\qquad+ |Y_{i,n}| \left|\dist\left(X^{i}_{\mathcal{I}(n-1)},\partial D\right) - \dist\left(\left(\mathcal{R}(\Phi_{i,n-1})\right)(x),\partial D\right)\right| \\
&\qquad+ |Y_{i,n}|\left| \dist\left(\left(\mathcal{R}(\Phi_{i,n-1})\right)(x),\partial D\right) -\left(\mathcal{R}(\Phi_{\dist}) \circ \mathcal{R}(\Phi_{i,n-1})\right)(x)\right|.
\end{aligned}
\]
Using the hypothesis on $\Phi_{\dist}$ and the fact that the function $x \to \dist(x,\partial D)$ is 1-Lipschitz one has
\[
\left|X^{i}_{\mathcal{I}(n)}- \left(\mathcal{R}(\Phi_{i,n})\right)(x)\right|\leq \left(\sum_{n=1}^{\overline{N}_i}|Y_{i,n}| + 1\right)\left|X^{i}_{\mathcal{I}(n-1)}- \left(\mathcal{R}(\Phi_{i,n-1})\right)(x)\right| + \delta_{\dist} \sum_{n=1}^{\overline{N}_i}|Y_{i,n}|.
\]
By the previous recursion one obtain that for all $i=1,...,M$
\[
\begin{aligned}
\left|X^{i}_{\mathcal{I}(\overline{N}_i)}- \left(\mathcal{R}(\Phi_{i,\overline{N}_i})\right)(x)\right|&\leq \left(\sum_{n=1}^{\overline{N}_i}|Y_{i,n}|\right) \delta_{\dist}\sum_{i=1}^{\overline{N}_i} \left(\sum_{n=1}^{\overline{N}_i}|Y_{i,n}| + 1\right)^{i-1} \\
& \leq \left(\sum_{n=1}^{\overline{N}_i}|Y_{i,n}|\right) \delta_{\dist} \frac{\left(\sum_{n=1}^{\overline{N}_i}|Y_{i,n}|+1\right)^{\overline{N}_i}-1}{\left(\sum_{n=1}^{\overline{N}_i}|Y_{i,n}|+1\right)-1}\\
& \leq \delta_{\dist}\left(\sum_{n=1}^{\overline{N}_i}|Y_{i,n}| + 1\right)^{\overline{N}_i}. 
\end{aligned}
\]

\medskip

\noindent
{\bf Step 9.} With the ReLu DNNs defined in Step 8, we are able to find a ReLu DNN that approximates $E_M(x)$.
Define $\Phi^{i}_{g} \in \textbf{N}$ as follows
\begin{equation}\label{eq:DNN_gi}
\left(\mathcal{R}(\Phi^{i}_g)\right)(x) = \left(\mathcal{R}(\Phi_{g}) \circ \mathcal{R}(\Phi_{i,\overline{N}_i})\right)(x),
\end{equation}
valid for $x \in D$. Notice from Lemma \ref{lem:DNN_comp} that $\Phi_g^i$ is indeed a ReLu DNN. For full details see Section \ref{Sect:6p1}. By triangle inequality one has
\[
\begin{aligned}
&\left|\left(\mathcal{R}(\Phi_g)\right)\left(X^{i}_{\mathcal{I}(\overline{N}_i)}\right)-\left(\mathcal{R}(\Phi^{i}_g)\right)(x)\right| \\
&\qquad \leq \left|\left(\mathcal{R}(\Phi_g)\right)\left(X^{i}_{\mathcal{I}(\overline{N}_i)}\right) - g \left(X^{i}_{\mathcal{I}(\overline{N}_i)}\right)\right| + \left|g \left(X^{i}_{\mathcal{I}(\overline{N}_i)}\right) - \left(g \circ \mathcal{R}(\Phi_{i,\overline{N}_i})\right)(x)\right| \\
&\qquad \qquad + \left|\left(g \circ \mathcal{R}(\Phi_{i,\overline{N}_i})\right)(x) - \left(\mathcal{R}(\Phi_g^{i})\right)(x)\right|.
\end{aligned}
\]
We use the hypothesis \eqref{H1} and the assumption that $g$ is $L_g$-Lipschitz to obtain
\[
\begin{aligned}
	&\left|\left(\mathcal{R}(\Phi_g)\right)\left(X^{i}_{\mathcal{I}(\overline{N}_i)}\right)-\left(\mathcal{R}(\Phi^{i}_g)\right)(x)\right|\\
	&\quad \leq Bd^p \delta_g \left(\left(1+ \left|X^{i}_{\mathcal{I}(\overline{N}_i)}\right|\right)^p + \left(1+ \left|\left(\mathcal{R}(\Phi_{i,\overline{N}_i})\right)(x)\right|\right)^p\right) + L_g\left|X^{i}_{\mathcal{I}(\overline{N}_i)}- \left(\mathcal{R}(\Phi_{i,\overline{N}_i})\right)(x)\right|.
\end{aligned}
\]
By triangle inequality one has
\[
\left|\left(\mathcal{R}(\Phi_{i,\overline{N}_i})\right)(x)\right| \leq \left|X_{\mathcal{I}(\overline{N}_i)}^i - \left(\mathcal{R}(\Phi_{i,\overline{N}_i})\right)(x)\right| + \left|X_{\mathcal{I}(\overline{N}_i)}^i\right|.
\]
With the previous estimate and using that $(\cdot)^{p}$ is a convex function, we obtain
\[
\begin{aligned}
&\left|\left(\mathcal{R}(\Phi_g)\right)\left(X^{i}_{\mathcal{I}(\overline{N}_i)}\right)-\left(\mathcal{R}(\Phi^{i}_g)\right)(x)\right| \\
&   \leq Bd^p\delta_g\left(1 + \left|X_{\mathcal{I}(\overline{N}_i)}^i\right|\right)^p + L_g\left|X^{i}_{\mathcal{I}(\overline{N}_i)}- \left(\mathcal{R}(\Phi_{i,\overline{N}_i})\right)(x)\right|\\
&\quad  + 2^{p-1}Bd^p \delta_g \left(\left(1 + \left|X_{\mathcal{I}(\overline{N}_i)}^i\right|\right)^p + \left|X_{\mathcal{I}(\overline{N}_i)}^i - \left(\mathcal{R}(\Phi_{i,\overline{N}_i})\right)(x)\right|^p\right).
\end{aligned}
\]
Notice that from \eqref{eq:paseo},
\[
\left|X_{\mathcal{I}(\overline{N}_i)}^i\right| \leq |x| + \diam(D) \sum_{n=1}^{\overline{N}_i} |Y_{i,n}|.
\]
Therefore, in addition to Step 6, one has
\begin{equation}\label{eq:step8g}
\begin{aligned}
	&\left|\left(\mathcal{R}(\Phi_g)\right)\left(X^{i}_{\mathcal{I}(\overline{N}_i)}\right)-\left(\mathcal{R}(\Phi^{i}_g)\right)(x)\right| \leq 3Bd^p\delta_g \left(1 + |x| + \diam(D)\sum_{n=1}^{\overline{N}_i} |Y_{i,n}|\right)^p\\
	&\qquad + L_g \delta_{\dist}\left(1 + \sum_{n=1}^{\overline{N}_i} |Y_{i,n}|\right)^{\overline{N}_i}+ 2Bd^p\delta_g \delta_{\dist}^p \left(1 + \sum_{n=1}^{\overline{N}_i} |Y_{i,n}|\right)^{p \overline{N}_i}.
\end{aligned}
\end{equation}
Now define for $\varepsilon \in (0,1)$ the ReLu DNN $\Psi_{1,\varepsilon}$ such that it satisfies for $x \in D$
\[
\left(\mathcal{R}(\Psi_{1,\varepsilon})\right)(x) = \frac{1}{M} \sum_{i=1}^{M} \left(\mathcal{R}(\Phi^{i}_g)\right)(x).
\]
This is the requested DNN. Section \ref{Sect:6p3} shows that $\Psi_{1,\varepsilon}$ is a ReLu DNN. From the bound obtained in \eqref{eq:step8g}, we have that
\[
\begin{aligned}
&\left|E_M(x)- \left(\mathcal{R}(\Psi_{1,\varepsilon})\right)(x)\right| \leq \frac{1}{M} \sum_{i=1}^{M} \left|\left(\mathcal{R}(\Phi_g)\right)\left(X^{i}_{\mathcal{I}(\overline{N}_i)}\right)-\left(\mathcal{R}(\Phi^{i}_g)\right)(x)\right|\\
&\qquad\leq 3Bd^p \delta_g \left(K_1 + \diam(D)\sum_{i=1}^{M} \sum_{n=1}^{\overline{N}_i} |Y_{i,n}|\right)^p + L_g \delta_{\dist} \left( 1 + \sum_{i=1}^{M} \sum_{n=1}^{\overline{N}_i} |Y_{i,n}|\right)^{\sum_{i=1}^{M}\overline{N}_i} \\
&\qquad \qquad+ 2Bd^p\delta_g\delta_{\dist}^p \left(1 + \sum_{i=1}^{M} \sum_{n=1}^{\overline{N}_i} |Y_{i,n}|\right)^{p \sum_{i=1}^{M} \overline{N}_i}.
\end{aligned}
\]

\medskip

\noindent
{\bf Step 10.}
We want to bound error$_g$. Using that $\frac{1}{q} < 1$ one has
\[
\begin{aligned}
\hbox{error}_g &\leq \left(\delta_g + \frac{2\Theta_{q,s}}{M^{1-\frac{1}{s}}}\right) |D|^{\frac{1}{q}} B d^p C + 2 \frac{\Theta_{q,s}}{M^{1-\frac{1}{s}}} \left(1+K(\alpha,q)^\frac{1}{q}\right) \left(\frac{2-\widetilde{p}}{\widetilde{p}^2}\right)^{\frac{1}{q}}.\\
&= 2\frac{\Theta_{q,s}}{M^{1-\frac{1}{s}}} \left(|D|^{\frac{1}{q}}Bd^pC + \left(\frac{2-\widetilde{p}}{\widetilde{p}^2}\right)^{\frac{1}{q}}\left(1+K(\alpha,q)^{\frac{1}{q}}\right)\right) + \delta_g |D|^\frac{1}{q}Bd^{p}C
\end{aligned}
\]
Denote
\begin{equation}\label{eq:C1C2}
C_1 = 2\Theta_{q,s} \left(|D|^{\frac{1}{q}}Bd^pC + \left(\frac{2-\widetilde{p}}{\widetilde{p}^2}\right)^{\frac{1}{q}}\left(1+K(\alpha,q)^{\frac{1}{q}}\right)\right), \quad \hbox{and} \quad C_2 = |D|^\frac{1}{q}Bd^{p}C.
\end{equation}
Note that $C_1$ and $C_2$ are polynomial on the dimension $d$. Then
\begin{equation}\label{eq:error_g}
\hbox{error}_g \leq \frac{C_1}{M^{1-\frac{1}{s}}} + C_2 \delta_g.
\end{equation}
In adition, thanks to Step 5, one has
\[
\sum_{i=1}^{M} \sum_{n=1}^{\overline{N}_i} |Y_{i,n}| \leq M \left( \hbox{error}_g + \E_x \left[\sum_{n=1}^{N} |Y_{i,n}|\right]\right) \leq M \left( \hbox{error}_g + K(\alpha,1) \frac{1}{\widetilde{p}} \right).
\]
Define $C_3 := K(\alpha,1)/\widetilde{p}$, then
\[
\begin{aligned}
\sum_{i=1}^{M} \sum_{n=1}^{\overline{N}_i} |Y_{i,n}| &\leq M^{\frac{1}{s}}C_1 + M\delta_{g}C_2 + MC_3\\
&\leq M^{\frac{1}{s}}C_1 + M(C_2 + C_3).
\end{aligned}
\]
On the other hand side define $C_4 = \frac{1}{\widetilde{p}}$, then
\[
\begin{aligned}
\sum_{i=1}^{M} \overline{N}_i &\leq M (\hbox{error}_g + \E_x[N]) \leq M^{\frac{1}{s}} C_1 + M\delta_g C_2 + \frac{M}{\widetilde{p}}\\
&\leq M^{\frac{1}{s}} C_1 + M(C_2 + C_4).
\end{aligned}
\]
\medskip

\noindent
{\bf Step 11.} Using the auxiliary Lemma \ref{lem:sacada} and \eqref{eq:error_g}, it follows that
\[
\left(\int_D \left|\E_x \left[g\left(X_{\mathcal{I}(N)}\right)\right] - E_M(x)\right|^{q}dx\right)^{\frac{1}{q}} \leq \frac{C_1}{M^{1-\frac 1s}} + C_2 \delta_g.
\]
In addition, from Step 9 and \eqref{eq:error_g} one has
\[
\begin{aligned}
	&\left(\int_D \left|E_M(x) - \left(\mathcal{R}(\Psi_1)\right)(x)\right|^{q}dx\right)^{\frac{1}{q}}\\
	&\qquad\leq 3|D|^{\frac{1}{q}}Bd^p\delta_g \left(K_1 + \diam(D)\left(M^{\frac{1}{s}}C_1 + M (C_2 + C_3)\right)\right)^p \\
	&\qquad \quad+ |D|^{\frac{1}{q}}L_g \delta_{\dist} \left(1 + M^{\frac{1}{s}}C_1 + M (C_2 + C_3)\right)^{M^{\frac{1}{s}}C_1 + M (C_2 + C_4)} \\
	&\qquad \quad + 2|D|^{\frac{1}{q}}Bd^p\delta_g \delta_{\dist}^p  \left(1 + M^{\frac{1}{s}}C_1 + M (C_2 + C_3)\right)^{p\left(M^{\frac{1}{s}}C_1 + M(C_2 + C_4)\right)}. 
\end{aligned}
\]
Therefore, Minkowski inequality implies that
\begin{equation}\label{eq:step11g}
\begin{aligned}
	&\left(\int_D \left|\E_x \left[g\left(X_{\mathcal{I}(N)}\right)\right] - \left(\mathcal{R}(\Psi_{1,\varepsilon})\right)(x)\right|^{q}dx\right)^{\frac{1}{q}} \\
	&\leq \left(\int_D \left|\E_x \left[g\left(X_{\mathcal{I}(N)}\right)\right] - E_M(x)\right|^{q}dx\right)^{\frac{1}{q}} + \left(\int_D \left|E_M(x) - \left(\mathcal{R}(\Psi_1)\right)(x)\right|^{q}dx\right)^{\frac{1}{q}}\\
	&\leq \frac{C_1}{M^{1-\frac{1}{s}}} + C_2 \delta_g  + 3|D|^{\frac{1}{q}}Bd^p\delta_g \left(K_1 + \diam(D)\left(M^{\frac{1}{s}}C_1 + M (C_2 + C_3)\right)\right)^p \\
	&\qquad+ |D|^{\frac{1}{q}}L_g \delta_{\dist} \left(1 + M^{\frac{1}{s}}C_1 + M (C_2 + C_3)\right)^{M^{\frac{1}{s}}C_1 + M (C_2 + C_4)} \\
	&\qquad + 2|D|^{\frac{1}{q}}Bd^p\delta_g \delta_{\dist}^p  \left(1 + M^{\frac{1}{s}}C_1 + M (C_2 + C_3)\right)^{p\left(M^{\frac{1}{s}}C_1 + M(C_2 + C_4)\right)}. 
	%&\leq \hbox{error}_g + |D|^{\frac{1}{q}}L(K_Y+1)^{\sum_{i=1}^{M} \overline{N}_i} \delta_{\dist} + |D|^{\frac{1}{q}}Bd^p\delta_g \left(5K_1^2 + 5K_Y^2\diam(D)^2\sum_{i=1}^{M}\overline{N}_i^2 + 3(K_Y+1)^{2\sum_{i=1}^{M}\overline{N}_i}\delta_{\dist}^2\right).
\end{aligned}
\end{equation}
for $\varepsilon \in (0,1)$ let $M \in \N$ large enough such that
\[
M = \left\lceil \left(\frac{5C_1}{\varepsilon}\right)^{\frac{s}{s-1}} \right\rceil
\],
%\[
% \frac{C_1}{M^{1-\frac{1}{s}}} = \frac{\varepsilon}{5},
%\]
and $\delta_{\dist} \in (0,1)$ small enough such that
\[
\delta_{\dist} = \frac{\varepsilon}{5|D|^{\frac{1}{q}}L_g} \left(1 + M^{\frac{1}{s}}C_1 + M (C_2 + C_3)\right)^{-\left(M^{\frac{1}{s}}C_1 + M (C_2 + C_4)\right)}.
\]
Let
\begin{equation}\label{C5}
C_5 = \max \left\{ C_2, 3|D|^{\frac{1}{q}}Bd^p\left(K_1 + \diam(D)\left(M^{\frac{1}{s}}C_1 + M(C_2 + C_3)\right)\right)^{p},\frac{2|D|^{\frac{1}{q}}Bd^p}{5^p|D|^{\frac{p}{q}} L_g^p} \right\},
\end{equation}
and consider $\delta_g \in (0,1)$ small enough such that
\[
\delta_g = \frac{\varepsilon }{5C_5}.
\] 
Therefore each term of \eqref{eq:step11g} can be bounded by $\varepsilon/5$. Then 
\[
\begin{aligned}
	&\left(\int_D \left|\E_x \left[g\left(X_{\mathcal{I}(N)}\right)\right] - \left(\mathcal{R}(\Psi_{1,\varepsilon})\right)(x)\right|^{q}dx\right)^{\frac{1}{q}} \leq \varepsilon.\\
	%&\qquad\leq \frac{C_1}{M^{1-\frac{1}{s}}} + C_2 \delta_g + 3|D|^{\frac{1}{q}}Bd^p\delta_g \left(K_1 + \diam(D)\left(M^{\frac{1}{s}}C_1 + M\delta_g C_2 + MC_3\right)\right)^p\\
	%&\qquad \qquad + L_g \delta_g + 2Bd^p \delta_g^{p+1}.
\end{aligned}
\]
This allos us to conclude that \ref{eq:2.2} can be approximated in $L^{q}(D)$ by a DNN $\Psi_{1,\varepsilon}$ with accurateness $\varepsilon \in (0,1)$.
\medskip

\subsection{Proof of Proposition \ref{Prop:homo}: Quantification of DNNs} \label{Sect:6p3}
In this Section we will prove that $\Psi_{1,\varepsilon}$ is in fact a ReLu DNN wich does not suffer of the curse of dimensionality.

\medskip
\noindent
{\bf Step 12.} We now use the Definitions and Lemmas of Section \ref{Sect:3} to study $\Psi_{1,\varepsilon}$. Let
\[
\beta_{\dist} = \mathcal{D}(\Phi_{\dist}) \quad \hbox{and} \quad H_{\dist} = \dim(\beta_{\dist}) - 2.
\]
And we will verify by induction that  for all $i=1,...,M$ $n=1,...,\overline{N}_i$, $\Phi_{i,n}$ (defined in \ref{eq:Phi_i1} and \ref{eq:Phi_in}) is a ReLu DNN that satisfy %$\mathcal{R}(\Phi_{i,n}) \in \mathcal{C}(D,\R^d)$ with
\begin{equation}\label{eq:HI_homo}
\mathcal{D}(\Phi_{i,n}) = \overunderset{n}{m=1}{\odot}(d\mathfrak{n}_{H_{\dist}+2}\boxplus \widetilde{\beta}_{\dist}),
\end{equation}
where
\[
\widetilde{\beta}_{\dist} = (\beta_{\dist,0},...,\beta_{\dist,H_{\dist}},d) \in \N^{H_{\dist}+2}.
\]
If \ref{eq:HI_homo} is true, then from \eqref{eq:HI_homo} and the definition of the operator $\odot$ is easy to see that
\begin{equation}\label{eq:HI_homo2}
\vertiii{\mathcal{D}(\Phi_{i,n}) } \leq 2d + \vertiii{\mathcal{D}(\Phi_{\dist})}, \quad \hbox{and} \quad \dim(\mathcal{D}(\Phi_{i,n})) = (H_{\dist}+1)n+1.
\end{equation}
For $n=1$, recall the definition of $\Phi_{i,1}$ from \eqref{eq:Phi_i1}. By Lemma \ref{lem:DNN_mat} one has that
\[
Y_{i,1} \mathcal{R}(\Phi_{\dist}) \in \mathcal{R}\left(\left\{\Phi \in \textbf{N}: \mathcal{D}(\Phi) = \widetilde{\beta}_{\dist} \right\}\right).
\]
%Moreover, $Y_{i,1}\mathcal{R}(\Phi_{\dist}) \in C(D,\R^d)$ and $\dim(\widetilde{\beta}_{\dist})=H_{\dist}+2$. 
By Lemma \ref{lem:DNN_id}, the identity function can be represented by a ReLu DNN with $H_{\dist} + 2$ number of layers. Therefore by Lemma \ref{lem:DNN_sum} it follows that $\mathcal{R}(\Phi_{i,1}) \in C(D,\R^d)$ and
\[
\mathcal{D}(\Phi_{i,1}) = d\mathfrak{n}_{H_{\dist} + 2} \boxplus \widetilde{\beta}_{\dist}, \qquad \dim(\mathcal{D}(\Phi_{i,1})) = H_{\dist} + 2.
\]
Moreover
\[
\vertiii{\mathcal{D}(\Phi_{i,1}) } \leq 2d + \vertiii{\mathcal{D}(\Phi_{\dist})}.
\]
Now suppose that for $n=2,...,\overline{N}_i-1$ that \eqref{eq:HI_homo} is valid. Recall the definition of $\Phi_{i,n}$ from \eqref{eq:Phi_in}. Notice that $\mathcal{R}(\Phi_{i,n+1})$ can be written as
\[
\mathcal{R}(\Phi_{i,n+1}) = \mathcal{R}(\widetilde{\Phi}_{i,n+1}) \circ \mathcal{R}(\Phi_{i,n}).
\]
where $\widetilde{\Phi}_{i,n} \in \textbf{N}$ is a ReLu DNN that satisfies
\[
\left(\mathcal{R}(\widetilde{\Phi}_{i,n})\right)(x) = x + Y_{i,n} \left(\mathcal{R}(\Phi_{\dist})\right)(x).
\]
By the same arguments as in the case $n=1$, it follows for all $n=2,...,\overline{N}_i$ that %$\mathcal{R}(\widetilde{\Phi}_{i,n}) \in C(D,\R^d)$. In addition
\[
\mathcal{D}(\widetilde{\Phi}_{i,n}) = d\mathfrak{n}_{H_{\dist}+2} \boxplus \widetilde{\beta}_{\dist}, \quad \dim(\mathcal{D}(\widetilde{\Phi}_{i,n})) = H_{\dist} + 2, 
\]
and
\[
\vertiii{\mathcal{D}(\widetilde{\Phi}_{i,n})} \leq 2d + \vertiii{\mathcal{D}(\Phi_{\dist})}.
\]
Therefore from the inductive hypothesis \eqref{eq:HI_homo} and Lemma \ref{lem:DNN_comp}, $\Phi_{i,n+1}$ is a ReLu DNN that satisfies %$\mathcal{R}(\Phi_{i,n+1}) \in C(D,\R^d)$ and
\[
\mathcal{D}(\Phi_{i,n+1}) = (d\mathfrak{n}_{H_{\dist}+2} \boxplus \widetilde{\beta}_{\dist}) \odot \left(\overunderset{n}{m=1}{\odot}(d\mathfrak{n}_{H_{\dist}+2}\boxplus \widetilde{\beta}_{\dist})\right) = \overunderset{n+1}{m=1}{\odot}(d\mathfrak{n}_{H_{\dist}+2}\boxplus \widetilde{\beta}_{\dist}).
\]
Then the claim for $\Phi_{i,n}$ is proved for any $i=1,...,M$, $n=1,...,\overline{N}_i$. Recall that \ref{eq:HI_homo2} is valid too. Therefore
\[
\mathcal{D}(\Phi_{i,\overline{N}_i}) = \overunderset{\overline{N}_i}{m=1}{\odot}(d\mathfrak{n}_{H_{\dist}+2}\boxplus \widetilde{\beta}_{\dist}).
\]
Moreover
\[
\vertiii{\mathcal{D}(\Phi_{i,\overline{N}_i})} \leq 2d + \vertiii{\mathcal{D}(\Phi_{\dist})}, \quad \hbox{and} \quad \dim(\mathcal{D}(\Phi_{i,\overline{N}_i})) = (H_{\dist}+1)\overline{N}_i+1.
\]
Let $\beta_g = \mathcal{D}(\Phi_g)$ and $H_g = \dim(\beta_g) - 2$. By Lemma \ref{lem:DNN_comp} one has that %$\mathcal{R}(\Phi_g^i) \in C(D,\R)$ and
\[
\mathcal{D}(\Phi_{g}^{i}) = \beta_g \odot \left(\overunderset{\overline{N}_i}{m=1}{\odot}(d\mathfrak{n}_{H_{\dist}+2}\boxplus \widetilde{\beta}_{\dist})\right), \qquad \dim(\mathcal{D}(\Phi_{g}^{i})) =(H_{\dist}+1) \overline{N}_i + H_g + 2.
\]
Moreover
\[
\vertiii{\mathcal{D}(\Phi_g^{i})} \leq \max\{\vertiii{\mathcal{D}(\Phi_g)},2d+\vertiii{\mathcal{D}(\Phi_{\dist})} \}. 
\]

Recall that $\overline{N}_i$ not necessarily be the same for $i=1,...,M$. Now we need that for all $i=1,...,M$, $\Phi_{g}^{i}$ have the same number of layers to use Lemma \ref{lem:DNN_sum}. For any $i=1,...,M$ define
\[
H_i = (H_{\dist}+1)\left(\sum_{j=1}^{M} \overline{N}_j - \overline{N}_i\right) - 1.
\]
By Lemma \ref{lem:DNN_id}, The identity function can be represented by a ReLu DNN with $H_i$ hidden layers. Recall the definition of $\Phi_g^i$ in \eqref{eq:DNN_gi}. Using Lemma \ref{lem:DNN_comp} we have that %$\mathcal{R}(\Phi_g^i) \in C(D,\R)$ with
\[
\mathcal{D}(\Phi_g^i) = \mathfrak{n}_{H_i + 2} \odot \beta_g \odot \left(\overunderset{\overline{N}_i}{m=1}{\odot}(d\mathfrak{n}_{H_{\dist}+2}\boxplus \widetilde{\beta}_{\dist})\right),
\]
and
\[
\dim(\mathcal{D}(\Phi_{g}^{i})) = (H_{\dist}+1)\sum_{i=1}^{M}\overline{N}_i  + H_g + 2.
\]
Now we use Lemma \ref{lem:DNN_sum} to conclude that %$\mathcal{R}(\Psi_{1,\varepsilon}) \in C(D,\R)$ with
\[
\mathcal{D}(\Psi_{1,\varepsilon}) = \overset{M}{\underset{i=1}{\boxplus}} \left(\mathfrak{n}_{H_i+2} \odot \beta_g \odot \left(\overunderset{\overline{N}_i}{m=1}{\odot}(d\mathfrak{n}_{H_{\dist}+2}\boxplus \widetilde{\beta}_{\dist})\right)\right),
\] 
and
\[
\dim(\mathcal{D}(\Psi_{1,\varepsilon})) = (H_{\dist}+1)\sum_{i=1}^M \overline{N}_i  + H_g + 2.
\]
In addition
\begin{equation}\label{eq:normpsi1eps}
\begin{aligned}
	\vertiii{\mathcal{D}(\Psi_{1,\varepsilon})} &\leq \sum_{i=1}^{M} \max\{\vertiii{\mathcal{D}(\Phi_g)},2d+\vertiii{\mathcal{D}(\Phi_{\dist})}\}\\
	&\leq M(\vertiii{\mathcal{D}(\Phi_g)} + 2d + \vertiii{\mathcal{D}(\Phi_{\dist})}).
\end{aligned}
\end{equation}
Notice fron \eqref{eq:C1C2} that the constants $C_1$ and $C_2$ are bounded by a multiple of $|D|^{\frac 1q} d^p$. Therefore, by choice of $M$,
\begin{equation}\label{eq:Mfinalg}
	M \leq B_1 |D|^{\frac{s}{q(s-1)}}d^{\frac{ps}{s-1}}\varepsilon^{-\frac{s}{s-1}},
\end{equation}
where $B_1>0$ is a generic constant. With \eqref{eq:Mfinalg} and the bound of $C_1$ and $C_2$, we have that $C_5$ defined in \eqref{C5} is bounded by a multiple of
\[
|D|^{\frac{1}{q}\left(1+p+\frac{ps}{s-1}\right)}d^{p+p^2+\frac{p^2s}{s-1}}\varepsilon^{-\frac{ps}{s-1}}.
\]
By the choice of $\delta_g$ we have, for $B_2>0$ a generic constant that 
\begin{equation}\label{eq:deltagfinalg}
	\delta_g^{-a} \leq B_2 |D|^{\frac{a}{q}\left(1+p+\frac{ps}{s-1}\right)}d^{ap+ap^2+\frac{ap^2s}{s-1}}\varepsilon^{-a-\frac{aps}{s-1}}.
\end{equation}
For $\delta_{\dist}$ we estimate $\log(\delta_{\dist}^{-1})$ as indicates Assumption \ref{Sup:D}. By the choice of $\delta_{\dist}$ and properties of $\log$ function, we have that
\[
\log(\delta_{\dist}^{-1}) \leq 5|D|^{\frac 1q}L_g \varepsilon^{-1} +(M^{\frac 1s}C_1 + M(C_2 + C_4))(1 + M^{\frac 1s}C_1 + M(C_2 + C_3)).
\]
Therefore
\begin{equation}\label{eq:deltadistfinalg}
	\lceil \log(\delta_{\dist}^{-1}) \rceil^{a} \leq B_3 |D|^{\frac{2a}{q}\left(1+\frac{s}{s-1}\right)}d^{2ap\left(1+\frac{s}{s-1}\right)}\varepsilon^{-a-\frac{2as}{s-1}},
\end{equation}
where $B_3>0$ is a generic constant. Assumptions \ref{Sup:g} and \ref{Sup:D}, in addition with \eqref{eq:normpsi1eps} implies that
\[
\vertiii{\mathcal{D}(\Psi_{1,\varepsilon})} \leq B_4 d^b M (\delta_g^{-a}+\lceil\log(\delta_{\dist}^{-1})\rceil^{a}),
\]
where $B_4>0$ is a generic constant. Therefore, from \eqref{eq:Mfinalg}, \eqref{eq:deltagfinalg} and \eqref{eq:deltadistfinalg} we conclude that there exists $\widetilde{B}>0$ such that
\[
\vertiii{\mathcal{D}(\Psi_{1,\varepsilon})} \leq \widetilde{B}|D|^{\frac{1}{q}\left(2a + ap+\frac{s}{s-1}(1+2a+ap)\right)} d^{b+2ap+ap^2+\frac{ps}{s-1}(1+2a + ap)}\varepsilon^{-a-\frac{s}{s-1}(1+2a+ap)}.
\]
Note also that this implies from Remark \ref{rem:CoD} that $\Psi_{1,\varepsilon}$ overcomes the curse of dimensionality. This completes the proof of Proposition \ref{Prop:homo}.

\section{Approximation of solutions of the Fractional Dirichlet problem using DNNs: the source case}\label{Sect:7}

\subsection{Non-homogeneous Fractional Laplacian} In the previous subsection we have proved the the solution \eqref{eq:2.2} of the fractional Dirichlet Problem without source can be approximated by a ReLu DNN. In this subsection we focus in the term
\begin{equation}\label{eq:7.1}
\E_x \left[\sum_{n=1}^{N} r_n^{\alpha} \kappa_{d,\alpha} \E^{(\mu)}\left[f\left(X_{\mathcal{I}(n-1)}+r_n \cdot\right)\right]\right].
\end{equation}
We will prove that \eqref{eq:7.1} can be approximated by a ReLu DNN that does not suffer of the curse of dimensionality. Notice that \eqref{eq:7.1} corresponds to the extra term in the solution \eqref{u(x)} of the fractional Dirichlet Problem with source \eqref{eq:1.1}. In order to do this approximation, the following assumption will be introduced
\begin{ass}\label{Sup:f}
	Let $d \geq 2$. Let $f: D \to \R$  a function satisfying \eqref{Hf0}. Let $\delta_f \in (0,1)$, $a,b\geq 1$ and $B>0$. Then there exists a ReLu DNN $\Phi_f \in \textbf{N}$ with
	\begin{enumerate}
		\item $\mathcal{R}(\Phi_f):D \to \R$ is $\widetilde{L}_f$-Lipschitz continuous, $\widetilde{L}_f>0$, and
		\item The following are satisfied:
		\begin{align}
			|f(x) - \left(\mathcal{R}(\Phi_f)\right)(x)| &\leq \delta_f, \qquad x \in D. \tag{Hf-1} \label{H5}\\
			\vertiii{\mathcal{D}(\Phi_f)} &\leq Bd^b\delta_f^{-a}. \tag{Hf-2} \label{H6} 
		\end{align} 
	\end{enumerate}
\end{ass}

\begin{rem}\label{rem:Rf}
If $\Phi_f$ satisfies Assumptions \ref{Sup:f}, then it holds for all $x \in D$ that
\[
	|\left(\mathcal{R}(\Phi_f)\right)(x)| \leq |f(x)-\left(\mathcal{R}(\Phi_f)\right)(x)|+|f(x)| \leq \delta_f + \norm{f}_{L^{\infty}(D)}.
\]
Then
\begin{equation}\label{eq:Phif}
	\norm{\mathcal{R}(\Phi_f)}_{L^{\infty}(D)} \leq \delta_f + \norm{f}_{L^{\infty}(D)}.
\end{equation}
\end{rem}

The main result of this section is the following proposition, that ensures the existence of a ReLu DNN such that \eqref{eq:7.1} is well approximated
\begin{prop}\label{Prop:6p1}
	Let $\alpha \in (1,2)$, $L_f>0$ and
	\begin{equation}\label{condiciones_f}
		\hbox{ $p,s \in (1,\alpha)$ such that $s < \frac{\alpha}{p}$, \quad and \quad $q \in \left[s, \frac{\alpha}{p} \right)$. }
	\end{equation}
	Suppose that $f$ is a function satisfying \eqref{Hf0} and Assumptions \ref{Sup:f}. Suppose additionally that $D$ satisfies Assumptions \ref{Sup:D}.
	
	\medskip 
%	\begin{itemize}
%		%\item For any $\delta_f \in (0,1)$, the function $f$ can be approximated by a ReLu DNN $\Phi_f$ satisfying \eqref{H5} y \eqref{H6}. 
%		%\medskip
%		
%	\end{itemize}
	Then  for all $\widetilde{\varepsilon} \in (0,1)$, there exists a ReLu DNN $\Psi_{2,\widetilde{\varepsilon}}$ such that
	\begin{enumerate}
		\item Proximity in $L^q(D)$: %there exists $\delta_f$ small enough such that
		\begin{equation}\label{eq:prop_f}
			\left(\int_{D} \left| \E_{x}\left[\sum_{n=1}^{N} r_n^{\alpha} V_1(0,f(X_{\mathcal{I}(n-1)}+r_n \cdot))\right] - \left(\mathcal{R}(\Psi_{2,\widetilde{\varepsilon}})\right)(x)  \right|^q\right)^{\frac{1}{q}} \leq \widetilde{\varepsilon}.
		\end{equation}
		\item Realization: $\mathcal{R}(\Psi_{2,\widetilde{\varepsilon}})$ has the following structure: there exist $M_1, M_2 \in \N$, $\overline{N}_{i} \in \N$, $Y_{i,n}$ i.i.d copies of $X_{\sigma_{B(0,1)}}$ under $\PP_0$, $v_{i,j,n}$ i.i.d copies with law $\mu$ under $B(0,1)$, for $i =1,...,M_1$, $j=1,...,M_2$, $n=1,...,\overline{N}_{i}$ such that for all $x \in D$,
		\begin{equation}
			\begin{aligned}
				\left(\mathcal{R}(\Psi_{2,\widetilde{\varepsilon}})\right)(x) = \frac{1}{M_1} \sum_{i=1}^{M_1} \sum_{n=1}^{\overline{N}_i} \kappa_{d,\alpha}\left(\mathcal{R}(\Upsilon)\right)\Big( \left(\mathcal{R}(\Phi_{\alpha}) \circ \mathcal{R}(\Phi_r^{i,n})\right)(x), \left(\mathcal{R}(\Phi_{f}^{i,n})\right)(x)\Big),
			\end{aligned}
		\end{equation}
		where for all $y \in D$
		\begin{equation}
			\left(\mathcal{R}(\Phi_{r}^{i,n})\right)(y)=\left(\mathcal{R}(\Phi_{\dist}) \circ \mathcal{R}(\Phi_{i,n-1})\right)(y),
		\end{equation}
		\begin{equation}
			\left(\mathcal{R}(\Phi_f^{i,n})\right)(y) = \frac{1}{M_2} \sum_{j=1}^{M_2} \left(\mathcal{R}(\Phi_f) \circ \left(\mathcal{R}(\Phi_{i,n}) + v_{i,j,n} \mathcal{R}(\Phi_{r}^{i,n})\right)\right)(y),
		\end{equation}
		and $\mathcal{R}(\Phi_{i,n})$ is a Relu DNN that approximates $X_{\mathcal{I}(n)}^{i}$, for $i=1,...,M_1$, $n = 1,...,\overline{N}_i$.
		\item Bounds: there exists $\widetilde{B}>0$ such that
		\begin{equation}
			\vertiii{\mathcal{D}(\Psi_{2,\widetilde{\varepsilon}})} \leq \widetilde{B} |D|^{\frac{1}{q}\left(1 +2a+\frac{2s}{s-1}(1+a)\right)}d^b\widetilde{\varepsilon}^{-a-\frac{2s}{s-1}(1+a)}.
		\end{equation}
	\end{enumerate}
\end{prop}

\subsection{Proof of Proposition \ref{Prop:6p1}: existence} As in the proof of Proposition \ref{Prop:homo}, this proof will be divided in several steps. Let $s,p \hbox{ and } q$ as in \eqref{condiciones_f}.
	
	\medskip 
	
	\noindent
	{\bf Step 1.}
	Let $(\rho_n)_{n \in \N}$ the WoS process starting at $x \in D$. Recall that for all $n=1,...,N$ the process $X_{\mathcal{I}(n)}$ depends of the point $x \in D$ and $n$ copies of $X_{\sigma_{B(0,1)}}$, namely
	\begin{equation}\label{eq:X_Inf}
	X_{\mathcal{I}(n)} = X_{\mathcal{I}(n)}(x,Y_1,...,Y_n),
	\end{equation}
	where $Y_{k}$, $k=1,...,n$ are i.i.d. copies of $X_{\sigma_{B(0,1)}}$ under $\PP_0$. Let $M_1 \in \N$. Consider $M_1$ copies of $X_{\mathcal{I}(n)}$ starting at $x \in D$, as described in \eqref{eq:X_Inf}. We denote such copies as
	\[
	X^i_{\mathcal{I}(n)} = X_{\mathcal{I}(n)}^i(x,Y_{i,1},...,Y_{i,n}),
	\]
	where $Y_{i,k}$, $i=1,...,M_1$, $n=1...,N_i$, $k=1,...,n$ are i.i.d. copies of $X_{\sigma_{B(0,1)}}$ under $\PP_0$, and each $N_i$ is an i.i.d. copy of $N$. Recall that $N_i$ not necessarily be the same (as a random variable).
	\medskip
	
	Let $M_2 \in \N$. for all $n=1,...,N$ let $(v_{j,n})_{j=1}^{M_2}$ be $M_2$ copies of a random variable $v$ with distribution $\mu$ over $B(0,1)$. For all $n=1,...,N$ and $\chi \in L^2(B(0,1),\mu)$ define the Monte Carlo operator
	\begin{equation}\label{eq:EM2}
		E_{M_2}^{n}(\chi(\cdot)) = \frac{1}{M_2} \sum_{j=1}^{M_2} \chi(v_{j,n}),
	\end{equation}
	and we will refer to $E_{M_2}$ when the copies of $v$ in \ref{eq:EM2} do not depend on $n$. Additionally define the operator
	\begin{equation}\label{eq:EM1}
		E_{M_1}(x) = \frac{1}{M_1} \sum_{i=1}^{M_1} \sum_{n=1}^{N_i} (r^i_{n})^{\alpha} \kappa_{d,\alpha} E_{M_2} \left(\left(\mathcal{R}(\Phi_f)\right)\left(X_{\mathcal{I}(n-1)}^i + r_n^i \cdot\right)\right),
	\end{equation}
	where
	\begin{equation}\label{eq:rin}
	r_{i,n} = \dist\left(X_{\mathcal{I}(n-1)}^i,\partial D\right),
	\end{equation}
	and $\mathcal{R}(\Phi_f)$ denotes the realization as a Lipschitz continuous funtion of the DNN $\Phi_f \in \textbf{N}$ that approximates $f$ in Assumption \eqref{Sup:f}. Note that $E_{M_1}$ not necessarily be a DNN.
	\medskip
	
	We want to establish suitable bounds of the difference between \eqref{eq:7.1} and the operator $E_{M_1}(x)$. For this, in the next step we work for all $n=1,...,N$ with the term
	\[
	E_{M_2}^{n}\left(\left(\mathcal{R}(\Phi_f)\right)\left(X_{\mathcal{I}(n-1)} + r_n \cdot\right)\right).
	\]
 	
	\medskip
	\noindent 
	{\bf Step 2.}
	Notice by Remark \ref{rem:Rf} that for all $n=1,...,N$ that
	\begin{equation}\label{eq:Rfnorm}
		\E^{(\mu)}\left(\left|\left(\mathcal{R}(\Phi_f)\right)\left(X_{\mathcal{I}(n-1)}+r_n\cdot\right)\right|\right) \leq \delta_f + \norm{f}_{L^{\infty}(D)}.
	\end{equation}
	Then for all $j=1,...,M_2$, $\left(\mathcal{R}(\Phi_f)\right)\left(X_{\mathcal{I}(n-1)}+r_n v_{j,n}\right) \in L^1(B(0,1),\mu)$. For $s$ as in \eqref{condiciones_f} it follows from Corollary \ref{cor:MCq} that for all $q$ as in \eqref{condiciones_f}
	\[
	\begin{aligned}
		&\norm{\E^{(\mu)} \left[\left(\mathcal{R}(\Phi_f)\right)\left(X_{\mathcal{I}(n-1)} + r_n \cdot\right)\right]-E_{M_2} \left(\left(\mathcal{R}(\Phi_f)\right)\left(X_{\mathcal{I}(n-1)} + r_n \cdot\right)\right)}_{L^{q}(B(0,1),\mu)}\\
		&\qquad \leq \frac{2\Theta_{q,s}}{M_2^{1-\frac{1}{s}}} \E^{(\mu)} \left[\left(\mathcal{R}(\Phi_f)\right)\left(X_{\mathcal{I}(n-1)}+r_n\cdot\right)^q\right]^{\frac{1}{q}}.
	\end{aligned}
	\]
	From Remark \ref{rem:Rf} it follows that
	\[
	\E^{(\mu)}\left(\left|\left(\mathcal{R}(\Phi_f)\right)\left(X_{\mathcal{I}(n-1)}+r_n\cdot\right)\right|^q\right)^{\frac 1q} \leq \delta_f + \norm{f}_{L^{\infty}(D)}.
	\]
	Therefore
	\begin{align*}
	& \norm{\E^{(\mu)} \left[\left(\mathcal{R}(\Phi_f)\right)\left(X_{\mathcal{I}(n-1)} + r_n \cdot\right)\right]-E_{M_2} \left(\left(\mathcal{R}(\Phi_f)\right)\left(X_{\mathcal{I}(n-1)} + r_n \cdot\right)\right)}_{L^{q}(B(0,1),\mu)}\\
	& \qquad \leq \frac{2\Theta_{q,s}\left(\delta_f + \norm{f}_{L^{\infty}(D)}\right)}{M_2^{1-\frac{1}{s}}}.
	\end{align*}
	Then for any $n=1,...,N$ there exists $v_{j,n}$, $j=1,...,M_2$, i.i.d random variables with distribution $\mu$ such that
	\[
	\begin{aligned}
	&\left|\E^{(\mu)} \left[\left(\mathcal{R}(\Phi_f)\right)\left(X_{\mathcal{I}(n-1)} + r_n \cdot\right)\right]-\frac{1}{M_2} \sum_{j=1}^{M_2} \left(\mathcal{R}(\Phi_f)\right)\left(X_{\mathcal{I}(n-1)} + r_n v_{j,n}\right)\right|\\
	&\qquad\leq \frac{2\Theta_{q,s}\left(\delta_f + \norm{f}_{L^{\infty}(D)}\right)}{M_2^{1-\frac{1}{s}}}.
	\end{aligned}
	\]
	We redefine $E_{M_2}^{n}$ with the random variables $v_{j,n}$ found for all $n=1,...,N$.
		
	\medskip
	In the two next steps we control the difference between \eqref{eq:7.1} and $E_{M_1}$ with the intermediate term
	\[
	\E_x \left[\sum_{n=1}^{N}r_n^{\alpha}\kappa_{d,\alpha} E^n_{M_2}\left(\mathcal{R}(\Phi_f)\right)\left(X_{\mathcal{I}(n-1)}+r_n \cdot\right)\right],
	\]
	\medskip
	\noindent
	{\bf Step 3.} Define
	\[
\begin{aligned}
	J_3 &= \Bigg\|\E_x\left[\sum_{n=1}^{N}r^{\alpha}_{n} \kappa_{d,\alpha} \E^{(\mu)}\left[\left(\mathcal{R}(\Phi_f)\right)\left(X_{\mathcal{I}(n-1)}+r_n\cdot\right)\right] \right] \\
	& \qquad -\E_x\left[\sum_{n=1}^{N}r^{\alpha}_{n} \kappa_{d,\alpha} E_{M_2}^{n}\left(\left(\mathcal{R}(\Phi_f)\right) \left(X_{\mathcal{I}(n-1)}+r_n \cdot\right)\right) \right]\Bigg\|_{L^q(\Omega,\PP_x)}.
\end{aligned}
\]
	From Step 2 we have
	\[
	J_3 \leq \frac{2\Theta_{q,s}}{M_2^{1-\frac 1s}} \left(\delta_f + \norm{f}_{L^{\infty}(D)} \right) \norm{\E_x \left[\sum_{n=1}^{N} r_n^{\alpha}\kappa_{d,\alpha}\right]}_{L^{q}(\Omega,\PP_x)}.
	\]
	Using Lemma \ref{lem:aux_f} it follows that
	\[
	J_3 \leq \frac{2\Theta_{q,s}}{M_2^{1-\frac 1s}} \left(\delta_f + \norm{f}_{L^{\infty}(D)} \right) \E_x[\sigma_D].
	\]
	
	\medskip
	
	\noindent
	{\bf Step 4.} Define 
	\[
	J_4 = \norm{\E_x \left[\sum_{n=1}^{N}r_n^{\alpha}\kappa_{d,\alpha} E^n_{M_2}\left(\mathcal{R}(\Phi_f)\right)\left(X_{\mathcal{I}(n-1)}+r_n \cdot\right)\right] - E_{M_1}(x)}_{L^{q}(\Omega,\PP_x)}.
	\]
	Recall Remark \ref{rem:Rf}. Then
	\[
	\E_x \left[\left|\sum_{n=1}^{N}r_n^{\alpha}\kappa_{d,\alpha} E^n_{M_2}\left(\mathcal{R}(\Phi_f)\right)\left(X_{\mathcal{I}(n-1)}+r_n \cdot\right)\right|\right] \leq (\delta_f + \norm{f}_{L^{\infty}(D)})\E_x[\sigma_D]<\infty.
	\]
	This implies that
	\[
	\sum_{n=1}^{N}r_{i,n}^{\alpha}\kappa_{d,\alpha} E^n_{M_2}\left(\mathcal{R}(\Phi_f)\right)\left(X^i_{\mathcal{I}(n-1)}+r_{i,n} \cdot\right) \in L^1(\Omega,\PP_x).
	\]
	Then using Corollary \ref{cor:MCq} we have for $s$ as in \eqref{condiciones_f} it holds for all $q$ as in \eqref{condiciones_f} that
	\begin{align}
		J_4 &\leq \frac{2\Theta_{q,s}}{M_1^{1-\frac{1}{s}}} \E_x \left[\left(\sum_{n=1}^{N} r_n^{\alpha}\kappa_{d,\alpha} E^n_{M_2}\left(\left(\mathcal{R}(\phi_f)\right)(X_{\mathcal{I}(n-1)}+r_n\cdot)\right)\right)^q\right]^{\frac{1}{q}}.
	\end{align}
	Remark \ref{rem:Rf} implies that
	\[
	J_4 \leq \frac{2\Theta_{q,s}}{M_1^{1-\frac{1}{s}}}\left(\delta_f + \norm{f}_{L^{\infty}(D)}\right) \E_x \left[\left(\sum_{n=1}^{N}r_{n}^{\alpha} \kappa_{d,\alpha}\right)^q\right]^{\frac{1}{q}}
	\]
	By Lemma \ref{lem:aux_f} and Jensen inequality with $(\cdot)^{\frac{2}{q}}$, $q<2,$ we have
	\[
	\E_x\left[\left(\sum_{n=1}^{N}r_n^{\alpha}\kappa_{d,\alpha}\right)^q\right]^{\frac{1}{q}} \leq \E_x\left[\left(\sum_{n=1}^{N}r_n^{\alpha}\kappa_{d,\alpha}\right)^2\right]^{\frac{1}{2}} \leq \E_x [\sigma_D^2]^{\frac{1}{2}}.
	\]
	Therefore
	\begin{align}
		J_4 \leq \frac{2\Theta_{q,s}}{M_1^{1-\frac{1}{s}}}\left(\delta_f + \norm{f}_{L^{\infty}(D)}\right) \E_x[\sigma_D^2]^{\frac{1}{2}}.
	\end{align}
	
	\medskip
	
	\noindent
	{\bf Step 5.}
	With the bounds obtained in Steps 3 and 4, we have by Minkowski inequality that 
	\begin{equation}\label{eq:step5f}
	\begin{aligned}
		&\norm{\E_x\left[\sum_{n=1}^{N}r^{\alpha}_{n} \kappa_{d,\alpha} \E^{(\mu)}\left[\left(\mathcal{R}(\Phi_f)\right)\left(X_{\mathcal{I}(n-1)}+r_n\cdot\right)\right] \right]-E_{M_1}(x)}_{L^q(\Omega,\PP_x)} \leq J_3 + J_4\\
		&\hspace{3.5cm}\leq 2\Theta_{q,s}\left(\delta_f + \norm{f}_{L^{\infty}(D)}\right)\left(\frac{\E_x[\sigma_D]}{M_2^{1-\frac 1s}}+\frac{\E_x[|\sigma_D|^2]^{\frac{1}{2}}}{M_1^{1-\frac 1s}}\right).
	\end{aligned}
	\end{equation}
	Fubini and \eqref{eq:step5f} implies that
	\begin{equation}\label{eq:cotafP4}
		\begin{aligned}
			&\E_x \left[\int_D \left|\E_x\left[\sum_{n=1}^{N}r_n^{\alpha} \kappa_{d,\alpha} \E^{(\mu)}\left[\left(\mathcal{R}(\Phi_f)\right)\left(X_{\mathcal{I}(n-1)}+r_n\cdot\right)\right]\right]-E_{M_1}(x)\right|^qdx\right] \\
			&\qquad \leq 2^q|D|\Theta_{q,s}^{q} \left(\delta_f + \norm{f}_{L^{\infty}(D)}\right)^q \left(\frac{\E_x[\sigma_D]}{M_2^{1-\frac 1s}} + \frac{\E_x[\sigma_D^2]^\frac{1}{2}}{M_1^{1-\frac 1s}}\right)^{q}
		\end{aligned}
	\end{equation}
	On the other hand side, from hypothesis \eqref{H5} and Lemma \ref{lem:aux_f} one has
	\[
	\begin{aligned}
	&\E_x \left[\int_D\left|\E_x\left[\sum_{n=1}^{N}r^{\alpha}_{n} \kappa_{d,\alpha} \E^{(\mu)}\left[\left(\mathcal{R}(\phi_f)\right)\left(X_{\mathcal{I}(n-1)}+r_n\cdot\right)\right] \right] \right.\right.\\
	&\qquad\left.\left.-\E_x\left[\sum_{n=1}^{N}r^{\alpha}_{n} \kappa_{d,\alpha} \E^{(\mu)}\left[f\left(X_{\mathcal{I}(n-1)}+r_n\cdot\right)\right] \right]\right|^q dx\right] \leq \delta_f^q|D|\E_x[\sigma_D]^q.
	\end{aligned}
	\]
	Therefore, using that $(\cdot)^q$ is a convex function and \eqref{eq:cotafP4} it follows that
	\begin{equation}\label{eq:cotafP42}
	\begin{aligned}
		&\E_x \left[\int_D \left|\E_x\left[\sum_{n=1}^{N}r_n^{\alpha} \kappa_{d,\alpha} \E^{(\mu)}\left[f\left(X_{\mathcal{I}(n-1)}+r_n\cdot\right)\right]\right]-E_{M_1}(x)\right|^qdx\right] \\
		&\quad \leq 2^{q-1}\delta_f^q|D|\E_x[\sigma_D]^q+2^{q-1}2^q|D|\Theta_{q,s}^{q} \left(\delta_f + \norm{f}_{L^{\infty}(D)}\right)^q \left(\frac{\E_x[\sigma_D]}{M_2^{1-\frac 1s}} + \frac{\E_x[\sigma_D^2]^\frac{1}{2}}{M_1^{1-\frac 1s}}\right)^{q}\\
		&\quad \leq 2\delta_f^q|D|\E_x[\sigma_D]^q+2^{q+1}|D|\Theta_{q,s}^{q} \left(\delta_f + \norm{f}_{L^{\infty}(D)}\right)^q \left(\frac{\E_x[\sigma_D]}{M_2^{1-\frac 1s}} + \frac{\E_x[\sigma_D^2]^\frac{1}{2}}{M_1^{1-\frac 1s}}\right)^{q},
	\end{aligned}
	\end{equation}
	where we use that $2^{q-1}<2$ since $q<2$.
	\medskip
	
	\noindent
	{\bf Step 6.}
	In order to bound the following expectation
	\[\E_x \left[\left|\E_x[N] - \frac{1}{M_1} \sum_{i=1}^{M_1} N_i\right|^q\right],
	\]
	we use Corollary \ref{cor:MCq}. Notice by Theorem \ref{teo:N} that for all $x \in D$ there exists a geometric random variable $\Gamma$ with parameter $\widetilde{p} = \widetilde{p}(\alpha,d) > 0$ such that
	\[
	\E_x\left[|N|\right] \leq \E_x\left[\Gamma\right] = \frac{1}{\widetilde{p}} < \infty,
	\]
	and then for all $i \in \{1,...,M_1\}$, $N_i \in L^1(\PP_x,|\cdot|)$. For $s$ as in \eqref{condiciones_f}, Corollary \ref{cor:MCq} implies for all $q$ as in \eqref{condiciones_f} that
	\begin{equation}\label{eq:cotafP5_1}
		\E_x\left[\left|\E_x[N]-\frac{1}{M_1} \sum_{i=1}^{M_1} N_i\right|^q\right]\leq \left( \frac{2\Theta_{q,s}}{M_1^{1-\frac{1}{s}}}\right)^q \E_x \left[|N|^{q}\right] \leq \left( \frac{2\Theta_{q,s}}{M_1^{1-\frac{1}{s}}}\right)^q \E_x \left[|N|^{2}\right],
	\end{equation}
	where we used that $q<2$ and then $\E_x\left[|\cdot|^q\right]\leq \E_x\left[|\cdot|^2\right]$. Recall that
	\begin{equation}\label{eq:cotafN^2}
		\E_x\left[|N|^2\right] \leq \E_x\left[\Gamma^2\right] = \frac{2-\widetilde{p}}{\widetilde{p}^2} < \infty,
	\end{equation}
	and therefore, it holds from \eqref{eq:cotafP5_1} and \eqref{eq:cotafN^2} that
	\begin{equation}\label{eq:cotafP5_2}
		\E_x\left[\left|\E_x[N]-\frac{1}{M_1} \sum_{i=1}^{M_1} N_i\right|^q\right] \leq \left( \frac{2\Theta_{q,s}}{M_1^{1-\frac{1}{s}}}\right)^q \frac{2-\widetilde{p}}{\widetilde{p}^{2}}.
	\end{equation}
	\medskip
	\noindent
	{\bf Step 7.} We want to estimate
	\[
	\E_x \left[\left|\E_x \left[\sum_{n=1}^{N} |Y_{n}|\right] - \frac{1}{M_1} \sum_{i=1}^{M_1} \sum_{n=1}^{N_i} |Y_{i,n}| \right|^{q} \right]
	\]
	As in the previous step, we use the Corollary \ref{cor:MCq}. First of all, it follows from the independence of $\left(Y_{n}\right)_{n=1}^{k}$ and $N$ for fixed $k \in \N$ ($Y_{n}$ and $X$ are independent), and law of total expectation that
	\[
	\begin{aligned}
		\E_x \left[\left|\sum_{n=1}^{N} |Y_{n}|\right|\right] &= \sum_{k \geq 1} \E_0 \left[\left.\left|\sum_{n=1}^{N}|Y_{n}|\right|\right|N=k\right] \PP_x (N = k)\\
		&= \sum_{k \geq 1} \E_0 \left[\left|\sum_{n=1}^{k}|Y_{n}|\right|\right] \PP_x (N = k).
	\end{aligned}
	\]
	Recall that $(Y_{n})_{n=1}^{k}$ are i.i.d. with the same distribution as $X_{\sigma_{B(0,1)}}$. Triangle inequality ensures that
	\[
	\begin{aligned}
		\E_x \left[\left|\sum_{n=1}^{N} |Y_{n}|\right|\right] &\leq \sum_{k\geq1} \sum_{n=1}^{k} \E_{0}\left[|Y_n|\right] \PP_x(N=k)\\
		&= \E_0\left[\left|X_{\sigma_{B(0,1)}}\right|\right] \sum_{k \geq 1} k \PP_x(N=k)\\
		&= K(\alpha,1) \E_x[N].
	\end{aligned}
	\]
	and then for all $i \in \{1,...,M_1\}$, $\sum_{n=1}^{N} |Y_n| \in L^{1}(\PP_x,|\cdot|)$. Moreover, with similar arguments it holds that
	\begin{equation}\label{eq:cotafP6_1}
		\E_x \left[\left|\sum_{n=1}^{N} |Y_{i,n}|\right|^q\right] \leq K(\alpha,q) \E_x[N^{q}].
	\end{equation}
	For $s$ as in \eqref{condiciones_f}, Corollary \ref{cor:MCq} implies for all $q$ as in \eqref{condiciones_f} that
	\[
	\E_x \left[\left|\E_x \left[\sum_{n=1}^{N} |Y_{n}|\right] - \frac{1}{M_1} \sum_{i=1}^{M_1} \sum_{n=1}^{N_i} |Y_{i,n}| \right|^{q} \right] \leq \left( \frac{2\Theta_{q,s}}{M_1^{1-\frac{1}{s}}}\right)^{q} \E_x \left[\left|\sum_{n=1}^{N}|Y_{n}|\right|^q\right].
	\]
	And therefore it follows from \eqref{eq:cotafN^2} and \eqref{eq:cotafP6_1} that
	
	\begin{equation}\label{eq:cotafP6_2}
		\E_x \left[\left|\E_x \left[\sum_{n=1}^{N} |Y_{n}|\right] - \frac{1}{M_1} \sum_{i=1}^{M_1} \sum_{n=1}^{N_i} |Y_{i,n}| \right|^{q} \right] \leq \left( \frac{2\Theta_{q,s}}{M_1^{1-\frac{1}{s}}}\right)^{q} K(\alpha,q) \frac{2-\widetilde{p}}{\widetilde{p}^2}.
	\end{equation}
	\medskip
	\noindent
	{\bf Step 8.} It follows from \eqref{eq:cotafP42}, \eqref{eq:cotafP5_2} and \eqref{eq:cotafP6_2} that
	\begin{equation}
		\begin{aligned}
			&\E_x \left[\int_D \left|\E_x\left[\sum_{n=1}^{N}r_n^{\alpha} \kappa_{d,\alpha} \E^{(\mu)}\left[\left(\mathcal{R}(\Phi_f)\right)\left(X_{\mathcal{I}(n-1)}+r_n\cdot\right)\right]\right]-E_{M_1}(x)\right|^qdx\right. \\
			&\qquad \left.+ \left|\E_x\left[N\right] - \frac{1}{M_1} \sum_{i=1}^{M_1} N_i\right|^q + \left|\E_x\left[\sum_{n=1}^{N}|Y_n|\right]-\frac{1}{M_1}\sum_{i=1}^{M_1}\sum_{n=1}^{N_i} |Y_{i,n}|\right|^q\right]\\
			&\leq 2\delta_f^q|D|\E_x[\sigma_D]^q + 2^{q+1}|D|\Theta_{q,s}^{q} \left(\delta_f + \norm{f}_{L^{\infty}(D)}\right)^q \left(\frac{\E_x[\sigma_D]}{M_2^{1-\frac 1s}} + \frac{\E_x[\sigma_D^2]^\frac{1}{2}}{M_1^{1-\frac 1s}}\right)^{q} \\
			& \qquad + \left( \frac{2\Theta_{q,s}}{M_1^{1-\frac{1}{s}}}\right)^{q} (1+K(\alpha,q)) \frac{2-\widetilde{p}}{\widetilde{p}^2} =: \hbox{error}_f^q.
		\end{aligned}
	\end{equation}
	Using now that $\E(Z)\leq c<\infty$ we have the following result
	\begin{lem}\label{lem:sacada2}
	This implies that there exists $\overline{N}_i \in \N$, $Y_{i,n}$ i.i.d. copies of $X_{\sigma_{B(0,1)}}$, $v_{i,j,n}$ i.i.d. random variables with law $\mu$, $i=1,...,M_1$, $j=1,...,M_2$, $n=1,...,\overline{N}_i$ such that 
	\begin{equation}
		\begin{aligned}
			&\int_D \left|\E_x\left[\sum_{n=1}^{N}r_n^{\alpha} \kappa_{d,\alpha} \E^{(\mu)}\left[f\left(X_{\mathcal{I}(n-1)}+r_n\cdot\right)\right]\right]-E_{M_1}(x)\right|^qdx \\
			&\qquad + \left|\E_x\left[N\right] - \frac{1}{M_1} \sum_{i=1}^{M_1} \overline{N}_i\right|^q + \left|\E_x\left[\sum_{n=1}^{N}|Y_n|\right]-\frac{1}{M_1}\sum_{i=1}^{M_1}\sum_{n=1}^{\overline{N}_i} |Y_{i,n}|\right|^q\\
			&\qquad\leq \hbox{error}_f^q.
		\end{aligned}
	\end{equation}
	\end{lem}
	Here, $E_{M_1}$ will be redefined from \eqref{eq:EM1} according to the copies found in Lemma \ref{lem:sacada2}, with the Monte Carlo operator $E_{M_2}^{i,n}$ defined from the copies $(v_{i,j,n})_{j=1}^{M_2}$.
	
	\medskip
	\noindent
	{\bf Step 9.} As similar in Step 8 of Proposition 6.1, we can see that for all $i=1,...,M_1$ $n=1,...,\overline{N}_i$ $X^{i}_{\mathcal{I}(n-1)}$ can be approximated by a ReLu DNN $\Phi_{i,n-1} \in \textbf{N}$ which satisfies for all $x \in D$
	\[
	\left|X_{\mathcal{I}(n-1)}^{i}-\left(\mathcal{R}(\Phi_{i,n-1})\right)(x)\right| \leq \delta_{\dist}\left(\sum_{n=1}^{\overline{N}_i}|Y_{i,n}| + 1\right)^{\overline{N}_i}.
	\]
	We can find now a DNN that approximates $r_{i,n}$ in \eqref{eq:rin}. Indeed, define $\Phi_{r}^{i,n} \in \textbf{N}$ as follows:
	\[
	\left(\mathcal{R}(\Phi_r^{i,n})\right)(x) = \left(\mathcal{R}(\Phi_{\dist}) \circ \mathcal{R}(\Phi_{i,n-1})\right)(x),
	\] 
	valid for $x \in D$. In Section \ref{Sect:7p3} we show that $\Phi_{r}^{i,n}$ is in fact a ReLu DNN. For $x \in D$ we have by triangle inequality that
	\[
	\begin{aligned}
		\left|r_{i,n} - \left(\mathcal{R}(\Phi_r^{i,n})\right)(x)\right|&\leq \left|r_{i,n}-\dist\left(\left(\Phi_{i,n-1}\right)(x),\partial D\right)\right|\\
		&\qquad + \left|\dist\left(\left(\Phi_{i,n-1}\right)(x),\partial D\right) - \left(\mathcal{R}(\Phi_r^{i,n})\right)(x)\right|
	\end{aligned}
	\]
	Hypothesis \eqref{HD-1} and the fact that the function $x \mapsto \dist(x,\partial D)$ is 1-Lipschitz implies
	\[
	\begin{aligned}
		\left|r_{i,n} - \left(\mathcal{R}(\Phi_r^{i,n})\right)(x)\right|&\leq \left|X_{\mathcal{I}(n-1)}^{i}-\left(\mathcal{R}(\Phi_{i,n-1})\right)(x)\right| + \delta_{\dist}\\
		&\leq \delta_{\dist}\left(\sum_{n=1}^{\overline{N}_i}|Y_{i,n}| + 1\right)^{\overline{N}_i} + \delta_{\dist}.
	\end{aligned}
	\]
	
	\medskip
	\noindent
	{\bf Step 10.} We will find now a DNN that approximates
	\begin{equation}\label{EM2in}
	E_{M_2}^{i,n}\left( \left(\mathcal{R}(\Phi_f)\right)\left(X_{\mathcal{I}(n-1)}^{i}+r_{i,n}\cdot\right)\right).
	\end{equation}
	Define the DNN $\Phi_f^{i,n} \in \textbf{N}$ as follows: for $x \in D$
	\[
	\left(\mathcal{R}(\Phi_f^{i,n})\right)(x) = \frac{1}{M_2} \sum_{j=1}^{M_2} \left(\mathcal{R}(\Phi_f) \circ \left(\mathcal{R}(\Phi_{i,n-1}) + v_{i,j,n} \mathcal{R}(\Phi_{r}^{i,n})\right)\right)(x).
	\]
	In Section \ref{Sect:7p3} we will prove that $\Phi_{f}^{i,n}$ is a ReLu DNN. We now use the assumption that $\mathcal{R}(\Phi_f)$ is a $\widetilde{L}_{f}$-Lipschitz function to obtain
	\[
	\begin{aligned}
		&\left|E_{M_2}^{i,n}\left( \left(\mathcal{R}(\Phi_f)\right)\left(X_{\mathcal{I}(n-1)}^{i}+r_{i,n}\cdot\right)\right)-\left(\mathcal{R}(\Phi_f^{i,n})\right)(x)\right| \\
		&\qquad \leq \frac{\widetilde{L}_f}{M_2} \sum_{j=1}^{M_2} \left(\left|X_{\mathcal{I}(n-1)}^{i} - \left(\mathcal{R}(\Phi_{i,n-1})\right)(x)\right| + \left|v_{i,j,n}\right| \left|r_{i,n} -\left(\mathcal{R}(\Phi_r^{i,n})\right)(x)\right|\right).
	\end{aligned}
	\] 
	Notice that for all $i=1,...,M_1$, $j=1,...,M_2$, $n=1,...,\overline{N}_i$ one has $|v_{i,j,n}|\leq 1$ ($v_{i,j,n}$ is a random variable on $B(0,1)$). Therefore, it follows from Step 9 that
	\[
	\begin{aligned}
	& \left|E_{M_2}^{i,n}\left( \left(\mathcal{R}(\Phi_f)\right)\left(X_{\mathcal{I}(n-1)}^{i}+r_{i,n}\cdot\right)\right)-\left(\mathcal{R}(\Phi_f^{i,n})\right)(x)\right| \\
	&~{} \qquad  \leq \widetilde{L}_f\delta_{\dist} \left(2\left(\sum_{n=1}^{\overline{N}_i} |Y_{i,n}| + 1\right)^{\overline{N}_i}+1\right).
	\end{aligned} 
	\]
	\medskip
	\noindent 
	{\bf Step 11.} We want to approximate the multiplication between $r_{i,n}^{\alpha}$ and \eqref{EM2in}. For all $i=1,...,M_1$, $n=1,...,\overline{N}_i$ define the DNN $\Upsilon_{i,n} \in \textbf{N}$ as
	\[
	\left(\mathcal{R}(\Upsilon_{i,n})\right)(x) = \left(\mathcal{R}(\Upsilon)\right)\left(\left(\mathcal{R}(\Phi_{\alpha})\circ \mathcal{R}(\Phi_{r}^{i,n})\right)(x),\left(\mathcal{R}(\Phi_f^{i,n})\right)(x)\right),
	\]
	valid for $x \in D$. In Section \ref{Sect:7p3} we show that $\Upsilon_{i,n}$ is a ReLu DNN. Note by triangle inequality that for all $x \in D$
	\[
	\begin{aligned}
		&\left|r_{i,n}^{\alpha}E_{M_2}^{i,n}\left( \left(\mathcal{R}(\Phi_f)\right)\left(X_{\mathcal{I}(n-1)}^{i}+r_{i,n}\cdot\right)\right) - \left(\mathcal{R}(\Upsilon_{i,n})\right)(x)\right| \\
		&\qquad\leq \left|r_{i,n}^{\alpha}\Big(E_{M_2}^{i,n}\left( \left(\mathcal{R}(\Phi_f)\right)\left(X_{\mathcal{I}(n-1)}^{i}+r_{i,n}\cdot\right)\right)-\left(\mathcal{R}(\Phi_f^{i,n})\right)(x)\Big)\right|\\
		&\qquad \quad+ \left|\left(\mathcal{R}(\Phi_f^{i,n})\right)(x) \left(r_{i,n}^{\alpha} - \left(\mathcal{R}(\Phi_{\alpha})\circ\mathcal{R}(\Phi_r^{i,n})\right)(x)\right)\right|\\
		&\qquad \quad+\left|\left(\mathcal{R}(\Phi_f^{i,n})\right)(x)\left(\mathcal{R}(\Phi_{\alpha})\circ\mathcal{R}(\Phi_r^{i,n})\right)(x) - \left(\mathcal{R}(\Upsilon_{i,n})\right)(x)\right|
	\end{aligned}
	\]
	For all $i=1,...,M_1$, $n=1,...,M_1$ one has $r_n^{i} < \diam(D)$. From Step 9, the first term can be bounded as
	\[
	\begin{aligned}
		&\left|r_{i,n}^{\alpha}\left(E_{M_2}^{i,n}\left( \left(\mathcal{R}(\Phi_f)\right)\left(X_{\mathcal{I}(n-1)}^{i}+r_{i,n}\cdot\right)\right)-\left(\mathcal{R}(\Phi_f^{i,n})\right)(x)\right)\right| \\
		&\qquad\leq \diam(D)^{\alpha}\widetilde{L}_f\delta_{\dist} \left(2\left(\sum_{n=1}^{\overline{N}_i} |Y_{i,n}| + 1\right)^{\overline{N}_i}+1\right).
	\end{aligned}
	\]
	For the second term of the inequality, note that
	\[
	\left|\left(\mathcal{R}(\Phi_f^{i,n})\right)(x)\right| \leq \norm{\mathcal{R}(\Phi_f)}_{L^{\infty}(D)} \leq \delta_f + \norm{f}_{L^{\infty}(D)}.
	\]
	Also, by triangle inequality
	\[
	\begin{aligned}
	\left|r_{i,n}^{\alpha} - \left(\mathcal{R}(\Phi_{\alpha})\circ\mathcal{R}(\Phi_r^{i,n})\right)(x)	\right| &\leq \left|r_{i,n}^{\alpha} - \left(\mathcal{R}(\Phi_{\alpha})\right)(r_{i,n})\right| \\
	&\quad+ \left|\left(\mathcal{R}(\Phi_{\alpha})\right)(r_{i,n})- \left(\mathcal{R}(\Phi_{\alpha})\circ\mathcal{R}(\Phi_r^{i,n})\right)(x)\right|.
	\end{aligned}
	\]
	From the Hypothesis \ref{HD-3} and the fact that $\mathcal{R}(\Phi_{\alpha})$ is $L_{\alpha}$-Lipschitz one has
	\[
	\left|r_{i,n}^{\alpha} - \left(\mathcal{R}(\Phi_{\alpha})\circ\mathcal{R}(\Phi_r^{i,n})\right)(x)	\right| \leq \delta_{\alpha} + L_{\alpha} \left|r_{i,n} - \left(\mathcal{R}(\Phi_r^{i,n})\right)(x)\right|
	\]
	And by Step 9 it follows that
	\[
	\begin{aligned}
		&\left|\left(\mathcal{R}(\Phi_f^{i,n})\right)(x) \left(r_{i,n}^{\alpha} - \left(\mathcal{R}(\Phi_{\alpha})\circ\mathcal{R}(\Phi_r^{i,n})\right)(x)\right)\right| \\
		&\quad\leq \left(\delta_f + \norm{f}_{L^{\infty}(D)}\right)\left(\delta_{\alpha} + L_{\alpha} \delta_{\dist} \left(\left(\sum_{n=1}^{\overline{N}_i}|Y_{i,n}|+1\right)^{\overline{N}_i}+1\right)\right).
	\end{aligned}
	\]
	Finally, by Lemma \ref{lem:DNN_mult} for all $\delta_{\Upsilon} \in \left(0,\frac{1}{2}\right)$ the third term can be bounded by
	\[
	\left|\left(\mathcal{R}(\Phi_f^{i,n})\right)(x)\left(\mathcal{R}(\Phi_{\alpha})\circ\mathcal{R}(\Phi_r^{i,n})\right)(x) - \left(\mathcal{R}(\Upsilon_{i,n})\right)(x)\right| \leq \delta_{\Upsilon}.
	\]
	with $\kappa$ from the Lemma \ref{lem:DNN_mult} equal to
	\[
	\kappa = \max\left\{1+\norm{f}_{L^{\infty}(D)},1+L_{\alpha}\left(\left(\sum_{i=1}^{\overline{N}_i}|Y_{i,n}|+1\right)^{\overline{N}_i}+1\right)+\diam(D)^{\alpha}\right\}
	\]
	Therefore
	\begin{equation}\label{eq:step11f}
	\begin{aligned}
		&\left|r_{i,n}^{\alpha}E_{M_2}^{i,n}\left( \left(\mathcal{R}(\Phi_f)\right)\left(X_{\mathcal{I}(n-1)}^{i}+r_{i,n}\cdot\right)\right) - \left(\mathcal{R}(\Upsilon_{i,n})\right)(x)\right| \\
		&\quad\leq \diam(D)^{\alpha}L_f\delta_{\dist} \left(2\left(\sum_{n=1}^{\overline{N}_i} |Y_{i,n}| + 1\right)^{\overline{N}_i}+1\right)\\
		&\quad \quad+\left(\delta_f + \norm{f}_{L^{\infty}(D)}\right)\left(\delta_{\alpha} + L_{\alpha} \delta_{\dist} \left(\left(\sum_{n=1}^{\overline{N}_i}|Y_{i,n}|+1\right)^{\overline{N}_i}+1\right)\right) + \delta_{\Upsilon}.\\
		&\quad\leq \delta_{\Upsilon} + \delta_{\alpha}\left(\delta_f + \norm{f}_{L^{\infty}(D)}\right) \\
		&\quad \quad+ \delta_{\dist} \left(\diam(D)^{\alpha}L_f + L_{\alpha}\left(\delta_f + \norm{f}_{L^{\infty}(D)}\right)\right) \left(2\left(\sum_{n=1}^{\overline{N}_i} |Y_{i,n}| + 1\right)^{\overline{N}_i}+1\right).
	\end{aligned}
	\end{equation}

	\noindent
	{\bf Step 12.}
	For $\widetilde{\varepsilon} \in (0,1)$ define the DNN $\Psi_{2,\widetilde{\varepsilon}} \in \textbf{N}$ as follows: for any $x \in D$
	\[
	\left(\mathcal{R}(\Psi_{2,\widetilde{\varepsilon}})\right)(x) = \frac{1}{M_1} \sum_{i=1}^{M_1} \sum_{n=1}^{\overline{N}_i} \kappa_{d,\alpha} \left(\mathcal{R}\left(\Upsilon_{i,n}\right)\right)(x).
	\]
	This is the requested DNN. See Section \ref{Sect:7p3} for the proof that $\Psi_{2,\widetilde{\varepsilon}}$ is indeed a ReLu DNN. By triangle inequality and \eqref{eq:step11f} we have
	\[
	\begin{aligned}
		&\left|E_{M_1}(x) - \left(\mathcal{R}(\Psi_{2,\widetilde{\varepsilon}})\right)(x)\right| \\
		&\quad\leq \frac{\kappa_{d,\alpha}}{M_1}\sum_{i=1}^{M_1} \sum_{n=1}^{\overline{N}_i} \left|r_{i,n}^{\alpha}E_{M_2}^{i,n}\left( \left(\mathcal{R}(\Phi_f)\right)\left(X_{\mathcal{I}(n-1)}^{i}+r_{i,n}\cdot\right)\right) - \left(\mathcal{R}(\Upsilon_{i,n})\right)(x)\right|.\\
		&\quad \leq \frac{\kappa_{d,\alpha}}{M_1} \left(\delta_{\Upsilon} + \delta_{\alpha} \left(\delta_f + \norm{f}_{L^{\infty}(D)}\right)\right) \sum_{i=1}^{M_1} \overline{N}_i \\
		&\quad \quad + \frac{\kappa_{d,\alpha}}{M_1} \delta_{\dist} \left(\diam(D)^{\alpha}L_f + L_{\alpha}\left(\delta_f + \norm{f}_{L^{\infty}(D)}\right)\right) \sum_{i=1}^{M_1} \overline{N}_i\left(2\left(\sum_{n=1}^{\overline{N}_i} |Y_{i,n}| + 1\right)^{\overline{N}_i}+1\right).
	\end{aligned}
	\]
	Therefore
	\[
	\begin{aligned}
		&\left|E_{M_1}(x) - \left(\mathcal{R}(\Psi_{2,\widetilde{\varepsilon}})\right)(x)\right| \\
		&\quad \leq \frac{\kappa_{d,\alpha}}{M_1} \left(\delta_{\Upsilon} + \delta_{\alpha} \left(\delta_f + \norm{f}_{L^{\infty}(D)}\right)\right) \sum_{i=1}^{M_1} \overline{N}_i \\
		&\quad \quad + \kappa_{d,\alpha} \delta_{\dist} \left(\diam(D)^{\alpha}L_f + L_{\alpha}\left(\delta_f + \norm{f}_{L^{\infty}(D)}\right)\right) \left(\sum_{i=1}^{M_1} \overline{N}_i\right)\ell,
	\end{aligned}
	\]
	with $\ell:=\left(2\left(\sum_{i=1}^{M_1}\sum_{n=1}^{\overline{N}_i} |Y_{i,n}| + 1\right)^{\sum_{i=1}^{M_1}\overline{N}_i}+1\right).$

	\medskip
	\noindent
	{\bf Step 13.} We want to bound error$_f$. Notice that
	\[
	\begin{aligned}
		\hbox{error}_f &\leq 2^{\frac{1}{q}}|D|^{\frac{1}{q}}\delta_f \E_x[\sigma_D] +  2^{1+\frac{1}{q}}|D|^{\frac{1}{q}} \Theta_{q,s} \left(\delta_f + \norm{f}_{L^{\infty}(D)}\right) \left(\frac{\E_x\left[\sigma_D\right]}{M_2^{1-\frac{1}{s}}} + \frac{\E_x\left[\sigma_D^2\right]^{\frac{1}{2}}}{M_1^{1-\frac{1}{s}}}\right)\\
		&\qquad + \frac{2^{1+\frac{1}{q}}\Theta_{q,s}}{M_1^{1-\frac{1}{s}}} \left(1 + K(\alpha,q)\right)^{\frac{1}{q}}\left(\frac{2-\widetilde{p}}{\widetilde{p}^2}\right)^{\frac{1}{q}}.
	\end{aligned}
	\]
	Consider now $M \in \N$ and let $M = M_1 = M_2$. Define the constant $\widetilde{C}_1$ as
	\[
	\widetilde{C}_1 = 2^{1+\frac{1}{q}}\Theta_{q,s}\left(|D|^{\frac{1}{q}} \left(1 + \norm{f}_{L^{\infty}(D)}\right)\left(\E_x\left[\sigma_D\right]+\E_x\left[\sigma_D^2\right]^{\frac{1}{2}}\right)+ \left(1 + K(\alpha,q)\right)^{\frac{1}{q}} \left(\frac{2-\widetilde{p}}{\widetilde{p}^{2}}\right)^{\frac{1}{q}}\right),
	\]
	and the constant $\widetilde{C}_2$ as
	\[
	\widetilde{C}_2 = 2^{\frac{1}{q}}|D|^{\frac{1}{q}} \E_x\left[\sigma_D\right].
	\]
	Therefore
	\begin{equation}\label{eq:error_f}
	\hbox{error}_f \leq \frac{\widetilde{C}_1}{M^{1-\frac{1}{s}}} + \widetilde{C}_2 \delta_f.
	\end{equation}
	In addition
	\[
	\sum_{i=1}^{M} \sum_{n=1}^{\overline{N}_i} |Y_{i,n}| \leq M \left( \hbox{error}_f + \E_x \left[\sum_{n=1}^{N} |Y_{i,n}|\right]\right) \leq M \left( \hbox{error}_f + K(\alpha,1) \frac{1}{\widetilde{p}} \right).
	\]
	Recall that $C_3 = K(\alpha,1) \frac{1}{\widetilde{p}}$. Then
	\[
	\begin{aligned}
	\sum_{i=1}^{M} \sum_{n=1}^{\overline{N}_i} |Y_{i,n}| &\leq M^{\frac{1}{s}}\widetilde{C}_1 +  M\left(\delta_f\widetilde{C}_2+C_3\right). \\
	&\leq M^{\frac 1s}\widetilde{C}_1 + M(\widetilde{C}_2 + C_3).
	\end{aligned}
	\]
	Recall that $C_4 = \frac{1}{\widetilde{p}}$, therefore
	\[
	\begin{aligned}
	\sum_{i=1}^{M} \overline{N}_i \leq M (\hbox{error}_f + \E_x[N]) &\leq M^{\frac{1}{s}} \widetilde{C}_1 +M \left(\delta_f\widetilde{C}_2+C_4\right)\\
	&\leq M^{\frac 1s}\widetilde{C}_1 + M(\widetilde{C}_2+C_4).
	\end{aligned}
	\]
	\medskip
	\noindent
	From the Step 12 and the estimates of this step it follows that
	\[
	\begin{aligned}
		&\left|E_{M_1}(x) - \left(\mathcal{R}(\Psi_{2,\widetilde{\varepsilon}})\right)(x)\right|\\
		& \leq \kappa_{d,\alpha} \left(\delta_{\Upsilon} + \delta_{\alpha} \left(\delta_f + \norm{f}_{L^{\infty}(D)}\right)\right) \left(M^{\frac{1}{s}-1}\widetilde{C}_1 + \widetilde{C}_2 + C_4\right)\\
		& \quad+ \kappa_{d,\alpha} \delta_{\dist} \left(\diam(D)^{\alpha}L_f + L_{\alpha}\left(\delta_f + \norm{f}_{L^{\infty}(D)}\right)\right) \left(M^{\frac{1}{s}}\widetilde{C}_1 + M(\widetilde{C}_2 + C_4)\right)\widetilde \ell,
	\end{aligned}
	\]
	where $\widetilde \ell :=\left(2\left(M^{\frac{1}{s}}\widetilde{C}_1 + M(\widetilde{C}_2 + C_3) + 1\right)^{M^{\frac{1}{s}}\widetilde{C}_1 + M(\widetilde{C}_2 + C_4)}+1\right).$

	\medskip
	\noindent
	{\bf Step 14.} Lemma \ref{lem:sacada2}, the inequality \eqref{eq:error_f}, Step 13 and Minkowski inequality ensure that
	\[
	\begin{aligned}
		&\left(\int_D \left|\E_x\left[\sum_{n=1}^{N}r_n^{\alpha} \kappa_{d,\alpha} \E^{(\mu)}\left[f\left(X_{\mathcal{I}(n-1)}+r_n\cdot\right)\right]\right] - \left(\mathcal{R}(\Psi_{2,\widetilde{\varepsilon}})\right)(x)\right|^q dx\right)^{\frac{1}{q}}\\
		&\leq \left(\int_D \left|\E_x\left[\sum_{n=1}^{N}r_n^{\alpha} \kappa_{d,\alpha} \E^{(\mu)}\left[f\left(X_{\mathcal{I}(n-1)}+r_n\cdot\right)\right]\right] - E_{M_1}(x)\right|^q dx\right)^{\frac{1}{q}} \\
		&~{} \quad + \left(\int_D \left|E_{M_1}(x)-\left(\mathcal{R}(\Psi_{2,\widetilde{\varepsilon}})\right)(x)\right|^q dx\right)^{\frac{1}{q}}\\
		&\leq \frac{\widetilde{C}_1}{M^{1-\frac{1}{s}}} + \widetilde{C}_2 \delta_f + |D|^{\frac{1}{q}}\kappa_{d,\alpha} \left(\delta_{\Upsilon} + \delta_{\alpha} \left(\delta_f + \norm{f}_{L^{\infty}(D)}\right)\right) \left(M^{\frac{1}{s}-1}\widetilde{C}_1 + \widetilde{C}_2 + C_4\right)\\
		&\quad +|D|^{\frac{1}{q}}\kappa_{d,\alpha} \delta_{\dist} \left(\diam(D)^{\alpha}L_f + L_{\alpha}\left(\delta_f + \norm{f}_{L^{\infty}(D)}\right)\right) \left(M^{\frac{1}{s}}\widetilde{C}_1 + M(\widetilde{C}_2 + C_4)\right)\widetilde{\ell}.
	\end{aligned}
	\]
	For $\widetilde{\varepsilon} \in (0,1)$, let $M \in \N$ large enough such that
	\[
	M = \left\lceil \left(\frac{5\widetilde{C}_1}{\widetilde{\varepsilon}}\right)^{\frac{s}{s-1}} \right\rceil,
	\]
	%\[
	%\frac{\widetilde{C}_1}{M^{1 - \frac 1s}} \leq \frac{\widetilde{\varepsilon}}{5},
	%\]
	and from the choice of $M$ let $\delta_{\Upsilon} \in \left(0,\frac{1}{2}\right)$ $\delta_{\dist},\delta_{\alpha} \in (0,1)$ small enough such that
	\begin{align*}
		\delta_{\Upsilon} &= \frac{\widetilde{\varepsilon}}{5|D|^{\frac 1q} \kappa_{d,\alpha}} \left(M^{\frac{1}{s}-1}\widetilde{C}_1 + \widetilde{C}_2 + C_4\right)^{-1},\\
		\delta_{\alpha}&= \frac{\widetilde{\varepsilon}}{5|D|^{\frac 1q}\kappa_{d,\alpha}} \left(1 + \norm{f}_{L^{\infty}(D)}\right)^{-1} \left(M^{\frac{1}{s} - 1}\widetilde{C}_1 + \widetilde{C}_2 + C_4\right)^{-1},\\
		\delta_{\dist} &= \frac{\widetilde{\varepsilon}}{5|D|^{\frac 1q}\kappa_{d,\alpha}\widetilde{\ell}} \left(\diam(D)^{\alpha}L_f + L_{\alpha}\left(1 + \norm{f}_{L^{\infty}(D)}\right)\right)^{-1} \left(M^{\frac 1s} \widetilde{C}_1 + M(\widetilde{C}_2 + C_4)\right)^{-1}.
	\end{align*}	
	Finally we choose $\delta_{f} \in (0,1)$ small enough such that
	\[
	\delta_f = \frac{\widetilde{\varepsilon}}{5\widetilde{C}_2}.
	\]
	Therefore
	\[
	\left(\int_D \left|\E_x\left[\sum_{n=1}^{N}r_n^{\alpha} \kappa_{d,\alpha} \E^{(\mu)}\left[f\left(X_{\mathcal{I}(n-1)}+r_n\cdot\right)\right]\right] - \left(\mathcal{R}(\Psi_{2,\widetilde{\varepsilon}})\right)(x)\right|^q dx\right)^{\frac{1}{q}} \leq \widetilde{\varepsilon}.
	\]
	We conclude that for all $\widetilde{\varepsilon} \in (0,1)$ there exists $\Psi_{2,\widetilde{\varepsilon}} \in \textbf{N}$ that approximates \eqref{eq:7.1} with accuracy $\widetilde{\varepsilon}$.

	\subsection{Proof of Proposition \ref{Prop:6p1}: quantification of DNNs}\label{Sect:7p3} In this Section we will prove that $\Psi_{2,\widetilde{\varepsilon}}$ is a ReLu DNN such that overcomes the curse of dimensionality.
	
	\medskip
	
	\noindent
	{\bf Step 15.}  We now study the DNN $\Psi_{2,\widetilde{\varepsilon}}$ with the Definitions and Lemmas of Section \ref{Sect:3}. Let
	\[
	\beta_{\dist} = \mathcal{D}(\Phi_{\dist}) \quad \hbox{and} \quad H_{\dist} = \dim(\beta_{\dist} )-2.
	\]
	Recall from Step 12 in the proof of Proposition \ref{Prop:homo} that for all $i=1,...,M$
	\[
	\mathcal{D}(\Phi_{i,1}) = d\mathfrak{n}_{H_{\dist}+2} \boxplus \widetilde{\beta}_{\dist}, \qquad \dim(\mathcal{D}(\Phi_{i,1})) = H_{\dist} +2,
	\]
	where
	\[
	\widetilde{\beta}_{\dist} = \left(\beta_{\dist,0}, ...,\beta_{\dist,H_{\dist}},d\right) \in \N^{H_{\dist}+2},
	\]
	and for all $n=2,...,\overline{N}_i$
	\[
	\mathcal{D}(\Phi_{i,n}) = \overunderset{n}{m=1}{\odot}(d\mathfrak{n}_{H_{\dist}+2}\boxplus\widetilde{\beta}_{\dist}), \quad \dim(\mathcal{D}(\Phi_{i,n})) = (H_{\dist}+1)n+1,
	\]
	with
	\[
	\vertiii{\mathcal{D}(\Phi_{i,n})} \leq 2d + \vertiii{\mathcal{D}(\Phi_{\dist})}.
	\]
	Denote
	\[
	\beta_f = \mathcal{D}(\Phi_f) \quad \hbox{and} \quad H_f = \dim(\beta_f) - 2.
	\]
	Define $\widetilde{\Phi}_{i,j,n} \in \textbf{N}$ as follows:
	\[
	\mathcal{R}(\widetilde{\Phi}_{i,j,n}) = x + v_{i,j,n}\mathcal{R}(\Phi_{\dist})(x).
	\]
	As similar as in the case of $\Phi_{i,1}$, we have that $\widetilde{\Phi}_{i,j,n}$ is a ReLu DNN such that
	\[
	\mathcal{D}(\widetilde{\Phi}_{i,j,n}) = d\mathfrak{n}_{H_{\dist}+2} \boxplus \widetilde{\beta}_{\dist}, \quad \dim(\mathcal{D}(\widetilde{\Phi}_{i,j,n})) = H_{\dist} + 2.
	\]
	Moreover
	\[
	\vertiii{\mathcal{D}(\widetilde{\Phi}_{i,j,n})} \leq 2d + \vertiii{\mathcal{D}(\Phi_{dist})}.
	\]
	Using the DNN $\widetilde{\Phi}_{i,j,n}$ we have that
	\[
	\mathcal{R}(\Phi_{i,n-1}) + v_{i,j,n}\mathcal{R}(\Phi_r^{i,n}) = \mathcal{R}(\widetilde{\Phi}_{i,j,n}) \circ \mathcal{R}(\Phi_{i,n-1}).
	\]
	Therefore, by Lemma \ref{lem:DNN_comp} it follows that
	\[
	\mathcal{R}(\Phi_f) \circ \left(\mathcal{R}(\Phi_{i,n}) + v_{i,j,n} \mathcal{R}(\Phi_{r}^{i,n})\right) \in \mathcal{R} \left(\left\{\Phi \in \textbf{N}: \mathcal{D}(\Phi)= \beta_{f} \odot \left(\overunderset{n}{m=1}{\odot}\left(d\mathfrak{n}_{H_{\dist}+2}\boxplus\widetilde{\beta}_{\dist}\right)\right)\right\}\right).
	\]
	with $(H_{\dist}+1)n+H_{f} + 2$ the total number of layers.	Note this ReLu DNN is continuous from $D$ to $\R$. Let
	\[
	\beta_{\alpha} = \mathcal{D}(\Phi_{\alpha}), \quad \hbox{and} \quad H_{\alpha} = \dim(\beta_{\alpha}) -2.
	\]
	For $n=1,...,\overline{N}_i$ we compound the previous DNN with the identity with $(H_{\dist}+1)(\sum_{i=1}^{M_1}\overline{N}_i-n) + H_{\alpha} + 1$ layers to obtain a DNN $\widehat{\Phi}_{i,j,n} \in \textbf{N}$ such that
	\[
	\mathcal{R}(\widehat{\Phi}_{i,j,n}) = \mathcal{R}(\Phi_f) \circ \mathcal{R}(\widetilde{\Phi}_{i,j,n}) \circ \mathcal{R}(\Phi_{i,n}),
	\]
	with
	\[
	\mathcal{D}(\widehat{\Phi}_{i,j,n}) =\left(  \mathfrak{n}_{(H_{\dist}+1)(\sum_{i=1}^{M_1}\overline{N}_i-n) + H_{\alpha} + 1} \right)\odot \beta_{f} \odot \left(\overunderset{n}{m=1}{\odot}\left(d\mathfrak{n}_{H_{\dist}+2}\boxplus\widetilde{\beta}_{\dist}\right)\right),
	\]
	and
	\[
	\dim(\mathcal{D}(\widehat{\Phi}_{i,j,n})) = (H_{\dist}+1)\sum_{i=1}^{M_1} \overline{N}_i  + H_{\alpha} + H_{f} +2.
	\]
	Note from the definition of DNN $\Phi_{f}^{i,n}$ that
	\[
	\mathcal{R}(\Phi_{f}^{i,n}) = \sum_{j=1}^{M_2} \mathcal{R}(\widehat{\Phi}_{i,j,n}).
	\]
	Therefore, Lemma \ref{lem:DNN_sum} implies that $\Phi_{f}^{i,n}$ is a ReLu DNN with
	%On the other hand side, Lemma \ref{lem:DNN_sum} follows that
	\[
	\mathcal{D}(\Phi_{f}^{i,n}) = \overunderset{M_2}{j=1}{\boxplus}\left( \mathfrak{n}_{(H_{\dist}+1)(\sum_{i=1}^{M_1}\overline{N}_i-n) +H_{\alpha}+ 1}\odot \beta_{f} \odot \left(\overunderset{n}{m=1}{\odot}\left(d\mathfrak{n}_{H_{\dist}+2}\boxplus\widetilde{\beta}_{\dist}\right)\right)\right),
	\]
	and
	\[
	\dim(\mathcal{D}(\Phi_f^{i,n})) = (H_{\dist}+1)\sum_{i=1}^{M_1} \overline{N}_i  +H_{\alpha}+ H_{f} + 2.
	\]
	Moreover
	\[
	\vertiii{\mathcal{D}(\Phi_{f}^{i,n})} \leq \sum_{j=1}^{M_2} \max\{\vertiii{\mathcal{D}(\Phi_f)},2d+\vertiii{\mathcal{D}(\Phi_{\dist})} \} = M_2 \max \{\vertiii{\mathcal{D}(\Phi_f)},2d+\vertiii{\mathcal{D}(\Phi_{\dist})} \}.
	\]
	On the other hand side, note that
	\[
	\mathcal{D}(\Phi_{r}^{i,n}) = \beta_{\dist} \odot \mathcal{D}(\Phi_{i,n-1}), \quad \dim(\mathcal{D}(\Phi_{r}^{i,n})) = (H_{\dist}+1)n+ 1,
	\]
	and
	\[
	\vertiii{\mathcal{D}(\Phi_{r}^{i,n})} \leq \max \{2d,\vertiii{\mathcal{D}(\Phi_{\dist})},2d+\vertiii{\mathcal{D}(\Phi_{\dist})}\}=2d + \vertiii{\mathcal{D}(\Phi_{\dist})}.
	\]
	Therefore by Lemma \ref{lem:DNN_comp}
	\[
	\mathcal{R}(\Phi_{\alpha}) \circ \mathcal{R}(\Phi_{r}^{i,n}) \in \mathcal{R}\left(\left\{\Phi \in \textbf{N}: \mathcal{D}(\Phi) = \beta_{\alpha} \odot \beta_{\dist} \odot \left(\overunderset{n}{m=1}{\odot}\left(d\mathfrak{n}_{H_{\dist}+2}\boxplus\widetilde{\beta}_{\dist}\right)\right)\right\}\right),
	\]
	with $(H_{\dist}+1)n + H_{\alpha} + 2$ number of layers. Like before, we compound the previous DNN with the identity with $(H_{\dist}+1)(\sum_{i=1}^{M_1} \overline{N}_i - n )  + H_{f} + 1$ to obtain, by Lemma \ref{lem:DNN_comp} a DNN $\widehat{\Phi}_{i,n} \in \textbf{N}$ such that
	\[
	\mathcal{R}(\widehat{\Phi}_{i,n}) = \mathcal{R}(\Phi_{\alpha}) \circ \mathcal{R}(\Phi_{r}^{i,n}),
	\]
	with
	\[
	\mathcal{D}(\widehat{\Phi}_{i,n}) = \mathfrak{n}_{(H_{\dist}+1)(\sum_{i=1}^{M_1}\overline{N}_i-n) + H_{f} + 1} \odot \beta_{\alpha} \odot \beta_{\dist} \odot \left(\overunderset{n}{m=1}{\odot}\left(d\mathfrak{n}_{H_{\dist}+2}\boxplus\widetilde{\beta}_{\dist}\right)\right),
	\]
	and
	\[
	\dim(\mathcal{D}(\widehat{\Phi}_{i,n})) = (H_{\dist}+1)\sum_{i=1}^{M_1} \overline{N}_i  +H_{\alpha}+ H_{f} + 2.
	\]
	Moreover
	\[
	\vertiii{\mathcal{D}(\widehat{\Phi}_{i,n})} \leq \max \{\vertiii{\mathcal{D}(\Phi_{\alpha})},2d + \vertiii{\mathcal{D}(\Phi_{\dist})}\}	.
	\]
	Define $H \in \N$ as
	\[
	H = (H_{\dist}+1)\sum_{i=1}^{M_1}\overline{N}_i + H_{\alpha} + H_f.
	\]
	We now realize a parallelization between the DNNs $\widehat{\Phi}_{i,n}$ and $\Phi_f^{i,n}$. By Lemma \ref{lem:DNN_para}, there exists a ReLu DNN $\overline{\Phi}_{i,n} \in \textbf{N}$ such that
	\[
	\mathcal{R}(\overline{\Phi}_{i,n}) = (\mathcal{R}(\widehat{\Phi}_{i,n}),\mathcal{R}(\Phi_{f}^{i,n})),
	\]
	with
	\[
	\mathcal{D}(\overline{\Phi}_{i,n}) = \mathcal{D}(\widehat{\Phi}_{i,n}) \boxplus \mathcal{D}(\Phi_{f}^{i,n}) + e_{H+2},
	\]
	and
	\[
	\dim(\mathcal{D}(\overline{\Phi}_{i,n})) = (H_{\dist}+1)\sum_{i=1}^{M_1} \overline{N}_i  +H_{\alpha}+ H_{f} + 2,
	\]
	where
	\[
	e_{H+2} = (0,...,0,1) \in \R^{H+2}.
	\]
	Moreover,
	\[
	\vertiii{\mathcal{D}(\overline{\Phi}_{i,n})} \leq \vertiii{\mathcal{D}(\widehat{\Phi}_{i,n})} + \vertiii{\mathcal{D}(\Phi_{f}^{i,n})},
	\]
	and thus
	\[
	\vertiii{\mathcal{D}(\overline{\Phi}_{i,n})} \leq \vertiii{\mathcal{D}(\Phi_{\alpha})} + M_2\vertiii{\mathcal{D}(\Phi_{f})} + (M_2 + 1)(2d+\vertiii{\mathcal{D}(\Phi_{\dist})}).
	\]
	Let
	\[
	\beta_{\Upsilon} = \mathcal{D}(\Upsilon), \quad \hbox{and} \quad H_{\Upsilon} = \dim(\beta_{\Upsilon}) + 2.
	\]
	Therefore, by Lemma \ref{lem:DNN_comp} it follows that
	\[
	\mathcal{D}(\Upsilon_{i,n}) = \beta_{\Upsilon} \odot \left(\mathcal{D}(\widehat{\Phi}_{i,n}) \boxplus \mathcal{D}(\Phi_{f}^{i,n}) + e_{H+2}\right),
	\]
	and
	\[
	\dim(\mathcal{D}(\Upsilon_{i,n})) =H+H_{\Upsilon} + 3.
	\]
	Moreover
	\[
	\begin{aligned}
	\vertiii{\mathcal{D}(\Upsilon_{i,n})} &\leq \max\{\vertiii{\mathcal{D}(\Upsilon)},\vertiii{\mathcal{D}(\overline{\Phi}_{i,n})}\}\\
	&\leq \vertiii{\mathcal{D}(\Upsilon)} + \vertiii{\mathcal{D}(\Phi_{\alpha})} + M_2\vertiii{\mathcal{D}(\Phi_{f})} + (M_2 + 1)(2d+\vertiii{\mathcal{D}(\Phi_{\dist})}).
	\end{aligned}
	\]
	Finally, from Lemma \ref{lem:DNN_sum} it follows that
	\[
	\mathcal{D}(\Psi_{2,\widetilde{\varepsilon}}) = \overunderset{M_1}{i=1}{\boxplus}\overunderset{\overline{N}_i}{n=1}{\boxplus} \left(\beta_{\Upsilon} \odot \left(\mathcal{D}(\widehat{\Phi}_{i,n}) \boxplus \mathcal{D}(\Phi_{f}^{i,n}) + e_{H+2}\right)\right),
	\]
	and
	\[
	\dim(\mathcal{D}(\Psi_{2,\widetilde{\varepsilon}})) = H+H_{\Upsilon} + 3.
	\]
	Moreover
	\begin{equation}\label{eq:psi2}
	\vertiii{\mathcal{D}(\Psi_{2,\widetilde{\varepsilon}})} \leq \Big(\vertiii{\mathcal{D}(\Upsilon)} + \vertiii{\mathcal{D}(\Phi_{\alpha})} + M_2\vertiii{\mathcal{D}(\Phi_{f})} + (M_2 + 1)(2d+\vertiii{\mathcal{D}(\Phi_{\dist})})\Big) \sum_{i=1}^{M_1} \overline{N}_i .
	\end{equation}
	Notice from the definition of $\widetilde{C}_1$ and $\widetilde{C}_2$ that both constants are multiple of $|D|^{\frac 1q}$. Therefore, by choice of $M_1$ and $M_2$ we have
	\begin{equation}\label{eq:Mfinalf}
	M_2 = M_1 \leq B_1 |D|^{\frac{s}{q(s-1)}}\widetilde{\varepsilon}^{-\frac{s}{s-1}}, 
	\end{equation}
	where $B_1>0$ is a generic constant. The choice of $\delta_f$ and the constant $\widetilde{C}_2$ implies that
	\begin{equation}\label{eq:deltaffinal}
	\delta_{f}^{-a} \leq B_2|D|^{\frac aq} \widetilde{\varepsilon}^{-a},
	\end{equation}
	for some constant $B_2 >0$. From the choice of $\delta_{\Upsilon}$, $\delta_{\alpha}$ and \eqref{eq:Mfinalf} it follows that
	\begin{equation}\label{eq:deltaupsfinal}
	\log(\delta_{\Upsilon}^{-1}) \leq \delta_{\Upsilon}^{-1} \leq B_3|D|^{\frac 2q} \widetilde{\varepsilon}^{-1},
	\end{equation}
	and
	\begin{equation} \label{eq:deltaalpfinal}
	\delta_{\alpha}^{-a} \leq B_4 |D|^{\frac{2a}q} \widetilde{\varepsilon}^{-a}.
	\end{equation}
	where $B_3,B_4>0$ are a generic constant, and from the choice of $\delta_{\dist}$ and properties of Logarithm function we have
	\[
	\begin{aligned}
		\log(\delta_{\dist}^{-1}) &\leq 5|D|^{\frac 1q}\kappa_{d,\alpha}\left(\diam(D)^{\alpha}L_f + L_{\alpha}\left(1 + \norm{f}_{L^{\infty}(D)}\right)\right)\left(M^{\frac 1s} \widetilde{C}_1 + M(\widetilde{C}_2 + C_4)\right)\widetilde{\varepsilon}^{-1}\\
		&\qquad+ 4\left(M^{\frac 1s} \widetilde{C}_1 + M(\widetilde{C}_2 + C_4)\right)\left(M^{\frac 1s} \widetilde{C}_1 + M(\widetilde{C}_2 + C_3)+1\right).
	\end{aligned}
	\]
	Therefore from \eqref{eq:Mfinalf}
	\begin{equation}\label{eq:deltadistfinalf}
	\lceil \log(\delta_{\dist}^{-1})\rceil^{a} \leq B_5 |D|^{\frac{2a}q\left(1+\frac{s}{s-1}\right)} \widetilde{\varepsilon}^{-a-\frac{2as}{s-1}},
	\end{equation}
	for some $B_5 > 0$ generic. Note also that
	\begin{equation}\label{eq:Nfinalf}
	\sum_{i=1}^{M_1}\overline{N}_i \leq B_6 |D|^{\frac{1}{q} \left(1 + \frac{s}{s-1}\right)}\widetilde{\varepsilon}^{-\frac{s}{s-1}},
	\end{equation}
	with $B_6 >0$. Finally from Assumptions \ref{Sup:D}, \ref{Sup:f} and inequalities \eqref{eq:psi2}, $\eqref{eq:Nfinalf}$ we got
	\[
	\begin{aligned}
	&\vertiii{\mathcal{D}(\Psi_{2,\widetilde{\varepsilon}})} \\
	&\leq B_6 |D|^{\frac{1}{q} \left(1 + \frac{s}{s-1}\right)}\widetilde{\varepsilon}^{-\frac{s}{s-1}}\Big(\log(\delta_{\Upsilon}^{-1})+\delta_{\alpha}^{-a}+M_2Bd^b\delta_f^{-a} + (M_2 + 1)(2d + Bd^b\lceil\log(\delta_{\dist}^{-1})\rceil^{a})\Big).
	\end{aligned}
	\]
	Finally, from inequalities \eqref{eq:Mfinalf}, \eqref{eq:deltaffinal}, \eqref{eq:deltaupsfinal}, \eqref{eq:deltaalpfinal} and \eqref{eq:deltadistfinalf} we conclude that there exists $\widetilde{B}>0$ such that
	\[
	\vertiii{\mathcal{D}(\Psi_{2,\widetilde{\varepsilon}})} \leq \widetilde{B} |D|^{\frac{1}{q}\left(1 +2a+\frac{2s}{s-1}(1+a)\right)}d^b\widetilde{\varepsilon}^{-a-\frac{2s}{s-1}(1+a)}.
	\]
	This completes the proof of Proposition \ref{Prop:6p1}.

\section{Proof of the Main Result}\label{Sect:8}

This final section is devoted to the proof of Theorem \ref{Main}. Gathering Propositions \ref{Prop:homo} and \ref{Prop:6p1}, Theorem \ref{Main} is finally proved.\\

\noindent
{\bf Step 1.}
Let $\alpha \in (1,2)$, $p,s \in (1,\alpha)$ such that $s< \frac{\alpha}{p}$ and $q \in \left[s,\frac{\alpha}{p}\right)$. Let Assumptions \eqref{Hg0} and \eqref{Hf0} be satisfied. Recall from Theorem \ref{teo:sol} and Lemma \ref{lem:sol2} that the solution $u$ of \eqref{eq:1.1} takes the form of \eqref{eq:solfinal}, namely
\[
u(x) = \E_x \left[g\left(X_{\mathcal{I}(N)}\right)\right] + \E_x \left[ \sum_{n=1}^{N} r_n^{\alpha} \kappa_{d,\alpha} \E^{(\mu)} \left[f\left(X_{\mathcal{I}(n-1)}+r_n \cdot\right)\right]\right], \qquad x \in D.
\]
From Propositions \ref{Prop:homo} and \ref{Prop:6p1} for all $\varepsilon, \widetilde{\varepsilon} \in (0,1)$ there exist ReLu DNNs $\Psi_{1,\varepsilon}$ and $\Psi_{2,\widetilde{\varepsilon}}$ that satisfy \eqref{eq:2.3} and \eqref{eq:prop_f}.
For the right approximation of the ReLu DNNs, $\delta_{\dist}$ and $M$ will be defined as
\[
\delta_{\dist} \leq \min\{\ell_1, \ell_2\}, \qquad M \geq \max \left\{\left\lceil\left(\frac{5C_1}{\varepsilon}\right)^{\frac{s}{s-1}}\right\rceil, \left\lceil\left(\frac{5\widetilde{C}_1}{\widetilde{\varepsilon}}\right)^{\frac{s}{s-1}}\right\rceil\right\},
\]
where
\[
\ell_1 := \frac{\varepsilon}{5|D|^{\frac{1}{q}}L_g} \left( 1 + M^{\frac{1}{s}}C_1 + M(C_2 + C_3)\right)^{-\left(M^{\frac{1}{s}}C_1 + M(C_2 + C_4)\right)},
\]
and
\[
\ell_2 := \frac{\widetilde{\varepsilon}}{5|D|^{\frac{1}{q}}\widetilde{\ell}} \left(\diam(D)^{\alpha}L_f + L_{\alpha}\left(1+\norm{f}_{L^{\infty}(D)}\right)\right)^{-1} \left(M^{\frac{1}{s}}\widetilde{C}_1 + M(\widetilde{C}_2 + C_4)\right)^{-1}.
\]
Recall that the constants in $\ell_1$ and $\ell_2$ are defined in Propositions \ref{Prop:homo} and \ref{Prop:6p1}.  Let $\epsilon \in (0,1)$ and define the ReLu DNN $\Psi_{\epsilon}$ that satisfies for all $x \in D$
\[
\left(\mathcal{R}(\Psi_{\epsilon})\right)(x) = \left(\mathcal{R}(\Psi_{1,\varepsilon})\right)(x) + \left(\mathcal{R}(\Psi_{2,\widetilde{\varepsilon}})\right)(x),
\]
where $\varepsilon = \widetilde{\varepsilon} = \frac{\epsilon}{2}$. From Minkowski inequality one has
\[
\begin{aligned}
	&\left(\int_D \left|u(x) - \left(\mathcal{R}(\Psi_{\epsilon})\right)(x)\right|^q dx\right)^{\frac{1}{q}} \\
	&\qquad\leq \left(\int_D \left|\E_x \left[g\left(X_{\mathcal{I}(N)}\right)\right] - \left(\mathcal{R}(\Psi_{1,\varepsilon})\right)(x)\right|^q dx\right)^{\frac{1}{q}}\\
	&\qquad \quad + \left(\int_D \left| \E_x \left[\sum_{n=1}^{N} r_n^{\alpha} \kappa_{d,\alpha} \E^{(\mu)}\left[f(X_{\mathcal{I}(n-1)}+r_n \cdot)\right]\right] - \left(\mathcal{R}(\Psi_{2,\widetilde{\varepsilon}})\right)\right|^q dx\right)^{\frac{1}{q}}\\
	&\qquad\leq \frac{\epsilon}{2} + \frac{\epsilon}{2} = \epsilon.
\end{aligned}
\]

\medskip

\noindent
{\bf Step 2} We now study the ReLu DNN $\Psi_{\epsilon}$. For $i=1,...,M$ Let $\overline{N}_{i,1}, \overline{N}_{i,2}$ the random variables $\overline{N}_i$ found in the Propositions \ref{Prop:homo} and \ref{Prop:6p1}, respectively. Recall that the DNN $\Psi_{1,\varepsilon}$ satisfy%, $\Psi_{2,\widetilde{\varepsilon}}$ satisfy
\[
\mathcal{D}(\Psi_{1,\varepsilon}) = \overset{M}{\underset{i=1}{\boxplus}} \left(\mathfrak{n}_{H_i+2} \odot \beta_g \odot \left(\overunderset{\overline{N}_{i,1}}{m=1}{\odot}(d\mathfrak{n}_{H_{\dist}+2}\boxplus \widetilde{\beta}_{\dist})\right)\right),
\]
where $H_{i}$ is defined as
\[
 H_{i} = (H_{\dist}+1)\left(\sum_{j=1}^{M} \overline{N}_{j,1} - \overline{N}_{i,1}\right) - 1.
\]
and
\[
\dim(\mathcal{D}(\Psi_{1,\varepsilon})) = (H_{\dist}+1)\sum_{i=1}^M \overline{N}_{i,1} + H_g + 2.
\]
Recall also that $\Psi_{2,\widetilde{\varepsilon}}$ satisfy
\[
\mathcal{D}(\Psi_{2,\widetilde{\varepsilon}}) = \overunderset{M_1}{i=1}{\boxplus}\overunderset{\overline{N}_{i,2}}{n=1}{\boxplus} \left(\beta_{\Upsilon} \odot \left(\mathcal{D}(\widehat{\Phi}_{i,n}) \boxplus \mathcal{D}(\Phi_{f}^{i,n}) + e_{H+2}\right)\right),
\]
where $e_{H+2}=(0,...,0,1) \in \R^{H+2}$ with $H$ defined as
\[
H = (H_{\dist}+1)\sum_{i=1}^{M_1}\overline{N}_{i,2} + H_{\alpha} + H_f,
\]
and
\[
\dim(\mathcal{D}(\Psi_{2,\widetilde{\varepsilon}})) = H+H_{\Upsilon} + 3.
\]
To use Lemma \ref{lem:DNN_sum}, the ReLu DNNs $\Psi_{1,\varepsilon}$ and $\Psi_{2,\widetilde{\varepsilon}}$ must have the same number of layers. We compound each DNN by a suitable ReLu DNN that represents the identity function. Define then the ReLu DNN $\overline{\Psi}_{1,\varepsilon}$ that satisfy $\mathcal{R}(\overline{\Psi}_{1,\varepsilon})=\mathcal{R}(\Psi_{1,\varepsilon})$ with%and $\overline{\Psi}_{2,\widetilde{\varepsilon}}$ that satisfy
\[
\mathcal{D}(\overline{\Psi}_{1,\varepsilon}) = \mathfrak{n}_{H+H_{\Upsilon}+3} \odot \mathcal{D}(\Psi_{1,\varepsilon}),
\]
and Define the ReLu DNN $\overline{\Psi}_{2,\widetilde{\varepsilon}}$ that satisfy $\mathcal{R}(\overline{\Psi}_{2,\widetilde{\varepsilon}}) = \mathcal{R}(\Psi_{2,\widetilde{\varepsilon}})$ with
\[
\mathcal{D}(\overline{\Psi}_{2,\widetilde{\varepsilon}}) = \mathfrak{n}_{(H_{\dist}+1)\sum_{i=1}^{M}\overline{N}_{i,1}+H_g+2} \odot \mathcal{D}(\Psi_{2,\widetilde{\varepsilon}}),
\]
Therefore we have that $\dim(\mathcal{D}(\overline{\Psi}_{1,\varepsilon}))=\dim(\mathcal{D}(\overline{\Psi}_{2,\widetilde{\varepsilon}}))$. Moreover
\[
	\dim(\mathcal{D}(\overline{\Psi}_{1,\varepsilon}))= (H_{\dist}+1)\sum_{i=1}^{M}(\overline{N}_{i,1} + \overline{N}_{i,2}) +H_{\alpha}+ H_{f}+ H_g +H_{\Upsilon} + 4.
\]
Therefore we can use Lemma \ref{lem:DNN_sum} to obtain that $\Psi_{\epsilon}$ is a ReLu DNN such that
\[
\mathcal{D}(\Psi_{\epsilon}) = \mathcal{D}(\overline{\Psi}_{1,\varepsilon}) \boxplus \mathcal{D}(\overline{\Psi}_{2,\widetilde{\varepsilon}}),
\]
and
\[
\dim(\mathcal{D}(\Psi_{\epsilon})) = (H_{\dist}+1)\sum_{i=1}^{M}(\overline{N}_{i,1} + \overline{N}_{i,2}) +H_{\alpha}+ H_{f}+ H_g +H_{\Upsilon} + 4.
\]
Moreover
\[
\vertiii{\mathcal{D}(\Psi_{\epsilon})} \leq \vertiii{\mathcal{D}(\Psi_{1,\varepsilon})} + \vertiii{\mathcal{D}(\Psi_{2,\widetilde{\varepsilon}})}.
\]
Recall from Propositions \ref{Prop:homo} and \ref{Prop:6p1} that there exists $\widetilde{B}>0$ such that
\[
\vertiii{\mathcal{D}(\Psi_{1,\varepsilon})} \leq \widetilde{B}|D|^{\frac{1}{q}\left(2a + ap+\frac{s}{s-1}(1+2a+ap)\right)} d^{b+2ap+2ap^2 + \frac{ps}{s-1}(1+2a + ap) }\varepsilon^{-a-\frac{s}{s-1}(1+2a+ap)}.
\]
and
\[
\vertiii{\mathcal{D}(\Psi_{2,\widetilde{\varepsilon}})} \leq \widetilde{B} |D|^{\frac{1}{q}\left(1 +2a+\frac{2s}{s-1}(1+a)\right)}d^b\widetilde{\varepsilon}^{-a-\frac{2s}{s-1}(1+a)}.
\]
Therefore
\[
\vertiii{\mathcal{D}(\Psi_{\epsilon})} \leq \widehat{B}|D|^{\frac{1}{q}\left(1+2a+ap+\frac{s}{s-1}(2+2a+ap)\right)}d^{b+2ap+ap^2+\frac{ps}{s-1}(1+2a + ap)} \epsilon^{-a-\frac{s}{s-1}(2+2a+ap)},
\]
where $\widehat{B}>0$ is a generic constant. Theorem \ref{Main} can be concluded choosing $\eta>0$ as the maximum between $\frac{1}{q}\left(1+2a+ap+\frac{s}{s-1}(2+2a+ap)\right)$, $b+2ap+ap^2+\frac{ps}{s-1}(1+2a + ap)$ and $\frac{s}{s-1}(2+2a+ap)$.
%{\color{red} Claim: It holds that
%\[
%\vertiii{\mathcal{D}(\Psi_{\epsilon})}\leq\widehat{C}d^{\frac{ps}{s-1}(1+2a+ap)+3ap+b+ap^2}\epsilon^{-\frac{s}{s-1}(2+2a+ap)-2a}.
%\]}

\end{document}